\documentclass[11pt]{article}
\usepackage[letterpaper, left = 1in, right = 1in, top = 1 in, bottom = 1in]{geometry}
\usepackage[utf8]{inputenc}
\usepackage{amsmath, amssymb, amsthm, bbm, bm}
\usepackage{authblk, graphicx, subcaption, xcolor}
\usepackage{enumerate}
\usepackage{subfiles}
\usepackage[sc]{mathpazo}
\usepackage[linktocpage=true]{hyperref}
\usepackage[normalem]{ulem}
\hypersetup{
 colorlinks,
 linkcolor=blue,          
 citecolor=red!50!black,       
 filecolor=black,   
 urlcolor=black, 
 pdftitle={},
 pdfauthor={},
 pdfcreator={},
 pdfsubject={},
 pdfkeywords={}
}
\usepackage[sort&compress,numbers]{natbib}
\usepackage{tikz}

\newcommand{\E}{\mathbb{E}}

\numberwithin{equation}{section}
\newtheorem{theorem}{Theorem}[section]
\newtheorem{lemma}[theorem]{Lemma}

\newtheorem{definition}[theorem]{Definition}
\newtheorem{proposition}[theorem]{Proposition}
\newtheorem{remark}{Remark}

\newcommand\blfootnote[1]{%
  \begingroup
  \renewcommand\thefootnote{}\footnote{#1}%
  \addtocounter{footnote}{-1}%
  \endgroup
}

\author{
  Zihao He$^1$,  
  Souvik Dhara$^2$,
  Debankur Mukherjee$^3$
}

\title{Waning Immunity Fails to Restore a Positive Epidemic Threshold on Power--Law Networks}
\date{}
\begin{document}
\maketitle

\begin{abstract} 
In a seminal work, Chatterjee and Durrett (2009) established that for the SIS epidemic process on random graphs with power-law degree distributions, the infection survives for an exponentially long time (in the network size) for any fixed,~positive infection rate. Equivalently, the critical infection rate separating polynomial and exponential survival regimes is zero.

In contrast, a substantial body of work in the physics literature conjectures, based primarily on numerical evidence and heuristic mean-field arguments, that introducing waning immunity (as in the SIRS process) yields a strictly positive critical infection rate on random graphs with power-law degrees; see, e.g., Pastor-Satorras et al.~(2015), Ferreira et al.~(2016), Silva et al.~(2022). In particular, below this threshold, the epidemic is expected to persist only for a polynomial duration. A recent work by Friedrich et al.~(2024) reinforces this perspective by proving polynomial survival for the SIRS process on star graphs, which is in contrast to the exponential survival in the SIS case that underpins Chatterjee and Durrett's arguments.

In this paper, we disprove this conjecture and show that the epidemic threshold is also zero for the SIRS process on the configuration model with power-law degree distribution with exponent $\tau>2$. Our proof uncovers a novel bottleneck structure for the SIRS dynamics, which we term a ``hierarchical star'' of order $2$, and show that it sustains the infection for an exponentially long time with high probability.
\end{abstract}

\blfootnote{$^1$Georgia Institute of Technology, \emph{Email:} \href{mailto:hezihao@gatech.edu}{hezihao@gatech.edu}}
\blfootnote{$^2$Georgia Institute of Technology, \emph{Email:}  \href{mailto:souvik.dhara@isye.gatech.edu}{souvik.dhara@isye.gatech.edu}}
\blfootnote{$^3$Georgia Institute of Technology, \emph{Email:} \href{mailto:debankur.mukherjee@isye.gatech.edu}{debankur.mukherjee@isye.gatech.edu}}

\section{Introduction}
Spreading processes on complex networks have long been a central topic in probability theory with classical foundations in mathematical epidemiology 
and extensive modern developments on heterogeneous contact structures 
\cite{AndersonMay1991,KeelingRohani2008,PastorSatorras2015rev,andersson2000stochastic,DraiefBook2009}.
A core objective is to understand the \emph{long-run behavior} of the dynamics on large graphs~\cite{Liggett1985IPS}, in particular, the dichotomy between rapid extinction and long-term persistence, as well as the scaling of survival times.

\paragraph{Two classes of spreading processes.}
Broadly, epidemic processes on graphs often fall into two qualitatively different classes.
In the first class, each edge can transmit infection at most once, as in Susceptible-Infected-Recovered (SIR) and Susceptible-Exposed-Infected-Recovered (SEIR)--type models.
In such cases, the `recovered' (R) state is an absorbing state, i.e., once a vertex reaches the R state, it is effectively removed from the network.
Consequently, these processes have an intrinsic {graph--exploration} structure and the epidemic typically discovers a growing infected set and then becomes extinct after exploring some portion of the network. The eventual infected population comprises either a negligible or a non-negligible fraction of vertices, depending on whether the infection rate lies below or above a certain positive threshold, respectively.
Rigorous asymptotics can be obtained via branching process-approximation and exploration process arguments on random graphs \cite{JansonLuczakWindridge2014,BallEtAlCLT2021,LashariTrapman2021}.

In the second class, vertices can be reinfected and each edge may transmit the infection repeatedly over time.
The canonical example is the Susceptible-Infected-Susceptible (SIS) process, also known as the {contact process}~\cite{Liggett1985IPS}. 
Here, each infected vertex becomes susceptible again at rate $1$ independently of all other processes, and each infected vertex transmits infection to a susceptible neighbor at rate $\lambda>0$ independently of all other processes.
Equivalently, a susceptible vertex becomes infected at total rate $\lambda$ times its number of infected neighbors.

\paragraph{Persistence of SIS on power-law random graphs.}
For SIS, a central object of interest is the \emph{extinction time} $T$, defined as the first time at which no infection remains in the network. A sequence of breakthrough results~\cite{BergerBorgsChayesSaberi2005, ChatterjeeDurrett2009, MountfordMourratValesinYao2013, BhamidiNamNguyenSly2021AOP} established that the extinction time on finite random networks can be exponentially large in the graph size $n$, even at arbitrarily small infection rates~$\lambda$. Early work by Berger et al.~\cite{BergerBorgsChayesSaberi2005} showed that, for any fixed $\lambda>0$, the infection can survive on an $n$--leaf star for exponential time, i.e.,  $T\geq \exp(\Omega(n))$ with high probability.
Due to the \emph{attractiveness} property~\cite[Definition III.2.1]{Liggett1985IPS} of the SIS process, this suggests that extinction times can be exponentially large on graphs containing high-degree vertices. This was first established rigorously for configuration models with power-law degree distributions having exponent $\tau>3$ by Chatterjee and Durrett~\cite{ChatterjeeDurrett2009}, who proved that for any fixed $\lambda>0$ and any $\delta>0$, the extinction time is at least $\exp(n^{1-\delta})$ with high probability. Subsequently, on the same random graph model, Mountford et al.~\cite{MountfordValesinYao2016} obtained a refined metastable description, and more recently, Bhamidi et al.~\cite{BhamidiNamNguyenSly2021AOP} showed that the survival time is $\exp(\Omega(n))$ for any fixed $\lambda>0$, as long as the degree distribution has a diverging moment generating function, i.e., $\E[e^{cD}] = \infty$ for all $c>0$. Across these results, the underlying mechanism is that a \emph{hub} (i.e., a star) can sustain infection for long periods, repeatedly reinfecting its neighborhood.

\paragraph*{SIRS model and recent literature.}~In this paper, we study the Susceptible--Infected--Recovered--Susceptible (SIRS) process on configuration models with power-law degree distributions. 
The SIRS model incorporates temporary immunity following infection: recovered vertices are immune to reinfection, but this immunity wanes over time at rate $\rho > 0$, after which the vertex returns to the susceptible state. The waning-immunity feature is both natural and widely studied in the physics and epidemiology literature. The model can be viewed as an interpolation between the SIR and SIS processes, obtained in the limiting cases $\rho = 0$ and $\rho = \infty$, respectively. Despite its modeling relevance, there are far fewer rigorous results for SIRS on graphs. To the best of our knowledge, the only works establishing sharp extinction-time estimates for SIRS are those of Friedrich et al.~\cite{FriedrichEtAlSIRS2024} and Lam et al.~\cite{LamNguyenYang2024}. In contrast to \cite{BergerBorgsChayesSaberi2005}, these works show that the SIRS epidemic on a star graph becomes extinct on polynomial time scales for any fixed $\lambda, \rho>0$. Moreover, \cite{FriedrichEtAlSIRS2024} identified exponential survival regimes on dense expanders, such as Erd\H{o}s--R\'enyi random graphs with $np \gg \log n$, using carefully constructed Lyapunov function arguments, an approach that breaks down for random graphs with bounded average degrees.

\paragraph{The physics conjecture.}
A large body of work in the physics literature, relying primarily on numerical simulations and analytical approximations such as pair quenched mean-field (PQMF) theory and recurrent dynamical message-passing (rDMP), conjectures the existence of a strictly positive epidemic threshold for the SIRS process on networks with power-law degree distributions with exponent $\tau > 3$~\cite{PastorSatorras2015rev,Ferreira2016collective,SilvaEtAlPRE2022}. 
Determining whether this threshold is zero or strictly positive is a foundational problem in non-equilibrium statistical mechanics. If introducing temporary immunity shifts the system from a zero threshold (as in SIS) to a finite threshold, this would indicate that local recovery dynamics can overcome the effects of extreme topological heterogeneity in scale-free networks. The intuition driving the physics consensus rests on the breakdown of the "hub activation" mechanism \cite{Ferreira2016collective}. In the SIS model, a single high-degree hub and its low-degree neighbors form a star graph capable of mutually sustaining the infection for exponentially long periods, thereby driving the global epidemic threshold to zero. In the SIRS model, however, the temporary immunity phase acts as a dynamic disruptor which causes stars to recover in polynomial duration as proved in \cite{FriedrichEtAlSIRS2024,LamNguyenYang2024}. Consequently, the physics literature argues that isolated hubs fail to maintain the endemic state, forcing the network to rely on ``collective activation'' across nodes, which theoretically necessitates a finite, macroscopic infection rate $\lambda_c>0$ to prevent rapid global extinction.

\paragraph{Our contribution.}
Our main result shows that this picture, while intuitively appealing, is inaccurate in the large-network limit.
We prove that on configuration models with i.i.d.~degrees sampled from a power--law distribution with exponent $\tau>2$, for every fixed infection rate $\lambda>0$ and immunity-loss rate $\rho>0$, the SIRS extinction time is stretched--exponential in $n$ with high probability.
In particular, the threshold for long-term survival is, in fact, zero.

The discrepancy between our findings and the prevailing physics intuition shows the inherent limitations of both mean-field theories on sparse random graphs and finite-size simulations. To overcome the failure of single-hub activation, our proof identifies a new localized survival mechanism, which we term a ``hierarchical star.'' We show that while a single isolated star graph cannot sustain the infection for exponentially long times under SIRS dynamics, a star whose neighborhood contains polynomially many stars can act as a robust, localized reservoir.
Even if individual hubs temporarily enter the immune state, this hierarchical arrangement ensures that enough interconnected hubs remain active or susceptible at any given moment to preserve the local chain of infection.

Generally, one key obstacle in analyzing SIRS is structural: unlike the contact process, SIRS does not possess the attractiveness property~\cite[Definition III.2.1]{Liggett1985IPS}. Consequently, even if one identifies a favorable subgraph, long survival on that subgraph does not automatically extend to the full graph via monotone coupling, as in the arguments of \cite{ChatterjeeDurrett2009}. To this end, our key technical contribution lies in overcoming this lack of attractiveness. We use a modified version of the Harris construction and provide a careful analysis to track infections occurring along a selected subset of edges, which allows us to bypass traditional monotone coupling arguments. This approach enables us to derive a lower bound on the survival time via infection persistence on hierarchical stars, thereby ruling out the possibility of a zero epidemic threshold in the presence of waning immunity.

\section{Model and Preliminaries}
\subsection{SIRS Process}
\label{ssec:SIRS}
Let $G=(V,E)$ be an undirected graph. If $G$ is not simple, we write $G^{\mathrm{s}}=(V,E^{\mathrm{s}})$ for its underlying simple graph, obtained by deleting all self-loops and replacing multiple edges between the same pair of vertices by a single edge. Throughout the paper, whenever $G$ is not simple, the SIRS process on $G$ is defined to be the SIRS process on $G^{\mathrm{s}}$.

The SIRS process on a simple graph $G=(V,E)$ with parameters $\lambda,\rho>0$, denoted by $\mathrm{SIRS}(G,\lambda,\rho)$, is defined as the following continuous-time Markov chain (CTMC) with state space $\{S,I,R\}^{|V|}$.
At any time $t\ge 0$, each vertex $v\in V$ is in one of three states—susceptible (S), infected (I), or recovered (R)—denoted by $\sigma_v(t)\in\{S,I,R\}$. The configuration of the network at time $t$ is represented by
\[
\boldsymbol{\sigma}(t):=(\sigma_v(t))_{v\in V}\in\{S,I,R\}^{|V|}.
\]
Each vertex evolves according to the cycle
\[
S \xrightarrow{\ \lambda k\ } I \xrightarrow{\ 1\ } R \xrightarrow{\ \rho\ } S,
\]
with the indicated transition rates, where $k$ denotes the number of infected neighbors of the vertex at that time.
To be precise, for any two configurations $X,Y\in \{S,I,R\}^{|V|}$, the time-homogeneous transition rate $q(X,Y)$ from state $X$ to state $Y$ is nonzero only when $Y$ differs from $X$ at exactly one vertex $i\in V$. The rates are given by: 
\begin{align*}
q(X,Y)=\left\{ \begin{aligned}
& \lambda \sum_{j' \in V: \{i,j'\} \in E} \mathbf{1}\{ X_{j'}=I \}, & &\text{if } X_i=S,\, Y_i=I,\, \text{and } X_j=Y_j\ \forall j\neq i,\\
& 1, & &\text{if } X_i=I,\, Y_i=R,\, \text{and } X_j=Y_j\ \forall j\neq i,\\
& \rho, & &\text{if } X_i=R,\, Y_i=S,\, \text{and } X_j=Y_j\ \forall j\neq i,\\
& 0, & &\text{otherwise.}
\end{aligned}\right.
\end{align*}
The SIR and SIS processes can be recovered as limiting cases of the SIRS dynamics by taking $\rho = 0$ and $\rho = \infty$, respectively, so that recovered vertices remain immune indefinitely or immediately return to the susceptible state. SIS can be viewed as a spin system and its mathematical theory is well developed, including fundamental tools such as attractiveness and duality \cite{Liggett1985IPS,PastorSatorras2015rev}.

Now we focus on the extinction time of the SIRS process. Let $T:=\inf\{t>0:\ \sigma_v(t)\neq I \text{ for all } v\in V\}$ be the first time when no vertex remains infected. We say that the infection dies out at time $T$, since after time $T$, all vertices are in states $S$ or $R$, and the system remains infection-free thereafter. Large $T$ indicates long quasi-stationary activity before absorption, while small $T$ signals rapid die-out. The extinction time for SIS is defined the same as SIRS.
Upper and lower bounds on $T$, its scaling (polynomial vs.\ exponential), and its dependence on graph structure have been extensively studied for contact-type dynamics; see, e.g., \cite{Liggett1985IPS,ChatterjeeDurrett2009,MountfordMourratValesinYao2013, BergerBorgsChayesSaberi2005,BhamidiNamNguyenSly2019,HuangDurrett2020}.

We also introduce the \emph{threshold-SIRS process}, which is a weaker spreading mechanism compared to the standard SIRS dynamics. In this model, a susceptible vertex becomes infected at a constant rate, rather than at a rate proportional to the number of infected neighbors, provided it has at least one infected neighbor. That is, the infection rate is modified from $\lambda k$ to $\lambda \mathbf{1}\{k \ge 1\}$. Formally, the transition rates are given by
$q(X,Y) = \lambda \mathbf{1}\{ \sum_{j \in V : \{i,j\} \in E} \mathbf{1}\{X_j = I\} \ge 1 \},$
whenever $X_i = S$, $Y_i = I$, and $X_j = Y_j$ for all $j \ne i$.
As we discuss in Remark~\ref{rem:threshold}, the conclusion of our main results also hold for the threshold-SIRS process.

\subsection{Modified Harris construction}
\label{ssec:harris}
On a common probability space, we realize the above $\mathrm{SIRS}(G,\lambda,\rho)$ using a graphical representation. For the SIS process (contact process), the Harris graphical construction is standard and widely used
\cite{Harris1972,Liggett1985IPS,DurrettTenLecturesPS}.
For the SIRS dynamics considered here, we use a straightforward modification of the SIS construction:
in addition to the usual infection transmission and recovery marks, we equip each vertex with an independent Poisson point process
of rate $\rho$ producing $R\to S$ marks. We spell out the resulting construction below.

Let
\[
\mathcal{H}\ :=\ \Bigl(\, \bigl(N_v^{\mathrm{rec}},N_v^{\mathrm{sus}}\bigr)_{v\in V},\ 
\bigl(N_{u\to v}^{\mathrm{inf}}\bigr)_{(u,v):\{u,v\}\in E}\Bigr),
\]
where the components of $\mathcal{H}$ are mutually independent Poisson point processes \cite{Liggett1985IPS,DurrettTenLecturesPS}:
$N_v^{\mathrm{rec}}$ has rate $1$, $N_v^{\mathrm{sus}}$ has rate $\rho$, and for each ordered pair
$(u,v)$ with $\{u,v\}\in E$, $N_{u\to v}^{\mathrm{inf}}$ has rate $\lambda$.
For a point process $N$, let $N(A)$ denote the number of points in a Borel set $A\subseteq [0,\infty)$.

For $t\ge 0$, let $\mathcal{H}|_{[0,t]}$ denote the restriction of all these point processes to $[0,t]$, and set
\[
\mathcal{F}_t\ :=\ \sigma\!\left(\mathcal{H}|_{[0,t]}\right),\qquad 
\mathcal{F}\ :=\ \sigma(\mathcal{H}),
\]
with $\mathbb{F}=(\mathcal{F}_t)_{t\ge 0}$ the associated filtration, where $\sigma(\cdot)$ denotes the smallest $\sigma$-algebra generated by the indicated random objects.
We view each Poisson point process $N$ as a random locally finite subset of $(0,\infty)$.
Thus, $t\in N$ means that $N$ has a mark at time $t$ (equivalently, $N(\{t\})=1$).

Fix an initial state $\boldsymbol{\sigma}(0)\in\{S,I,R\}^{V}$. Given $\mathcal{H}$, we construct $\boldsymbol{\sigma}(t)$ pathwise as a c\`adl\`ag process as follows.
At any mark time $t>0$:
\begin{enumerate}[(i)]
    \item if $t\in N_v^{\mathrm{rec}}$ and $\sigma_v(t^-)=I$, then $\sigma_v(t)=R$;
    \item if $t\in N_v^{\mathrm{sus}}$ and $\sigma_v(t^-)=R$, then $\sigma_v(t)=S$;
    \item if $t\in N_{u\to v}^{\mathrm{inf}}$ and $\sigma_u(t^-)=I$ and $\sigma_v(t^-)=S$, then $\sigma_v(t)=I$.
\end{enumerate}
All other marks have no effect. By independence of the Poisson marks, simultaneous events occur with probability zero; hence this defines
$\boldsymbol{\sigma}(t)$ uniquely, and $\boldsymbol{\sigma}$ is adapted to $\mathbb{F}$.
We refer to this graphical representation for the SIRS dynamics as the \emph{modified Harris construction}.\\

\subsection{Power-law Random Graphs and Configuration Model Construction}
\label{ssec:power-law}
Having specified the stochastic dynamics, we now fix the class of underlying random graphs on which the process evolves. Throughout, we work with sparse random graphs generated by the configuration model, a standard framework in random graph theory and a common choice in the study of epidemic processes on networks \cite{frieze2015random,janson2009simple,vanDerHofstad2024rgcn2,BallEtAlCLT2021,LashariTrapman2021}. Power-law random graphs are particularly natural in the present context, since their heavy-tailed degree distributions capture highly heterogeneous contact patterns with rare hubs, a structural feature known to play an important role in sustaining long-lived activity in reinfection processes such as SIS and SIRS \cite{PastorSatorras2015rev}.

Fix $\tau>2$ and assume that $D$ has a power-law distribution, namely
$\mathbb{P}(D=k)=C_\tau\,k^{-\tau}\mathbf{1}\{k\ge3\},$
where
$C_\tau=(\sum_{j=3}^{\infty}j^{-\tau})^{-1}.$ 
We impose the minimum-degree condition $D\ge 3$ in order to ensure connectivity and to facilitate logarithmic diameter estimates; see \cite[Theorem 4.24, Theorem~7.19]{vanDerHofstad2024rgcn2}.

For each fixed $n$, we construct the configuration model as follows: Independently sample degrees $d_1,\dots,d_n$ i.i.d.~from the distribution of $D$. Conditional on the event that $\sum_{i=1}^n d_i$ is even, attach $d_i$ half-edges to vertex $i$, and then pair all half-edges uniformly at random to form edges. The resulting object is, in general, a multigraph, since self-loops and multiple edges may occur. We denote it by $G(n,\tau)$ and refer to it as the \emph{power-law random graph} for simplicity.

For $\tau \in (2,3]$, the probability that $G(n,\tau)$ is simple tends to $0$ as $n \to \infty$; see \cite[Theorem~1.1]{janson2009simple}. In contrast, for $\tau > 3$, it is known that there exists a constant $c > 0$ such that, for all sufficiently large $n$, $\mathbb{P}\bigl(G(n,\tau)\text{ is simple}\bigr) \ge c$;
see \cite[Theorem~1.1]{janson2009simple}. Consequently, for $\tau > 3$, concentration results and high-probability bounds for the configuration model multigraph continue to hold after conditioning on simplicity. Accordingly, we will henceforth work with the configuration model $G(n,\tau)$ without conditioning on simplicity.

\subsection{Main Results} 
We establish the zero-threshold behavior for the SIRS process with a stretched-exponential survival time.
\begin{theorem}\label{thm_main1}
Fix $\lambda, \rho>0, \delta\in (0,1)$ and $\tau>2$. Consider the process $\mathrm{SIRS}(G(n,\tau),\lambda,\rho)$ initialized such that all vertices are infected at time $t=0$. Then, as $n\to\infty$,
\[
\mathbb{P}\big(T\ge \exp(n^{1-\delta})\big)\to 1,
\]
where $\mathbb{P}$ denotes the joint probability law of the random graph $G(n,\tau)$ and the SIRS Markov chain.
\end{theorem}

\noindent

We next consider partially infected initial conditions. For each $n$, 

let $\mu_n$ be an arbitrary probability distribution on the state space $\{S,I,R\}^{|V|}$. We consider an initial configuration generated as follows: sample a base configuration from $\mu_n$, select a vertex $v^* \in V$ uniformly at random, and deterministically set the state of $v^*$ to $I$ (leaving the states of all other vertices unchanged). Let $\mathbb{P}^*_{\mu_n}$ denote the joint probability measure incorporating the realization of the power-law random graph $G(n,\tau)$, the sampling from $\mu_n$, the choice of $v^*$, and the trajectory of the SIRS Markov chain.

\begin{theorem}\label{thm_main2}
Fix $\lambda, \rho>0$, $\delta \in (0,1)$, and $\tau>2$. Under the initialization framework described above, there exists a constant $p>0$ (depending only on $\lambda,\rho,\tau,$ and $\delta$) such that for all sufficiently large $n$,
\[
\inf_{\mu_n} \mathbb{P}^*_{\mu_n}\big(T\ge \exp(n^{1-\delta})\big)\;\ge\; p.
\]
\end{theorem}

\noindent
Theorem~\ref{thm_main2} asserts that there exists a constant $p>0$ (independent of $n$ and of the choice of $\mu_n$) such that, once a single uniformly chosen vertex is infected at time $0$, the infection survives for time at least $\exp(n^{1-\delta})$ with probability at least $p$, regardless of how the remaining $n-1$ vertices are initialized. This extends Theorem~\ref{thm_main1} beyond the fully infected start. An analogous statement is known for the SIS process \cite{ChatterjeeDurrett2009}.

\begin{remark}\normalfont 
    Note that Theorem~\ref{thm_main1} is stated here for the power-law random graph. However, the proof only requires two deterministic structural properties of the underlying graph: first, that the graph contains at least $\lceil n^{1-\delta}\rceil$ vertex-disjoint $(n^\varepsilon,n^\varepsilon)$-\emph{hierarchical-stars} in the sense of Definition~\ref{def:hierarchical-star}; and second, that the graph diameter is bounded by $c\log n$ for some constants $c,\varepsilon>0$. Therefore, the same conclusion remains valid for any graph satisfying these two properties. A precise deterministic-graph formulation is given in Proposition~\ref{prop:deterministic_many_hierarchicalstars}.
\end{remark}

\begin{remark}\normalfont
For $c>0$, let
\(
D_n(c):=\{\mathrm{diam}(G)\le c\log n\}
\).
Then the argument in the proof of Theorem~\ref{thm_main1} yields the following more explicit bound:
\[
\mathbb{P}\big(T\ge \exp(n^{1-\delta})\big)\ge \mathbb{P}(D_n(c))- c_1n^{-\min\{\tau-2,1\}}(\log n)^{\mathbf{1}\{\tau=3\}},
\]
for some constant $c_1>0$. This follows directly from the proof of Theorem~\ref{thm_main1}. Moreover, \cite[Theorem~7.19]{vanDerHofstad2024rgcn2} proves that there exists a constant $c>0$ such that $\mathbb{P}(D_n(c))\to 1$. For $\tau>3$, the proof of Lemma~1.2 in ~\cite{ChatterjeeDurrett2009} also implies that $\mathbb{P}(D_n(c))\ge 1-o(n^{-1})$, although this bound is not stated there explicitly.
\end{remark}

\begin{remark} \normalfont
The assumption in Theorem~\ref{thm_main1} that all vertices are infected at time $0$ can be substantially weakened. Inspecting the proof shows that it suffices to assume that, for any fixed $\varepsilon > 0$, the \emph{initial infection consists of a single vertex of degree at least $n^{\varepsilon}$}. Under this condition, $\mathbb{P}\bigl(T > \exp(n^{1-\delta})\bigr) \to 1$ as $n \to \infty$.
\end{remark}

\begin{remark}\label{rem:threshold} \normalfont
The conclusions of Theorems~\ref{thm_main1} and~\ref{thm_main2} also hold for the threshold-SIRS process. We do not give a separate proof, since the same arguments apply with only minor modifications: Proposition~\ref{prop:star_independent} and Proposition~\ref{hierarchicalstar_thm2} remain valid in this setting. The required modifications are discussed at the end of Sections~\ref{subsection_hierarchicalstar_2} and~\ref{subsection_hierarchicalstar_1}.
\end{remark}

\subsection{Proof Outline}
We now describe the main steps of the proof and indicate where each ingredient is established in the paper.

\paragraph{STEP 1: Local persistence on a star.} 
Our first ingredient is a local polynomial-time persistence statement for SIRS. 
Due to the loss of attractiveness, we cannot use the results from \cite{frieze2015random,LamNguyenYang2024}. The analysis also requires specific non-asymptotic bounds for the modified Harris construction. 
A vertex is called \emph{$m$-lit} if either $v$ is infected, or at least $\lceil m\rceil$ of its neighbors are infected.
Roughly speaking, if a vertex $v$ has degree at least $k$ and is initially $k^{1/24}$-lit, then with high probability it remains $k^{1/24}$-lit for time $3k^{\min\{\rho,\lambda,1\}/200}$; this is provided by Proposition~\ref{prop:star_independent}. The key feature is not only the persistence estimate itself, but also its \emph{locality}: the corresponding ``good event'' is measurable with respect to the modified Harris construction restricted to the edges incident to $v$ and its neighbors, as formalized in Definition~\ref{def:local_graphical_data}. This locality is one of the main substitutes for attractiveness. We believe these results are of independent interest as they can be used for graphs that are not locally tree-like.

\paragraph{STEP 2: A hierarchical-star mechanism.}
Because a single star only yields polynomial survival~\cite{FriedrichEtAlSIRS2024}, the second step introduces the core structural innovation of our proof: a larger localized structure we term the \emph{hierarchical-star} given in Definition~\ref{def:hierarchical-star}. A hierarchical-star is a two-layer tree structure in which the central hub has $\lceil k\rceil$ first-layer neighbors, and each first-layer neighbor in turn serves as the center of a star with $\lceil k\rceil$ distinct second-layer leaves. This hierarchical layering allows the infectious activity to be repeatedly transferred among many branches, sustaining the local outbreak for much longer times.

For a hierarchical-star, we let $(W_r)_{r\in\mathbb{N}}$ denote the number of first-layer vertices that are $k^{1/24}$-lit at time~$3rk^{\min\{\rho,\lambda,1\}/200}$. Using the local persistence result from STEP~1, we construct a locally measurable comparison process $\tilde W_r\le W_r$ satisfying $\mathbb E[\exp(-\tilde W_{r+1})\mid\mathcal F_r]\le \exp(-\tilde W_r)$ while $\tilde W_r\in [\frac{1}{4}k^{\min\{\rho,\lambda,1\}/800}+1, \lceil k^{\min\{\rho,\lambda,1\}/800}\rceil]$. The conclusion is that, with high probability, a hierarchical-star of size $k$ maintains a positive number of first-layer vertices that remain $k^{1/24}$-lit for at least $\exp(k^{\min\{\rho,\lambda,1\}/1600})$ time; this is summarized in Proposition~\ref{hierarchicalstar_thm2}. As in STEP~1, the corresponding good event is again \emph{local}: it is measurable with respect to the modified Harris construction restricted to the hierarchical-star itself, as formalized in Definition~\ref{def:hierarchical-star_sigmaalgebra}. In particular, for disjoint hierarchical-stars, the associated good events are independent.

\paragraph{STEP 3: Many hierarchical-stars inside the power-law graph.}
The next step is to show that the power-law random graph $G(n,\tau)$ typically contains many disjoint hierarchical-stars. We first establish the required degree-structure estimates through Lemma~\ref{lemma_degree_event_D}. These are then used to construct one-layer and two-layer local structures in Lemma~\ref{lemma_hierarchicalstarv11layer}, Lemma~\ref{lemma_hierarchicalstarvi1layer}, and Lemma~\ref{lemma_hierarchicalstar_second_layer}. Combining these ingredients yields Proposition~\ref{prop:existsmanyhierarchicalstar}, which states that, with high probability, the graph contains at least $\lceil n^{1-\delta}\rceil$ vertex-disjoint hierarchical-stars of size $n^{\varepsilon}$, where $\varepsilon$ is a constant depending on $\delta,\tau$.

A second geometric input is that the graph diameter is at most logarithmic (Lemma~\ref{lemma_diameter_log_n}). Together with Lemma~\ref{lemma_edge_transmission_one}, this implies that while one hierarchical-star stays active, infection has a polynomially small but still sufficient probability to travel to any other hierarchical-star within logarithmic time. The resulting spreading mechanism between hierarchical-stars is formalized in Lemma~\ref{lemma_infectionspreadthroughhierarchicalstars}.

Conditioning on this graph structure, we next lift the hierarchical-star mechanism from STEP~2 to the level of the full graph. We let $(W_r)_{r\in\mathbb{N}}$ denote the number of hierarchical-stars that satisfy the properties in STEP~2 at time $r\exp(n^{\varepsilon\min\{\rho,\lambda,1\}/1600})$. Using Proposition~\ref{hierarchicalstar_thm2} from STEP~2 together with Lemma~\ref{lemma_infectionspreadthroughhierarchicalstars} in STEP~3, we construct a lower-bounding comparison process $\tilde W_r\le W_r$ such that  again $\mathbb E[\exp(-\tilde W_{r+1})\mid\mathcal F_r]\le \exp(-\tilde W_r)$ while $\tilde{W}_r\in [1,\lceil n^{1-\delta}\rceil)$. This implies that the number of such hierarchical-stars has a strong upward drift and is therefore very unlikely to hit zero before returning to a larger macroscopic level. In this way, we obtain long survival on the whole power-law graph. This multiscale argument completes the proof of Theorem~\ref{thm_main1}.

\paragraph{STEP 4: From one infected vertex to the metastable regime.}
Finally, we consider more general initial conditions and prove Theorem~\ref{thm_main2}. The goal is to show that one uniformly chosen infected vertex already has a positive probability of triggering the large-scale mechanism from STEP~3. The proof proceeds by a multi-scale bootstrap argument. There exists a fixed degree threshold $M=M(\lambda,\rho,\tau)$ such that, once a vertex of degree at least $M$ is infected, the local results from STEP~2 and STEP~3 can be applied at this fixed scale: with positive probability, the vertex becomes the center of a local hierarchical-star and remains active long enough to infect a vertex of larger degree. Repeating this argument, the infection climbs through degree scales $M,M^2,M^3,\dots$ until it reaches a vertex of degree at least $n^\varepsilon$, where $\varepsilon$ is the exponent from Proposition~\ref{prop:existsmanyhierarchicalstar}.

Once such a vertex is infected, the process has entered the same high-degree regime used in STEP~3. At that point Theorem~\ref{thm_main1}, together with Lemma~\ref{lemma_infectionspreadthroughhierarchicalstars}, implies long survival with high probability. Since a uniformly chosen vertex in a power-law graph has degree at least $M$ with probability bounded away from $0$, this yields a uniform positive lower bound on the probability of long survival, and hence proves Theorem~\ref{thm_main2}.

\section{Analysis for Star Graphs}\label{subsection_hierarchicalstar_2}
\begin{definition}[{\normalfont Lit vertices}]\label{def:lit}
Let $G=(V,E)$ be a graph, let $v\in V$, and let $U\subseteq N(v)$ be a set of
neighbors of $v$. For any $m>0$, we say that $v$ is \emph{$m$-lit among $U$} at time $t$ if 
\[
\sigma_v(t)=I
\qquad\text{or}\qquad
\sum_{u\in U}\mathbf{1}\{\sigma_u(t)=I\}\ \ge\ m.
\]
\end{definition}
For the next definition, recall the Poisson processes introduced in the modified Harris construction from Section~\ref{ssec:harris}.
\begin{definition}[{\normalfont Local graphical data}]\label{def:local_graphical_data}
Let $G=(V,E)$ be a graph, let $v\in V$, and let $U\subseteq N(v)$ be a set of
neighbors of $v$. We define the \emph{local graphical data}
associated with $(v,U)$ to be the $\sigma$-field
\[
\mathcal{G}(v,U)
:=\sigma\!\Big(
N^{\mathrm{rec}}_{v},\,N^{\mathrm{sus}}_{v},\,
(N^{\mathrm{inf}}_{v\to u})_{u\in U},\,(N^{\mathrm{inf}}_{u\to v})_{u\in U},\,
(N^{\mathrm{rec}}_{u},\,N^{\mathrm{sus}}_{u})_{u\in U}
\Big),
\]
where $\sigma(\mathcal{X})$ denotes the smallest sigma-algebra generated by the collection of $\mathcal{X}$.
\end{definition}

\medskip
The goal of this section is to prove Proposition~\ref{prop:star_independent}, which constructs a \emph{local, high-probability}
event $A^{(v)}\in\mathcal{G}(v,U)$ ensuring that the center $v$ stays lit over a time interval whose length is polynomial in the parameter $|U|$.

\begin{proposition}\label{prop:star_independent} 
Let $G=(V,E)$ be a graph, let $v\in V$, and let $U\subseteq N(v)$ be a set of neighbors of~$v$ with $|U|\ge k$. Let $\boldsymbol{\sigma}(0)$ be an arbitrary deterministic initial configuration on $V$, subject only to the condition that $v$ is $k^{1/24}$-lit among $U$ at time 0. Then there exists an event $A^{(v)}\in \mathcal{G}(v, U)$, such that for all sufficiently large $k$, the following holds:
\begin{enumerate}
  \item[\normalfont (i)] On ${A}^{(v)}$, the vertex $v$ remains $k^{1/96}$-lit among $U$ throughout the time interval $\big[0,\tfrac{1}{48}\log k\big]$;
  \item[\normalfont (ii)] On ${A}^{(v)}$, the vertex $v$ remains $k^{1/24}$-lit among $U$ throughout the time interval $[\frac{1}{48}\log k ,3k^{\min\{\rho,\lambda,1\}/200}]$;
    \item[\normalfont (iii)] On ${A}^{(v)}$, there exist random times $T_i$ for $i=1,\dots,\lfloor k^{\min\{\rho,\lambda,1\}/400}\rfloor+1$ such that each $T_i$ is measurable with respect to $\mathcal{G}(v,U)$, $T_{i+1}-T_i\ge 3$, $T_{i}\le k^{\min\{\rho,\lambda,1\}/200}$, and $v$ is infected at time $T_i$ for every $i=1,\dots, \lfloor k^{\min\{\rho,\lambda,1\}/400}\rfloor$;
  \item[\normalfont (iv)] $\mathbb{P}({A}^{(v)}) \ge 1 - k^{-\min\{\rho,\lambda,1\}/200}$.
\end{enumerate}

\end{proposition}

We split the proof into three steps, corresponding to the following three subsections.

In Subsection~\ref{ssec:PP}, we work directly with the Poisson clocks in the modified Harris construction and construct local measurable events
$A_1^{(v)},A_2^{(v)}$ (Lemma~\ref{lemma_star_A1v} and Lemma~\ref{lemma_star_A2v}), each of which occurs with high probability. Here and below, ``local measurable'' means measurable with respect to the local graphical $\sigma$-field $\mathcal{G}(v,U)$. We then prove that when $v$ is initially infected, the event $A_1^{(v)}\cap A_2^{(v)}$ ensures that throughout the partial transition $I\to R\to S$ at $v$, $v$ remains $k^{1/24}$-lit
(Lemma~\ref{lemma_star_lit_from_A1A2}).

In Subsection~\ref{ssec:persist-inf}, building on $A_1^{(v)}\cap A_2^{(v)}$, we construct a local measurable event $A_3^{(v)}$ which again occurs with high probability, and which provides a local measurable time $T_1$ at which $v$ is infected at $T_1$ while remaining $k^{1/24}$-lit among $U$ up to that time. Moreover, the intersection $A_1^{(v)}\cap A_2^{(v)}\cap A^{(v)}_3$ has probability at least $1-2k^{-2c}$,. On this event, $v$ remains $k^{1/24}$-lit among $U$ throughout one full $I\to R\to S\to I$ cycle (Lemma~\ref{lemma_star_Asigmav}) in time interval $[0, T_1]$. Iterating this measurable reinfection mechanism for $k^{c}$ successive rounds, we obtain that $v$ remains $k^{1/24}$-lit among $U$ for a time of order $k^{c}$ with high probability (Lemma~\ref{lemma_star_Av}), which establishes item~(ii) of Proposition~\ref{prop:star_independent} in the case $\sigma_v(0)=I$.

Subsection~\ref{ssec:persist-lit} addresses item~(i) of Proposition~\ref{prop:star_independent}. In this subsection we assume that $v$ is $k^{1/24}$-lit at time $0$
without requiring that $\sigma_v(0)= I$. We construct an event $A_5^{(v)}\in\mathcal{G}(v,U)$, with high probability,
such that on $A_5^{(v)}$ the vertex $v$ becomes infected at some local measurable time $T_0$ , while remaining $k^{1/96}$-lit up to time $T_0$
(Lemma~\ref{lemma_star_startwithvnotinfected}). Combining this with the results of Subsection~\ref{ssec:persist-inf} reduces the general initial condition to the infected case.

Putting these three parts together, all the auxiliary events are local measurable, and their intersection occurs
with probability $1-k^{-c'}$ for some constant $c'>0$, proving Proposition~\ref{prop:star_independent}.

\subsection{Auxiliary Results about Poisson Processes}
\label{ssec:PP}
In this subsection, we show that, starting from $v$ infected, the vertex $v$ becomes susceptible at some time, while remaining $k^{1/24}$-lit among $U$ up to that time. Lemma~\ref{lemma_star_A1v} is the first step: on the event $A_1^{(v)}$, if $v$ is infected at time $0$, then $v$ becomes susceptible at some time before $\frac{1}{6}\log k$.

\begin{lemma}\label{lemma_star_A1v}
Let $G=(V,E)$ be a graph and let $v\in V$. Define the event
\[
A_1^{(v)}
:=\Big\{
N_v^{\mathrm{rec}}([0,k^{-1/3}])=0,\ 
N_v^{\mathrm{rec}}\big((k^{-1/3},\,\tfrac{1}{12}\log k]\big)\ge 1,\ 
N_v^{\mathrm{sus}}\big((\tfrac{1}{12}\log k,\,\tfrac{1}{6}\log k]\big)\ge 1
\Big\},
\]
with $N_v^{\text{rec}}, N_v^{\text{sus}}$ the Poisson processes defined in the modified Harris construction. Then $A_1^{(v)}$ is \\ $\mathcal{G}(v, \{u_1,\dots, u_k\})$ measurable, with $\mathcal{G}(v, \{u_1,\dots, u_k\})$ defined in Definition~\ref{def:local_graphical_data} for all sufficiently large $k$,
\[
\mathbb{P}\!\left(A_1^{(v)}\right)\ \ge\ 1-k^{-\min\{1,\rho\}/18}.
\]
\end{lemma}

\begin{proof}
Under the modified Harris construction, $N_v^{\mathrm{rec}}$ and $N_v^{\mathrm{sus}}$ are independent Poisson
point processes with respective rates $1$ and $\rho$. Moreover, for a Poisson point process, the counts on
disjoint time intervals are independent. Consequently,
\begin{align*}
\mathbb{P}\big(A_1^{(v)}\big)
&=\mathbb{P}\big(N_v^{\mathrm{rec}}([0,k^{-1/3}])=0\big)\,
  \mathbb{P}\big(N_v^{\mathrm{rec}}((k^{-1/3},\tfrac{1}{12}\log k])\ge 1\big)\,
  \mathbb{P}\big(N_v^{\mathrm{sus}}((\tfrac{1}{12}\log k,\tfrac{1}{6}\log k])\ge 1\big)\\
&=\exp(-k^{-1/3})\,
  \Big(1-\exp\big(-(\tfrac{1}{12}\log k-k^{-1/3})\big)\Big)\,
  \Big(1-\exp\big(-\rho(\tfrac{1}{6}-\tfrac{1}{12})\log k\big)\Big)\\
&=\exp(-k^{-1/3})\,
  \Big(1-\exp\big(-(\tfrac{1}{12}\log k-k^{-1/3})\big)\Big)\,
  \big(1-k^{-\rho/12}\big).
\end{align*}
Using $e^{-x}\ge 1-x$ for $x\ge 0$, we obtain $\exp(-k^{-1/3})\ge 1-k^{-1/3}$. For all sufficiently large $k$, we have $k^{-1/3}\le (\log k)/48$, and hence
\[
\tfrac{1}{12}\log k-k^{-1/3}\ \ge\ \Big(\tfrac{1}{12}-\tfrac{1}{48}\Big)\log k=\tfrac{1}{16}\log k,
\]
which yields
\[
1-\exp\big(-(\tfrac{1}{12}\log k-k^{-1/3})\big)\ \ge\ 1-\exp\!\big(-\tfrac{1}{16}\log k\big)\ =\ 1-k^{-1/16}.
\]
Therefore,
\[
\mathbb{P}\big(A_1^{(v)}\big)
\ge (1-k^{-1/3})(1-k^{-1/16})(1-k^{-\rho/12})
\ge 1-\big(k^{-1/3}+k^{-1/16}+k^{-\rho/12}\big),
\]
where we used $(1-a)(1-b)(1-c)\ge 1-a-b-c$ for $a,b,c\in [0,1]$.
Finally, each term on the right-hand side is at most $k^{-\min\{1,\rho\}/16}$ for $k\ge 1$, and thus
\[
k^{-1/3}+k^{-1/16}+k^{-\rho/12}
\le 3k^{-\min\{1,\rho\}/16}
\le k^{-\min\{1,\rho\}/18}
\]
for all sufficiently large $k$. This proves the claim.

\end{proof}

We recall a standard large deviation estimate (e.g., see~\cite[Lemma~5.1]{ChatterjeeDurrett2009}), which will be used repeatedly throughout the paper.
\begin{lemma}\label{lemma_largedeviation}
Let $X_1,X_2,\dots$ be i.i.d.\ nonnegative random variables with mean $\mu$. If $c_1<\mu$, then there exists a constant $c_2>0$ such that, for all $n\ge 1$,
\[
\mathbb{P}(X_1+\cdots+X_n\le c_1 n)\le e^{-c_2 n}.
\]
\end{lemma}

Lemma~\ref{lemma_star_A2v} constructs a good set $U_v$. On the event $A_2^{(v)}$, this set satisfies $|U_v|\ge k^{1/24}$. In Lemma~\ref{lemma_star_A2v}, we only verify through a direct calculation that this event occurs with probability tending to $1$, without yet explaining the dynamical role of $U_v$. Its role will be made clear in Lemma~\ref{lemma_star_lit_from_A1A2}: starting from $v$ infected, on the event $A_1^{(v)}\cap A_2^{(v)}$, every vertex in $U_v$ remains infected throughout the time interval $[k^{-1/3},(\log k)/4]$.

\begin{lemma}\label{lemma_star_A2v} 
Let $G_\star$ denote the star graph on vertex set $\{v,u_1,\dots, u_k\}$ with center $v$. 
Define the random set $U_v$ by
\begin{align*}
U_v
:=\Big\{ u\in \{u_1,\dots, u_k\}:
N^{\mathrm{sus}}_{u}\big((0,\tfrac12 k^{-1/3})\big)\ge 1,\ 
N^{\mathrm{inf}}_{v\to u}\big([\tfrac12 k^{-1/3},\,k^{-1/3}]\big)\ge 1, 
N^{\mathrm{rec}}_{u}\big([0,(\log k)/4)\big)=0
\Big\},
\end{align*}
and define the event
\begin{equation}\label{eq:A2v}
A_2^{(v)}:=\{|U_v|\ge k^{1/24}\},
\end{equation}
where $N^{\mathrm{sus}}_{u}$, $N^{\mathrm{inf}}_{v\to u}$, and $N^{\mathrm{rec}}_{u}$ are the Poisson clocks in the modified Harris construction.
Then, $U_v, A_2^{(v)}$ are $\mathcal{G}(v, \{u_1,\dots, u_k\})$ measurable, with $\mathcal{G}(v, \{u_1,\dots, u_k\})$ defined in Definition~\ref{def:local_graphical_data}, and for all sufficiently large $k$,
\begin{equation}
\mathbb{P}\!\left( A_{2}^{(v)}\right)
\ \ge\ 1-\exp\!\big(-k^{1/24}\big).
\end{equation}
\end{lemma}

\begin{proof}
Under the modified Harris construction, for each fixed $i$ the point processes
$N^{\mathrm{sus}}_{u_i}$, $N^{\mathrm{inf}}_{v\to u_i}$, and $N^{\mathrm{rec}}_{u_i}$
are mutually independent. Moreover, as $i$ varies, the collections
\(
(N^{\mathrm{sus}}_{u_i},\,N^{\mathrm{inf}}_{v\to u_i},\,N^{\mathrm{rec}}_{u_i})_{i=1,\dots,k}
\)
are independent and identically distributed. Consequently, the indicators
$\mathbf{1}\{u_1\in U_v\},\dots,\mathbf{1}\{u_k\in U_v\}$ are i.i.d.
Let $p_k:=\mathbb{P}(u_1\in U_v)$. By independence,
\begin{align*}
p_k
&=\mathbb{P}\!\left(N^{\mathrm{sus}}_{u_1}\big((0,\tfrac12 k^{-1/3})\big)\ge 1\right)
 \mathbb{P}\!\left(N^{\mathrm{inf}}_{v\to u_1}\big([\tfrac12 k^{-1/3},k^{-1/3}]\big)\ge 1\right)
 \mathbb{P}\!\left(N^{\mathrm{rec}}_{u_1}\big([0,(\log k)/4)\big)=0\right) \\
&=\Big(1-\exp\big({-(\rho/2)k^{-1/3}}\big)\Big)\Big(1-\exp\big({-(\lambda/2)k^{-1/3}}\big)\Big)\,\exp\big({-(\log k)/4}\big)\\
&=\Big(1-\exp\big({-(\rho/2)k^{-1/3}}\big)\Big)\Big(1-\exp\big({-(\lambda/2)k^{-1/3}}\big)\Big)\,k^{-1/4}.
\end{align*}
Using $1-e^{-x}\ge x/3$ for $x\in[0,1]$, and noting that
$(\rho/2)k^{-1/3},(\lambda/2)k^{-1/3}\in[0,1]$ for all sufficiently large $k$, we obtain
\[
p_k\ \ge\ \frac{1}{9}\cdot\frac{\rho}{2}k^{-1/3}\cdot\frac{\lambda}{2}k^{-1/3}\cdot k^{-1/4}
\ =\ \frac{\rho\lambda}{36}\,k^{-11/12}.
\]
Therefore,
\[
\sum_{i=1}^k \mathbf{1}\{u_i\in U_v\}\ \sim\ \mathrm{Binomial}(k,p_k),
\qquad
\mathbb{E}\Big[\sum_{i=1}^k \mathbf{1}\{u_i\in U_v\}\Big]=kp_k\ \ge\ \frac{\rho\lambda}{36}\,k^{1/12}.
\]
In particular, for all sufficiently large $k$ we have $kp_k\ge 2k^{1/24}$.
Hence by Lemma~\ref{lemma_largedeviation},
\[
\mathbb{P}\Big(\sum_{i=1}^k \mathbf{1}\{u_i\in U_v\} < k^{1/24}\Big)
\le
\mathbb{P}\Big(\sum_{i=1}^k \mathbf{1}\{u_i\in U_v\} < \tfrac12\, kp_k\Big)
\le \exp\!\big(-c\,kp_k\big)
\le \exp\!\big(-k^{1/24}\big)
\]
for all sufficiently large $k$, where $c>0$ is a constant.
Therefore,
\[
\mathbb{P}\!\left(A_2^{(v)}\right)
=\mathbb{P}\Big(\sum_{i=1}^k \mathbf{1}\{u_i\in U_v\}\ge k^{1/24}\Big)
\ge 1-\exp\!\big(-k^{1/24}\big).
\]
\end{proof}

\medskip
Lemma~\ref{lemma_star_lit_from_A1A2} is the deterministic bridge between the local measurable events $A_1^{(v)}$ and $A_2^{(v)}$
and the \emph{lit} property.

\begin{lemma}\label{lemma_star_lit_from_A1A2}
Fix a graph $G=(V,E)$ and let $G_\star$ denote the star subgraph on vertex set $\{v,u_1,\dots,u_k\}$ with center $v$ (so $G_\star\subseteq G$). Assume that the initial configuration $\boldsymbol{\sigma}(0)$ be arbitrary on $V$ satisfying $\sigma_v(0)=I$.  Let $A_1^{(v)}$ and $A_2^{(v)}$ be the events defined in
Lemma~\ref{lemma_star_A1v} and Lemma~\ref{lemma_star_A2v}, respectively.
Then on $A_1^{(v)}\cap A_2^{(v)}$, the center $v$ is $k^{1/24}$-lit among $\{u_1,\dots, u_k\}$ throughout the time interval
$[0,(\log k)/4]$, for all sufficiently large $k$.
\end{lemma}

\begin{proof}
We assume that $k$ is sufficiently large so that $(\log k)/4 > k^{-1/3}$, and work on the event $A_1^{(v)}\cap A_2^{(v)}$.
On the event $A_1^{(v)}$, we have $N_v^{\mathrm{rec}}([0,k^{-1/3}])=0$. Since $\sigma_v(0)=I$, no recovery mark occurs at $v$ before time $k^{-1/3}$, and hence $v$ remains infected throughout the interval $[0,k^{-1/3}]$.

For each $u\in U_v$ defined in Lemma~\ref{lemma_star_A2v}, we have
\[
N_u^{\mathrm{rec}}\big([0,(\log k)/4]\big)=0,
\]
so once $u$ becomes infected at any time $t\le k^{-1/3}$, it stays infected on the entire interval
$[t,(\log k)/4]$.

We claim that every $u$ satisfying the conditions in \eqref{eq:A2v} is infected by time $k^{-1/3}$,
regardless of its initial state. Consider three cases.
\begin{enumerate}[(i)]
    \item $\sigma_u(0)=I$.
Then $u$ is infected at time $0$, and since $N_u^{\mathrm{rec}}([0,(\log k)/4])=0$, it remains infected on $[0,(\log k)/4]$.
    \item $\sigma_u(0)=S$. In this case, the infection time of $u$ is not necessarily measurable with respect to $\mathcal{G}(v,\{u_1,\dots,u_k\})$. Nevertheless, there exists an event in this $\sigma$-field that guarantees that $u$ becomes infected before time $k^{-1/3}$. Indeed, since $v$ is infected on $[0,k^{-1/3}]$ and
\(
N_{v\to u}^{\mathrm{inf}}\big([\tfrac12 k^{-1/3},\,k^{-1/3}]\big)\ge 1
\),
there exists an infection transmission from $v$ to $u$ at some time in $[\tfrac12 k^{-1/3}, k^{-1/3}]$. If $u$ has not already been infected prior to that infection transmission time, then this infects $u$ at that moment, so in any case $u$ becomes infected by time $k^{-1/3}$. By the no-recovery condition above, once $u$ is infected it remains infected up to time $(\log k)/4$.
\item $\sigma_u(0)=R$.
The condition
\(
N_u^{\mathrm{sus}}\big((0,\tfrac12 k^{-1/3})\big)\ge 1
\)
implies that $u$ becomes susceptible at some time before $\tfrac12 k^{-1/3}$.
After that time, $v$ is still infected on $[\tfrac12 k^{-1/3},k^{-1/3}]$, and again
\(
N_{v\to u}^{\mathrm{inf}}\big([\tfrac12 k^{-1/3},\,k^{-1/3}]\big)\ge 1
\)
forces $u$ to be infected by time $k^{-1/3}$. By $N_u^{\mathrm{rec}}([0,(\log k)/4])=0$,
it remains infected until time $(\log k)/4$.
\end{enumerate}

Therefore, on $A_1^{(v)}\cap A_2^{(v)}$, at least $k^{1/24}$ vertices in $\{u_1,\dots, u_k\}$ are infected throughout the time interval $[k^{-1/3},(\log k)/4]$. Together with the fact that $v$ is infected on $[0,k^{-1/3}]$,
this shows that $v$ is $k^{1/24}$-lit among $\{u_1,\dots, u_k\}$ on the entire interval $[0,(\log k)/4]$.
\end{proof}

\subsection{Persistence of Infected Starting from Infected Central Vertex}
\label{ssec:persist-inf}

Subsection~\ref{ssec:PP} shows that, starting from $v$ infected, the vertex $v$ remains $k^{1/24}$-lit among $U$ up to time $(\log k)/4$ with high probability. In this subsection, we first prove the existence of a local measurable time
\(
T_{I,1}^{(v)}\le \frac14\log k
\)
such that $v$ is infected at time $T_{I,1}^{(v)}$; this is the content of Lemma~\ref{lemma_star_Asigmav}. We then restart the same argument from time $T_{I,1}^{(v)}$ and iterate Lemma~\ref{lemma_star_Asigmav} to construct times
\(
T_{I,2}^{(v)},\dots,T_{I,M}^{(v)}
\).
All these times are local measurable, the vertex $v$ is infected at each time $T_{I,i}^{(v)}$, and $v$ remains $k^{1/24}$-lit among $U$ throughout the interval $[0,T_{I,M}^{(v)}]$; this is formalized in Lemma~\ref{lemma_star_Av}. Finally, we choose $M$ so that the $M$ successive iterations occur with high probability, while
\(
T_{I,M}^{(v)}\ge 3k^{\min\{\rho,\lambda,1\}/200}
\),
as required in part~(ii) of Proposition~\ref{prop:star_independent}.

\begin{lemma}\label{lemma_star_Asigmav}
Fix a graph $G=(V,E)$ and let $G_\star$ denote the star subgraph on vertex set $\{v,u_1,\dots,u_k\}$ with center $v$ (so $G_\star\subseteq G$). Assume that the initial configuration $\boldsymbol{\sigma}(0)$ is arbitrary on $V$ and satisfies $\sigma_v(0)=I$. Let $T_S^{(v)}$ be the first time at which $v$ is susceptible. Then $T_S^{(v)}$ is measurable with respect to $\mathcal{G}(v,\{u_1,\dots,u_k\})$, where $\mathcal{G}(v,\{u_1,\dots,u_k\})$ is defined in Definition~\ref{def:local_graphical_data}. Then there exists an event
\(
A_3^{(v)} \in \mathcal{G}(v,\{u_1,\dots,u_k\})
\)
such that, for all sufficiently large $k$, the following hold:
\begin{enumerate}
    \item[\normalfont (i)] On $A_3^{(v)}$, $T_S^{(v)}\le \frac{1}{6}\log k$, and the vertex $v$ is infected at time $T_S^{(v)}+k^{-1/48}$;
    \item[\normalfont (ii)] On $A_3^{(v)}$, the vertex $v$ is $k^{1/24}$-lit among $\{u_1,\dots,u_k\}$ throughout the time interval  $[0, T_S^{(v)}+k^{-1/48}]$;
    \item[\normalfont (iii)] 
    \(
    \mathbb{P}\big(A_3^{(v)}\big)\ \ge\ 1-2k^{-\min\{1,\rho,\lambda\}/48}.
    \)
\end{enumerate}
\end{lemma}

\begin{proof}
Let $T_R^{(v)}:=\inf\{t> 0:\ t\in N_v^{\mathrm{rec}}\}$ be the first recovery time of $v$. Then $T_S^{(v)}=\inf\{t> T_R^{(v)}:\ t\in N_v^{\mathrm{sus}}\}$, so $T_S^{(v)}$ is measurable with respect to $\{N_v^{\mathrm{rec}},N_v^{\mathrm{sus}}\}$, and hence measurable with respect to $\mathcal{G}(v,\{u_1,\dots,u_k\})$.
Let $A_1^{(v)}$ and $A_2^{(v)}$ (together with $U_v$) be as defined in Lemma~\ref{lemma_star_A1v} and Lemma~\ref{lemma_star_A2v}, respectively. Define $$A_{3,1}^{(v)}:=\Big\{\exists\,u\in U_v:\ N^{\mathrm{inf}}_{u\to v}\big((T_S^{(v)},\,T_S^{(v)}+k^{-1/48}]\big)\ge 1\Big\}, \quad A_{3,2}^{(v)}:=\Big\{N_v^{\mathrm{rec}}\big((T_S^{(v)}, T_S^{(v)}+k^{-1/48}]\big)=0 \Big\},$$ and set
\[
A_3^{(v)}:= A_1^{(v)}\cap A_2^{(v)}\cap A_{3,1}^{(v)} \cap A_{3,2}^{(v)} .
\]
Since both $T_S^{(v)}$ and $U_v$ are measurable with respect to $\mathcal{G}(v,\{u_1,\dots,u_k\})$, it follows that $A_3^{(v)}\in \mathcal{G}(v,\{u_1,\dots,u_k\})$.
By Lemma~\ref{lemma_star_A1v}, 
$$\mathbb{P}\big(A_1^{(v)}\big)\ge 1-k^{-\min\{1,\rho\}/18}.$$ 
By Lemma~\ref{lemma_star_A2v}, 
$$\mathbb{P}\big(A_2^{(v)}\big)\ge 1-\exp(-k^{1/24}).$$

On $A_1^{(v)}$, the conditions $N_v^{\mathrm{rec}}([0,k^{-1/3}])=0$, $N_v^{\mathrm{rec}}(k^{-1/3}, \frac{1}{12}\log k]\ge 1$, and $N_v^{\mathrm{sus}}(\frac{1}{12}\log k, \frac{1}{6}\log k]\ge 1$ imply that $k^{-1/3}<T_S^{(v)}\le \frac{1}{6}\log k$. Moreover, Lemma~\ref{lemma_star_lit_from_A1A2} shows that on $A_1^{(v)}\cap A_2^{(v)}$, the vertex $v$ is $k^{1/24}$-lit among $\{u_1,\dots,u_k\}$ up to time $\frac{1}{4}\log k\ge T_S^{(v)}+k^{-1/48}$. Therefore, it suffices to show that on $A_3^{(v)}$, the vertex $v$ is infected at time $T_S^{(v)}+k^{-1/48}$.

For any $u\in U_v$, Lemma~\ref{lemma_star_lit_from_A1A2} implies that $u$ remains infected throughout the time interval $[k^{-1/3},(\log k)/4]$. Hence, on $A_1^{(v)}\cap A_2^{(v)}$, the event $A_{3,1}^{(v)}$ implies that $v$ becomes infected at some time in the interval $(T_S^{(v)},T_S^{(v)}+k^{-1/48}]$. Note that this infection time need not be measurable with respect to $\mathcal{G}(v,\{u_1,\dots,u_k\})$, since the infection may also arrive at $v$ from outside $\{u_1,\dots, u_k\}$. However, on $A_{3,2}^{(v)}$, once $v$ becomes infected at some time in $(T_S^{(v)},T_S^{(v)}+k^{-1/48}]$, it remains infected at time $T_S^{(v)}+k^{-1/48}$. Therefore, $A_{3,1}^{(v)}\cap A_{3,2}^{(v)}$ implies that $v$ is infected at time $T_S^{(v)}+k^{-1/48}$.

Finally, we estimate the probability of $A_3^{(v)}$. On the event $A_1^{(v)}\cap A_2^{(v)}$, we have $|U_v|\ge k^{1/24}$. For each $u\in U_v$, the event $N_{u\to v}^{\mathrm{inf}}((T_S^{(v)}, T_S^{(v)}+k^{-1/48}])\ge 1$ has probability $1-\exp(- \lambda k^{-1/48})$, and these events are independent, since $T_S^{(v)}$ is measurable with respect to $\{N_v^{\mathrm{rec}},N_v^{\mathrm{sus}}\}$, which is independent of all processes $N^{\mathrm{inf}}_{u\to v}$. Hence
\[
\mathbb{P}(A_{3,1}^{(v)}|A_1^{(v)}\cap A_2^{(v)})\ge \mathbb{P}(\mathrm{Binomial}(k^{1/24}, 1-\exp(-\lambda k^{-1/48}))\ge 1)=1-\exp(-\lambda k^{1/48}).
\]

On the event $A_1^{(v)}$, we have $T_S^{(v)}>k^{-1/3}$. By the independent increments property of the Poisson process, 
$$\mathbb{P}(A_{3,2}^{(v)}|A_1^{(v)}\cap A_2^{(v)})=\exp(-k^{-1/48}) \ge 1-k^{-1/48},$$ 
where we used $e^{-x}\ge 1-x$ for all $x\ge 0$.

Hence, for all sufficiently large $k$,
\begin{align*}
\mathbb{P}(A_3^{(v)})
&\ge 1-\mathbb{P}((A_1^{(v)})^{\mathrm{c}})-\mathbb{P}((A_2^{(v)})^{\mathrm{c}})-\mathbb{P}((A_{3,1}^{(v)})^{\mathrm{c}}|A_1^{(v)}\cap A_2^{(v)})-\mathbb{P}((A_{3,2}^{(v)})^{\mathrm{c}}|A_1^{(v)}\cap A_2^{(v)})\\
&\ge 1-k^{-\min\{1,\rho\}/18}-\exp(-k^{1/24})-\exp(-\lambda k^{1/48})-k^{-1/48}
\ge 1-2k^{-\min\{1,\rho\}/48}.
\end{align*}
\end{proof}

\begin{lemma}\label{lemma_star_Av}
Fix a graph $G=(V,E)$ and let $G_\star$ denote the star subgraph on vertex set $\{v,u_1,\dots,u_k\}$ with center $v$ (so $G_\star\subseteq G$). Assume that the initial configuration $\boldsymbol{\sigma}(0)$ is arbitrary on $V$ satisfying $\sigma_v(0)=I$. Let $T_{I,0}^{(v)}:=0$, and define
$$T_{R,i+1}^{(v)}:=\inf\{t>T_{I,i}^{(v)}: t\in N_v^{\mathrm{rec}} \}, \quad T_{S,i+1}^{(v)}:=\inf\{t>T_{R, i+1}^{(v)}:t\in N_v^{\mathrm{sus}} \}, \quad T_{I,i+1}^{(v)}:=T_{S,i+1}^{(v)}+k^{-1/48} $$
for all $i$; then they are all measurable with respect to $\mathcal{G}(v, \{u_1,\dots, u_k\})$, where $\mathcal{G}(v, \{u_1,\dots, u_k\})$ is defined in Definition~\ref{def:local_graphical_data}. Then there exists an event $A_4^{(v)}\in \mathcal{G}(v, \{u_1,\dots, u_k\})$ such that, for all sufficiently large $k$, the following hold:
\begin{enumerate}
    \item[\normalfont (i)] On $A_4^{(v)}$, $v$ is infected at $T_{I,i}^{(v)}$, and $T_{I,i}^{(v)}-T_{I,i-1}^{(v)}\le \frac{1}{6}\log k+k^{-1/48}$ for all $i=1,\dots,\lfloor 10ek^{\min\{\rho, \lambda,1\}/96}\rfloor$;
    \item[\normalfont (ii)] On $A_4^{(v)}$, $|\{i\le \frac{\frac{1}{2} k^{\min\{\rho,\lambda,1\}/200}}{\log k} : T_{I,i+1}^{(v)}-T_{I,i}^{(v)}\ge 3 \}|\ge \lfloor k^{\min\{\rho,\lambda,1\}/400}\rfloor +1 $;
    \item[\normalfont (iii)] On $A_4^{(v)}$, $v$ is $k^{1/24}$-lit among $\{u_1,\dots, u_k\}$ throughout the time interval $[0, 3k^{\min\{\rho,\lambda,1\}/96}]$;
    \item[\normalfont (iv)] $\mathbb{P}(A_4^{(v)})\ge 1-k^{-\min\{\rho,\lambda,1\}/100}$.
\end{enumerate}
\end{lemma}

\begin{proof}
By Lemma~\ref{lemma_star_Asigmav}, there exists an event $A_{3,1}^{(v)}\in \mathcal{G}(v,\{u_1,\dots,u_k\})$ such that $A_{3,1}^{(v)}$ implies that $v$ is $k^{1/24}$-lit among $\{u_1,\dots,u_k\}$ throughout the time interval $[T_{I,0}^{(v)},T_{I,1}^{(v)}]$, that $v$ is infected at time $T_{I,1}^{(v)}\le \frac{1}{6}\log k+k^{-1/48}$, and that $\mathbb{P}(A_{3,1}^{(v)})\ge 1-2k^{-\min\{\rho,\lambda,1\}/48}$.

Let $M:=\lfloor 10ek^{\min\{\rho,\lambda,1\}/96}\rfloor$. For each $i=1,\dots,M-1$, since $T_{I,i}^{(v)}$ is measurable with respect to $\mathcal{G}(v,\{u_1,\dots,u_k\})$, Lemma~\ref{lemma_star_Asigmav} yields an event $A_{3,i+1}^{(v)}\in \mathcal{G}(v,\{u_1,\dots,u_k\})$ such that $A_{3,i+1}^{(v)}$ implies that $v$ is $k^{1/24}$-lit among $\{u_1,\dots,u_k\}$ throughout the time interval $[T_{I,i}^{(v)},T_{I,i+1}^{(v)}]$, that $v$ is infected at time $T_{I,i+1}^{(v)}$, and that $T_{I,i+1}^{(v)}\le T_{I,i}^{(v)}+\frac{1}{6}\log k+k^{-1/48}$. Moreover,
\[
\mathbb{P}(A_{3,i+1}^{(v)}\mid \mathcal{F}_{T_{I,i}^{(v)}})\ge 1-2k^{-\min\{\rho,\lambda,1\}/48}.
\]

Thus, by iteration, we may define $A_{3,i}^{(v)}$ for all $i=1,\dots,M$. Let
$$A_{4,1}^{(v)}:=\bigcap_{i=1}^M A_{3,i}^{(v)}.$$
Consequently,
\[
\mathbb{P}(A_{4,1}^{(v)})\ge 1-2Mk^{-\min\{\rho,\lambda,1\}/48}\ge 1-20ek^{-\min\{\rho,\lambda,1\}/96}.
\]

The event $A_{4,1}^{(v)}$ implies that $v$ is $k^{1/24}$-lit among $\{u_1,\dots,u_k\}$ throughout the time interval $[0,T_{I,M}^{(v)}]$, that $v$ is infected at $T_{I,i}^{(v)}$, and that $T_{I,i}^{(v)}-T_{I,i-1}^{(v)}\le \frac{1}{6}\log k+k^{-1/48}$ for all $i=1,\dots,M$. It therefore remains to estimate the probabilities of the events
$$A_{4,2}^{(v)}:=\{T_{I,M}^{(v)}\ge 3k^{\min\{\rho,\lambda,1\}/96}\}$$
and
$$A_{4,3}^{(v)}:=\left\{\left|\left\{i\le \frac{\frac{1}{2} k^{\min\{\rho,\lambda,1\}/200}}{\log k} : T_{I,i+1}^{(v)}-T_{I,i}^{(v)}\ge 3 \right\}\right|\ge \lfloor k^{\min\{\rho,\lambda,1\}/400}\rfloor+1 \right\}.$$

Now we work on the event $A_{4,1}^{(v)}$. For each $i=1,\dots,M$, since $v$ is infected at time $T_{I,i-1}^{(v)}$, the waiting time until the next recovery mark of $v$ has the $\mathrm{Exp}(1)$ distribution. Moreover, by the strong Markov property and the independent increments of the Poisson process $N_v^{\mathrm{rec}}, N_v^{\mathrm{sus}}$, the random variables $\{T_{S,i}^{(v)}-T_{I,i-1}^{(v)}\}_{i=1}^{M}$ are independent, and $A_{4,2}^{(v)}, A_{4,3}^{(v)}\in \mathcal{G}(v,\{u_1,\dots, u_k\})$. In particular,
\[
\mathbb{P}\big(T_{S,i}^{(v)}-T_{I,i-1}^{(v)}\ge 1\big)\ge e^{-1}, \quad \mathbb{P}\big(T_{S,i}^{(v)}-T_{I,i-1}^{(v)}\ge 3\big)\ge e^{-3}, \qquad i=1,\dots,M.
\]
Therefore, by Lemma~\ref{lemma_largedeviation}, there exists a constant $c>0$ such that
\begin{eqnarray*}
   &&\mathbb{P}(T_{I,M}^{(v)}\ge 3k^{\min\{\rho,\lambda,1\}/96}\mid A_{4,1}^{(v)})\\
   &\ge& \mathbb{P}\Big(|\{i\in\{1,\dots,M\}: T_{S,i}^{(v)}-T_{I,i-1}^{(v)}\ge 1\}|\ge 3k^{\min\{\rho,\lambda,1\}/96}\mid A_{4,1}^{(v)}\Big) \\
   &\ge& \mathbb{P}\Big(\mathrm{Binomial}(\lfloor 10e\,k^{\min\{\rho,\lambda,1\}/96}\rfloor,\,e^{-1})\ge 3k^{\min\{\rho,\lambda,1\}/96}\Big) \\
   &\ge& 1-\exp\!\big(-c\,k^{\min\{\rho,\lambda,1\}/96}\big),
\end{eqnarray*}
and similarly by Lemma~\ref{lemma_largedeviation}, there exists a constant $c'>0$ such that for all sufficiently large $k$,
\begin{eqnarray*}
   &&\mathbb{P}(A_{4,3}^{(v)}\mid A_{4,1}^{(v)})\\
   &\ge& \mathbb{P}\left(\mathrm{Binomial}\left(\left\lfloor\frac{\frac{1}{2} k^{\min\{\rho,\lambda,1\}/200}}{\log k}\right\rfloor,e^{-3}\right)\ge k^{\min\{\rho,\lambda,1\}/400}+1\right) \\
   &\ge& 1-\exp\!\big(-c'k^{\min\{\rho,\lambda,1\}/400}\big).
\end{eqnarray*}
Let
$$A_4^{(v)}:=A_{4,1}^{(v)}\cap A_{4,2}^{(v)}\cap A_{4,3}^{(v)}.$$
Then, on $A_4^{(v)}$, item~(i) follows from the definition of $A_{4,1}^{(v)}$, item~(iii) follows from the definition of $A_{4,2}^{(v)}$, and item~(ii) follows since on $A_{4,3}^{(v)}$ there are at least $\lceil k^{\min\{\rho,\lambda,1\}/400}\rceil $ indices $i$ such that $T_{S,i+1}^{(v)}-T_{I,i}^{(v)}\ge 3$, which implies $T_{I,i+1}^{(v)}-T_{I,i}^{(v)}\ge 3$. We choose the first $\lceil k^{\min\{\rho,\lambda,1\}/400}\rceil $ such indices.

Finally, for all sufficiently large $k$,
\begin{align*}
\mathbb{P}(A_4^{(v)})
&\ge 1-\mathbb{P}\big((A_{4,1}^{(v)})^{\mathrm{c}}\big)-\mathbb{P}\big((A_{4,2}^{(v)})^{\mathrm{c}}\mid A_{4,1}^{(v)}\big)-\mathbb{P}\big((A_{4,3}^{(v)})^{\mathrm{c}}\mid A_{4,1}^{(v)}\big)\\
&\ge 1-20ek^{-\min\{\rho,\lambda,1\}/96}-\exp(-ck^{\min\{\rho,\lambda,1\}/96})-\exp(-c'k^{\min\{\rho,\lambda,1\}/400})\\
&\ge 1-k^{-\min\{\rho,\lambda,1\}/100}.
\end{align*}
\end{proof}

\subsection{Persistence of Infected Starting from Lit Central Vertex}
\label{ssec:persist-lit}

If Proposition~\ref{prop:star_independent} is started from $\sigma_v(0)=I$, then Lemma~\ref{lemma_star_Av} already gives the desired conclusion.

We next consider the case where $v$ is initially $k^{1/24}$-lit but not infected. In the worst-case configuration, $v$ has exactly $k^{1/24}$ infected neighbors and all other vertices in the graph are uninfected. In such a situation, with high probability, one of these infected neighbors recovers before any new infection reaches $v$, so $v$ cannot be expected to maintain the full $k^{1/24}$-lit level throughout the initial waiting period. Nevertheless, Lemma~\ref{lemma_star_startwithvnotinfected} shows that, for every initial configuration, with high probability $v$ becomes infected within a short time window, while up to that time it remains at least $k^{1/96}$-lit. Thus, although $v$ may lose some infected neighbors before becoming infected, the loss is controlled and the lit level does not drop too far.

\begin{lemma}\label{lemma_star_startwithvnotinfected} 
Fix a graph $G=(V,E)$ and let $G_\star$ denote the star subgraph on vertex set $\{v,u_1,\dots,u_k\}$ with center $v$ (so $G_\star\subseteq G$). Assume that the initial configuration $\boldsymbol{\sigma}(0)$ is arbitrary on $V$, satisfies $\sigma_v(0)\neq I$ and $v$ is $k^{1/24}$-lit among $\{u_1,\dots,u_k\}$ at time $0$. Let $T_S^{(v)}$ be the first time at which $v$ is susceptible. Then $T_S^{(v)}$ is measurable with respect to $\mathcal{G}(v,\{u_1,\dots,u_k\})$, where $\mathcal{G}(v,\{u_1,\dots,u_k\})$ is defined in Definition~\ref{def:local_graphical_data}. Then there exists an event
\(
A_5^{(v)} \in \mathcal{G}(v,\{u_1,\dots,u_k\})
\)
, such that, for all sufficiently large $k$, the following hold:
\begin{enumerate}
    \item[\normalfont (i)] On $A_5^{(v)}$, the vertex $v$ is infected at time $T_S^{(v)}+k^{-1/100}\le \frac{1}{48}\log k$;
    \item[\normalfont (ii)] On $A_5^{(v)}$, the vertex $v$ is $k^{1/96}$-lit among $\{u_1,\dots,u_k\}$ throughout the time interval $[0,T_S^{(v)}+k^{-1/100}]$;
    \item[\normalfont (iii)] $\mathbb{P}\big(A_5^{(v)}\big)\ge 1-k^{-\min\{\rho,\lambda,1\}/192}$.
\end{enumerate}
\end{lemma}

\begin{proof}
If $\sigma_v(0)=R$, then $T_S^{(v)}$ is $N_v^{\mathrm{sus}}$-measurable, hence $\mathcal{G}(v,\{u_1,\dots,u_k\})$-measurable. If $\sigma_v(0)=S$, then $T_S^{(v)}=0$.

Let $U_I:=\{u\in \{u_1,\dots,u_k\}:\sigma_u(0)=I\}$ denote the set of infected vertices in $\{u_1,\dots,u_k\}$ at time $0$. Then $|U_I|\ge k^{1/24}$. Define
$$\bar{U}_I:=\{u\in U_I: N_u^{\mathrm{rec}}([0,\tfrac{1}{48}\log k])=0\}, \qquad A_{5,1}^{(v)}:=\{|\bar{U}_I|\ge k^{1/96}\},$$
$$A_{5,2}^{(v)}:=\{N_v^{\mathrm{sus}}([0,\tfrac{1}{96}\log k])\ge 1\},$$
$$A_{5,3}^{(v)}:=\{\exists u\in \bar{U}_I: N_{u\to v}^{\mathrm{inf}}((T_S^{(v)}, T_S^{(v)}+k^{-1/100}])\ge 1\}, \qquad A_{5,4}^{(v)}:=\{N_v^{\mathrm{rec}}((T_S^{(v)}, T_S^{(v)}+k^{-1/100}])=0\},$$
and
$$A_5^{(v)}:= A_{5,1}^{(v)}\cap A_{5,2}^{(v)}\cap A_{5,3}^{(v)}\cap A_{5,4}^{(v)}.$$
Then $A_{5,1}^{(v)}, A_{5,2}^{(v)}, A_{5,3}^{(v)}, A_{5,4}^{(v)}$, and $A_5^{(v)}$ are all $\mathcal{G}(v,\{u_1,\dots,u_k\})$-measurable.

The event $A_{5,1}^{(v)}$ implies that $v$ is $k^{1/96}$-lit among $\{u_1,\dots,u_k\}$ throughout the time interval $[0,\tfrac{1}{48}\log k]$. The event $A_{5,2}^{(v)}$ implies that $T_S^{(v)}\le \tfrac{1}{96}\log k$. Hence, on $A_{5,1}^{(v)}\cap A_{5,2}^{(v)}$, the vertex $v$ is $k^{1/96}$-lit among $\{u_1,\dots,u_k\}$ throughout the time interval $[0, T_S^{(v)}+k^{-1/100}]$. It therefore remains to show that $v$ is infected at time $T_S^{(v)}+k^{-1/100}$ on the event $A_5^{(v)}$.

The event $A_{5,3}^{(v)}$ implies that $v$ becomes infected at some time in the interval $(T_S^{(v)}, T_S^{(v)}+k^{-1/100}]$. Note that this infection time need not be measurable with respect to $\mathcal{G}(v,\{u_1,\dots,u_k\})$, since the infection may also arrive at $v$ from outside $\{u_1,\dots,u_k\}$. However, on $A_{5,4}^{(v)}$, once $v$ becomes infected at some time in $(T_S^{(v)}, T_S^{(v)}+k^{-1/100}]$, it remains infected at time $T_S^{(v)}+k^{-1/100}$. Therefore, $A_{5,3}^{(v)}\cap A_{5,4}^{(v)}$ implies that $v$ is infected at time $T_S^{(v)}+k^{-1/100}$. 

Finally, we estimate the probability of $A_5^{(v)}$. First,
$$\mathbb{P}(A_{5,2}^{(v)})=1-k^{-\rho/96}.$$
The events $\{N_u^{\mathrm{rec}}([0,\tfrac{1}{48}\log k])=0\}_{u\in U_I}$ are independent, and each has probability $k^{-1/48}$. Since $|U_I|\ge k^{1/24}$, Lemma~\ref{lemma_largedeviation} yields
$$\mathbb{P}(A_{5,1}^{(v)})\ge \mathbb{P}(\mathrm{Binomial}(k^{1/24}, k^{-1/48})\ge k^{1/96} ) \ge 1-\exp(-ck^{1/96})$$
for some constant $c>0$.
Since $T_S^{(v)}$ is $N_v^{\mathrm{sus}}$-measurable and independent of all processes $N_{u\to v}^{\mathrm{inf}}$, the events $\{N_{u\to v}^{\mathrm{inf}}((T_S^{(v)}, T_S^{(v)}+k^{-1/100}])\ge 1\}_{u\in \bar{U}_I}$ are independent, and each has probability $1-\exp(-\lambda k^{-1/100})$. Hence
$$\mathbb{P}(A_{5,3}^{(v)}|A_{5,1}^{(v)})\ge \mathbb{P}(\mathrm{Binomial}(k^{1/96}, 1-\exp(-\lambda k^{-1/100}))\ge 1)=1-\exp(-\lambda k^{1/2400}).$$
Since $T_S^{(v)}$ is $N_v^{\mathrm{sus}}$-measurable and independent of $N_v^{\mathrm{rec}}$,
$$\mathbb{P}(A_{5,4}^{(v)})=\exp(-k^{-1/100})\ge 1-k^{-1/100},$$
where we used $e^{-x}\ge 1-x$ for all $x\ge 0$.

Therefore, for all sufficiently large $k$,
\begin{align*}
\mathbb{P}(A_5^{(v)})
&\ge 1-\mathbb{P}((A_{5,1}^{(v)})^{\mathrm{c}})-\mathbb{P}((A_{5,2}^{(v)})^{\mathrm{c}})-\mathbb{P}((A_{5,3}^{(v)})^{\mathrm{c}}|A_{5,1}^{(v)})-\mathbb{P}((A_{5,4}^{(v)})^{\mathrm{c}}) \\
&\ge 1-\exp(-ck^{1/96})-\exp(-\lambda k^{1/2400})-k^{-\rho/96}-k^{-1/100} 
\ge 1-k^{-\min\{\rho,1\}/192}.
\end{align*}
\end{proof}

\medskip
Finally, we prove Proposition~\ref{prop:star_independent} by combining the two cases treated above. If $\sigma_v(0)=I$, then Lemma~\ref{lemma_star_Av} directly yields the desired event. If $\sigma_v(0)\neq I$, then Lemma~\ref{lemma_star_startwithvnotinfected} first provides a short initial time interval on which $v$ remains weakly lit and becomes infected at a local measurable time, after which Lemma~\ref{lemma_star_Av} can be applied from that time onward. In either case, we obtain an event measurable with respect to $\mathcal{G}(v,U)$ on which items~(i), (ii), and (iii) hold, and the corresponding probability bound yields item~(iv).

\begin{proof}[Proof of Proposition~\ref{prop:star_independent}]
It suffices to consider the case $|U|=k$. Indeed, if $|U|>k$, then since $v$ is $k^{1/24}$-lit among $U$, we may choose a subset $U'\subseteq U$ with $|U'|=k$ such that $v$ is still $k^{1/24}$-lit among $U'$. Since $\mathcal{G}(v,U')\subseteq \mathcal{G}(v,U)$, any event constructed using $U'$ is also measurable with respect to $\mathcal{G}(v,U)$.

\medskip
\noindent\emph{Case 1: $\sigma_v(0)=I$.}
Let $A^{(v)}:=A_4^{(v)}$, where $A_4^{(v)}$ is the event constructed in Lemma~\ref{lemma_star_Av}. On $A_4^{(v)}$, the vertex $v$ remains $k^{1/24}$-lit among $U$ throughout the time interval $[0, 3k^{\min\{\rho,\lambda,1\}/96}]$. Since $3k^{\min\{\rho,\lambda,1\}/96}\ge 3k^{\min\{\rho,\lambda,1\}/200}$, $A_4^{(v)}$ implies part~(ii) of Proposition~\ref{prop:star_independent}. Since $k^{1/24}\ge k^{1/96}$ and $\frac{1}{48}\log k\le 3k^{\min\{\rho,\lambda,1\}/96}$, $A_4^{(v)}$ also implies part~(i) of Proposition~\ref{prop:star_independent}.

On $A_4^{(v)}$, we may choose the reinfection times so that
\(
T_{I,i+1}^{(v)}-T_{I,i}^{(v)}\ge 3
\)
for all
\(
i\le 
\frac{\frac12 k^{\min\{\rho,\lambda,1\}/200}}
{\log k+k^{-1/48}}
\), which is defined in Lemma~\ref{lemma_star_Av}.
Denote these times by $(T_j)_{j\le \lfloor k^{\min\{\rho,\lambda,1\}/400}\rfloor+1}$. Then the sequence $(T_j)$ satisfies the requirement in part~(iii) of Proposition~\ref{prop:star_independent} since for any $j\le \lfloor k^{\min\{\rho,\lambda,1\}/400}\rfloor+1$, $T_{j+1}-T_j\ge 3$, and
\[
T_j\le T^{(v)}_{I,\left\lfloor 
\frac{\frac12 k^{\min\{\rho,\lambda,1\}/200}}{\log k}
\right\rfloor}
\le
\left(\frac16\log k+k^{-1/48}\right)
\frac{\frac12 k^{\min\{\rho,\lambda,1\}/200}}
{\log k+k^{-1/48}}
\le
\frac12 k^{\min\{\rho,\lambda,1\}/200},
\]
for all sufficiently large $k$. 

Moreover,
\[
\mathbb{P}\big(A_4^{(v)}\big)\ge 1-k^{-\min\{\rho,\lambda,1\}/100}\ge 1-k^{-\min\{\rho,\lambda,1\}/200}
\]
for all sufficiently large $k$.

\medskip
\noindent\emph{Case 2: $\sigma_v(0)\neq I$.}
Let $A_5^{(v)}$ and $T_S^{(v)}$ be as defined in Lemma~\ref{lemma_star_startwithvnotinfected}; both are measurable with respect to $\mathcal{G}(v,U)$. Then
\[
\mathbb{P}\big(A_5^{(v)}\big)\ge 1-k^{-\min\{\rho,\lambda,1\}/192}.
\]
On $A_5^{(v)}$, we have $T_S^{(v)}\le \frac{1}{96}\log k$, the vertex $v$ is infected at time $T_S^{(v)}+k^{-1/100}$, and $v$ is $k^{1/96}$-lit among $U$ throughout the time interval $[0,T_S^{(v)}+k^{-1/100}]$.

Set $\tau:=T_S^{(v)}+k^{-1/100}\le \frac{1}{48}\log k$. Then $\tau$ is also measurable with respect to $\mathcal{G}(v,U)$. By the strong Markov property with respect to the natural filtration $\{\mathcal{F}_t\}$ generated by the modified Harris construction, and applying the time-shifted version of Lemma~\ref{lemma_star_Av} from time $\tau$ onward, we may construct an event $A_{4,\tau}^{(v)}\in \mathcal{G}(v,U)$ such that on $A_{4,\tau}^{(v)}$, the vertex $v$ remains $k^{1/24}$-lit among $U$ on $[\tau,\tau+3k^{\min\{\rho,\lambda,1\}/96}]$, and
\[
\mathbb{P}\big(A_{4,\tau}^{(v)}\mid \mathcal{F}_\tau \big)\ge 1-k^{-\min\{\rho,\lambda,1\}/100}.
\]

Define
\[
A^{(v)}:=A_5^{(v)}\cap A_{4,\tau}^{(v)}\in\mathcal{G}(v,U).
\]
Then, on $A^{(v)}$, the vertex $v$ is $k^{1/96}$-lit among $U$ throughout the time interval $[0,\tau]$, and $v$ is $k^{1/24}$-lit among $U$ throughout the time interval $[\tau,\tau+3k^{\min\{\rho,\lambda,1\}/96}]$. Since $\tau\le \frac{1}{48}\log k$, $k^{1/24}\ge k^{1/96}$, and $\tau+3k^{\min\{\rho,\lambda,1\}/96}\ge 3k^{\min\{\rho,\lambda,1\}/200}$, it follows that part~(i) and part~(ii) hold. Part~(iii) holds by construction of $A_{4,\tau}^{(v)}$, since, for any $j\le \lfloor k^{\min\{\rho,\lambda,1\}/400}\rfloor+1$, $T_{j+1}-T_j\ge 3$ still holds, and
\[
T_j+\tau \le \frac{1}{2}k^{\min\{\rho,\lambda,1\}/200}+\tau \le k^{\min\{\rho,\lambda,1\}/200}.
\]
Finally,
\[
\mathbb{P}\big(A^{(v)}\big)\ge 1-\mathbb{P}\big((A_5^{(v)})^{\mathrm{c}}\big)-\mathbb{E}\Big[\mathbb{P}\big((A_{4,\tau}^{(v)})^{\mathrm{c}}\mid \mathcal{F}_\tau\big)\mathbf{1}_{A_5^{(v)}}\Big]\ge 1-k^{-\min\{\rho,\lambda,1\}/200}
\]
for all sufficiently large $k$. This completes the proof of Proposition~\ref{prop:star_independent}.
\end{proof}

Note that in the statements of Proposition~\ref{prop:star_independent}, Lemma~\ref{lemma_star_Asigmav}, Lemma~\ref{lemma_star_Av}, and Lemma~\ref{lemma_star_startwithvnotinfected}, the term $\min\{\lambda,\rho,1\}$ can in fact be replaced by $\min\{\rho,1\}$, yielding a formally stronger result. We nevertheless keep the term $\min\{\lambda,\rho,1\}$ in order to formulate these arguments in a way that also applies to the threshold-SIRS process.

One can define the modified Harris construction for the threshold-SIRS process in a completely analogous way. The only difference is that the infection clocks are attached to vertices rather than to oriented edges. Thus the graphical construction becomes
\(
\mathcal{H}:=(N_v^{\mathrm{rec}},N_v^{\mathrm{sus}},N_v^{\mathrm{inf}})_{v\in V}
\),
and the process $\boldsymbol{\sigma}(t)$ is defined from the initial configuration $\boldsymbol{\sigma}(0)$ and the realization of $\mathcal{H}$ in the same manner as before.

Accordingly, the local graphical data for the threshold-SIRS process is defined by
\(
\mathcal{G}(v,U):=(N_x^{\mathrm{rec}},N_x^{\mathrm{sus}},N_x^{\mathrm{inf}})_{x\in \{v\}\cup U}
\).

With this modification, Lemma~\ref{lemma_star_A2v} continues to hold after replacing each term $N_{v\to u}^{\mathrm{inf}}$ by $N_u^{\mathrm{inf}}$, with probability at least $1-k^{-\lambda/24}$. Similarly, Lemma~\ref{lemma_star_A1v}, Lemma~\ref{lemma_star_lit_from_A1A2}, Lemma~\ref{lemma_star_Asigmav}, Lemma~\ref{lemma_star_Av}, and Lemma~\ref{lemma_star_startwithvnotinfected} remain valid with the same probability bounds, and therefore yield the same conclusion as Proposition~\ref{prop:star_independent}. The only difference is that the constructions of the events $A_3^{(v)}$, $A_4^{(v)}$, and $A_5^{(v)}$ must be adjusted to reflect the modified Harris construction of the threshold-SIRS.

\section{Hierarchical-Stars and Their Exponential Survival}\label{subsection_hierarchicalstar_1}
In this section we analyze the SIRS dynamics on a graph containing a single designated hierarchical-star, as defined in Definition~\ref{def:hierarchical-star}, and prove the main persistence result stated in Proposition~\ref{hierarchicalstar_thm2}.

\begin{definition}[\normalfont Hierarchical-star structure]\label{def:hierarchical-star}
 We say that a graph $G=(V,E)$ \emph{contains a $( k_1 ,k_2)$--hierarchical-star} if there
exist distinct vertices
\[
v,\; u_1,\ldots,u_{\lceil k_1\rceil},\; w_{i,1},\ldots,w_{i,\lceil k_2\rceil}\qquad (i=1,\ldots,\lceil k_1\rceil)
\]
such that the following edges are present:
\begin{enumerate}[{\normalfont (a)}]
\item $v$ is adjacent to $u_i$ for every $i=1,\ldots,\lceil k_1\rceil$;
\item for each $i=1,\ldots,\lceil k_1\rceil$, the vertex $u_i$ is adjacent to $w_{i,1},\ldots,w_{i,\lceil k_2\rceil}$.
\end{enumerate}
The spanning subgraph on these vertices consisting \emph{only} of the edges
\[
\big\{\{v,u_i\}: i=1,\ldots,\lceil k_1\rceil\big\}\ \cup\ \big\{\{u_i,w_{i,j}\}: i=1,\ldots,\lceil k_1\rceil,\ j=1,\ldots,\lceil k_2\rceil\big\}
\]
is called a \emph{$(k_1,k_2)$--hierarchical-star}, with center $v$, first-layer vertices $\{u_i\}_{i\le \lceil k_1\rceil}$, and
second-layer vertices $\{w_{i,j}\}_{i\le \lceil k_1\rceil,\,j\le \lceil k_2\rceil}$.
We denote this designated subgraph by $\mathrm{HStar}(v;\{u_i\}_{i\le \lceil k_1\rceil};\{w_{i,j}\}_{i\le \lceil k_1\rceil,\,j\le \lceil k_2\rceil})$.
\end{definition}

\medskip
\noindent
In Definition~\ref{def:hierarchical-star}, the graph $G$ may contain additional edges among the vertices
$\{v\}\cup\{u_i\}_{i\le \lceil k_1\rceil}\cup\{w_{i,j}\}_{i\le \lceil k_1\rceil,\,j\le \lceil k_2\rceil}$; the term ``hierarchical-star'' refers only to the
designated subgraph specified above.

\begin{definition}[\normalfont HS-lit]\label{def:HS-lit}
Suppose that the graph $G=(V,E)$ contains a hierarchical-star structure
$\mathrm{HStar}(v;\{u_i\}_{i\le \lceil k_1\rceil};\{w_{i,j}\}_{i\le \lceil k_1\rceil, j\le \lceil k_2\rceil})$ as in Definition~\ref{def:hierarchical-star}.
We say that $v$ is \emph{$(k_3,k_4)$-HS-lit} with respect to this
hierarchical-star at time $t$ if at least $\lceil k_3 \rceil$ of the vertices $u_1,\ldots,u_{\lceil k_1\rceil}$ are $k_4$-lit among their
corresponding second-layer neighbors, i.e.
\[
\big|\big\{\, i\in\{1,\ldots,\lceil k_1\rceil\}:\ u_i \text{ is $k_4$-lit among } \{w_{i,1},\ldots,w_{i,\lceil k_2\rceil}\} \text{ at time } t\,\big\}\big|
\ \ge\ \lceil k_3\rceil .
\]
\end{definition}

\begin{definition}[\normalfont Local graphical $\sigma$-field on a hierarchical-star]\label{def:hierarchical-star_sigmaalgebra}
Define the local graphical $\sigma$-field on the hierarchical-star by
\begin{align*}
 \mathcal{G}(&\mathrm{HStar}(v;\{u_i\}_{i\le \lceil k_1\rceil};\{w_{i,j}\}_{i\le \lceil k_1\rceil, j\le \lceil k_2\rceil})) \\
&:=   \sigma\!\Big(
N_v^{\mathrm{rec}},N_v^{\mathrm{sus}}, (N_{u_i}^{\mathrm{rec}},N_{u_i}^{\mathrm{sus}}, N_{v\to u_i}^{\mathrm{inf}}, N_{u_i\to v}^{\mathrm{inf}})_{i\le \lceil k_1\rceil}, (N_{w_{i,j}}^{\mathrm{rec}},N_{w_{i,j}}^{\mathrm{sus}}, N_{u_i\to w_{i,j}}^{\mathrm{inf}}, N_{w_{i,j}\to u_i}^{\mathrm{inf}})_{i\le \lceil k_1\rceil, j\le \lceil k_2\rceil}
\Big).  
\end{align*}
\end{definition}

\medskip
Proposition~\ref{hierarchicalstar_thm2} is the main result of this section. Starting from a ``moderately HS-lit'' configuration at the center, it shows that the center remains HS-lit for a stretched-exponential amount of time. More precisely, the center first maintains a weaker level of HS-lit state and subsequently a stronger level, both with high probability. Moreover, the corresponding good event is measurable with respect to the local graphical $\sigma$-field associated with the hierarchical-star.

\begin{proposition}\label{hierarchicalstar_thm2}
Let $G=(V,E)$ be a graph that contains the designated hierarchical-star
\[
HS(v):=\mathrm{HStar}\!\Big(v;\{u_i\}_{i\le \lceil n^{\varepsilon}\rceil};\{w_{i,j}\}_{i\le  \lceil n^{\varepsilon}\rceil,\,j\le \lceil n^{\varepsilon_1}\rceil }\Big)
\]
as in Definition~\ref{def:hierarchical-star}, for some $\varepsilon\ge \varepsilon_1>0$. Let the initial configuration $\boldsymbol{\sigma}(0)$ be arbitrary on $V$,
subject only to the condition that $v$ is $(k,n^{\varepsilon_2})$-HS-lit with respect to $HS(v)$ for some
\(
k\in[\tfrac14 n^{\varepsilon_3},\, n^{\varepsilon_3}],
\
\varepsilon_2=\tfrac{1}{200}\min\{\rho,\lambda,1\}\varepsilon_1,\
\varepsilon_3=\tfrac14\varepsilon_2,\
\varepsilon_4=\tfrac{1}{16\cdot 200}\min\{\rho, \lambda, 1\}\cdot24\varepsilon_2.
\)

Then, for all sufficiently large $n$, there exists an event $B^{(v)}\in \mathcal{G}(HS(v))$ such that
\begin{enumerate}[{\normalfont (i)}]
\item on $B^{(v)}$, the vertex $v$ remains $(n^{\varepsilon_4},n^{\varepsilon_4})$-HS-lit with respect to $HS(v)$
throughout the time interval $[0,\exp(n^{\varepsilon_4})]$;
\item on $B^{(v)}$, the vertex $v$ remains $(\tfrac14 n^{\varepsilon_3},n^{\varepsilon_2})$-HS-lit with respect to $HS(v)$
throughout the time interval $[\exp(n^{\varepsilon_4}),\exp(n^{\varepsilon_3/2})]$;
\item on $B^{(v)}$, there exist random times $T_i$ for $i=1,\dots, \lceil\exp(n^{\varepsilon_4/8})\rceil+1$, such that each $T_i$ is measurable with respect to $\mathcal{G}(HS(v))$, $T_{i+1}-T_i\ge n^{\varepsilon_2}, T_i\le \frac{1}{2}\exp(n^{\varepsilon_3/2})$, and $v$ is infected at time $T_i$ for every $i=1,\dots, \lceil\exp(n^{\varepsilon_4/8})\rceil$;
\item $\mathbb{P}\big(B^{(v)}\big)\ge 1-\exp(-n^{\varepsilon_4/10})$.
\end{enumerate}
\end{proposition}

Define $(W_t)_{t\ge 0}$ to be the number of first-layer vertices that are $n^{\varepsilon_1/24}$-lit at time $t$, then $W_t$ is not measurable with respect to $\mathcal{G}(HS(v))$. All stretched-exponential time scales appearing in Proposition~\ref{hierarchicalstar_thm2} ultimately stem from Lemma~\ref{lemma_hierarchicalstar_supermartingale}. In that lemma, we construct a comparison process $(\tilde W_r)_{r\in \mathbb{N}}$ such that $\tilde{W}_r\le W_{3n^{\varepsilon_2}r}$ for all $r\in\mathbb{N}$, $\tilde{W}_r$  is measurable with respect to $\mathcal{G}(HS(v))$, and such that $\mathbb{E}[\exp(-\tilde W_{r+1})\mid\mathcal{F}_{3n^{\varepsilon_2}r} ]\le \exp(-\tilde{W}_r)$  whenever $\tilde{W}_r\in [\frac14 n^{\varepsilon_3}+1,\,\lceil n^{\varepsilon_3}\rceil]$. Hence, starting from $\tilde W_0\ge \frac12 n^{\varepsilon_3}$, the probability that $\tilde W_r$ exits this interval through the lower endpoint is exponentially small. Lemma~\ref{lemma_hierarchicalstar_expsurvivaltime} then shows that, with high probability, the process either returns to the level $n^{\varepsilon_3}$ repeatedly or never drops below the lower level before time $\exp(n^{\varepsilon_3/2})$. This yields the persistence estimate used in part~(ii) of Proposition~\ref{hierarchicalstar_thm2}.

For part~(i), one has to start from a weaker initial condition. A worst-case scenario is that initially exactly $\lceil\frac14 n^{\varepsilon_3}\rceil$ first-layer vertices are $n^{\varepsilon_2}$-lit, while all other vertices, including the center $v$, are not infected, so that $W_0=\lceil\frac14 n^{\varepsilon_3}\rceil$. In this regime, the center can no longer remain $(\frac14 n^{\varepsilon_3},n^{\varepsilon_2})$-HS-lit until reinfection occurs. Nevertheless, another application of Lemma~\ref{lemma_hierarchicalstar_expsurvivaltime} at a smaller scale, shows that the center still remains $(n^{\varepsilon_4},n^{\varepsilon_4})$-HS-lit for a time of order $\exp(n^{\varepsilon_4})$, where $\varepsilon_4$ is much smaller compared to $\varepsilon_3$. This provides the weaker persistence estimate needed for part~(i).

Lemma~\ref{lemma_edge_transmission_one} gives a rigorous formulation of what it means for infection to spread across a single edge, and shows that this occurs with a fixed positive probability within one unit of time. Since the graph distance between the center and any vertex in the same hierarchical-star is at most $4$, it follows that, as long as infection is present somewhere in the hierarchical-star, every vertex in that hierarchical-star has a fixed positive probability of becoming infected within time $4$. Consequently, if infection persists for a time of order $n^{\varepsilon_2}$, then every vertex in the hierarchical-star is infected at least once before time $n^{\varepsilon_2}$ with high probability. This is formalized in Lemma~\ref{lemma_edge_transmission_k}.

Combining Lemma~\ref{lemma_edge_transmission_k} with Lemma~\ref{lemma_hierarchicalstar_expsurvivaltime}, we deduce that the center reaches the state $(n^{\varepsilon_3},n^{\varepsilon_2})$-HS-lit before time $\exp(n^{\varepsilon_4})$ with high probability; this is precisely the content of Lemma~\ref{lemma_hierarchicalstar_infectedtoHSlit}. Once this stronger HS-lit level has been reached, another application of Lemma~\ref{lemma_hierarchicalstar_expsurvivaltime} yields part~(ii) of Proposition~\ref{hierarchicalstar_thm2}, with arbitrary $k\in[\frac{1}{4}n^{\varepsilon_3}, n^{\varepsilon_3}]$.

\begin{lemma}\label{lemma_edge_transmission_one}
Let $G=(V,E)$ be a graph and suppose $\{u,v\}\in E$. Assume $\sigma_u(0)=I$. Then there exists a positive constant (depending only on $\lambda$ and $\rho$)
\[
c_{\mathrm{edge}}:=e^{-2}\bigl(1-e^{-\lambda/2}\bigr)\bigl(1-e^{-\rho/2}\bigr)>0
\]
and an event $E\in \mathcal{G}(u,\{v\})$ (as in Definition~\ref{def:local_graphical_data}) such that $E\subseteq\{\sigma_v(1)=I\}$ and
\[
\mathbb{P}(E)\ \ge\ c_{\mathrm{edge}}.
\]
\end{lemma}

\begin{proof}
Consider the event
\[
E:=\Big\{N^{\mathrm{rec}}_{u}([0,1])=0\Big\}\ \cap\
\Big\{N^{\mathrm{sus}}_{v}\big([0,\tfrac{1}{2}]\big)\ge 1\Big\}\ \cap\
\Big\{N^{\mathrm{inf}}_{u\to v}\big((\tfrac{1}{2},1]\big)\ge 1\Big\}\ \cap\
\Big\{N^{\mathrm{rec}}_{v}([0,1])=0\Big\}.
\]
On $E$, the condition $N^{\mathrm{rec}}_{u}([0,1])=0$ implies that the vertex $u$ remains infected throughout $[0,1]$. Likewise, the condition $N^{\mathrm{rec}}_{v}([0,1])=0$ implies that once $v$ becomes infected before time $1$, it is still infected at time $1$. Therefore, without loss of generality, we may assume that $\sigma_v(0)\in\{S,R\}$.

Moreover, the condition
\(
N^{\mathrm{sus}}_{v}\big([0,\tfrac{1}{2}]\big)\ge 1
\)
implies that $v$ is in state $S$ at some time $t\le \tfrac12$ (trivially if $\sigma_v(0)=S$, and otherwise at the first $R\to S$ mark). Since there is at least one infection transmission from $u$ to $v$ during $(\tfrac{1}{2},1]$, it follows that $v$ becomes infected before time $1$. Hence $E\subseteq\{\sigma_v(1)=I\}$.

Finally, by independence in the modified Harris construction,
\[
\mathbb{P}(E)= e^{-1}\bigl(1-e^{-\rho/2}\bigr)\bigl(1-e^{-\lambda/2}\bigr)e^{-1}=c_{\mathrm{edge}}.
\]
\end{proof}

\begin{lemma}\label{lemma_edge_transmission_k}
Let $G=(V,E)$ be a graph that contains the designated hierarchical-star 
\(
HS(v):=\mathrm{HStar}\!(v;\{u_i\}_{i\le \lceil n^{\varepsilon}\rceil};\{w_{i,j}\}_{i\le  \lceil n^{\varepsilon}\rceil,\,j\le \lceil n^{\varepsilon_1}\rceil })
\)
as in Definition~\ref{def:hierarchical-star}, for some $\varepsilon\ge \varepsilon_1>0$. Let $\varepsilon_2=\frac{1}{200}\min\{\rho,\lambda,1\}\varepsilon_1$. Fix any $i_0\le n^{\varepsilon}$. Suppose that $N\in \mathcal{G}(u_{i_0}, \{w_{{i_0},1},\dots, w_{{i_0},\lceil n^{\varepsilon_1}\rceil} \})$ is an event implying that there exist random times $T_\ell$ for $\ell =1,\dots,\lfloor n^{\varepsilon_2/2}\rfloor$ such that each $T_\ell$ is measurable with respect to $\mathcal{G}(u_{i_0}, \{w_{{i_0},1},\dots, w_{{i_0},\lceil n^{\varepsilon_1}\rceil} \})$, $T_{\ell +1}-T_\ell\ge 3$, and $u_{i_0}$ is infected at time $T_\ell$. Then, for any $x\in \{v\}\cup\{u_i\}_{i\le \lceil n^{\varepsilon}\rceil,\, i\ne i_0}$, there exists an event $A_x\in \mathcal{G}(HS(v))$ such that, for all sufficiently large $n$, the following hold:
\begin{enumerate}
    \item[\normalfont (i)] On $A_x\cap N$, there exists a random time $T\le T_{\lfloor n^{\varepsilon_2/2}\rfloor}+3$ such that $T$ is measurable with respect to $\mathcal{G}(HS(v))$, and the vertex $x$ is infected at time $T$;
    \item[\normalfont (ii)] $\mathbb{P}(A_x)\ \ge\ 1-\exp(-n^{\varepsilon_2/3})$.
\end{enumerate}
\end{lemma}

\begin{proof}
We first treat the case $x=v$. For each $\ell=1,\dots,\lfloor n^{\varepsilon_2/2}\rfloor$, define
\begin{eqnarray*}
E_\ell^{(v)}:=\Big\{N^{\mathrm{rec}}_{u_{i_0}}((T_\ell,T_\ell+1])=0\Big\}\ \cap\
\Big\{N^{\mathrm{sus}}_{v}\big((T_\ell,T_\ell+\tfrac{1}{2}]\big)\ge 1\Big\}\ \cap \\ \
\Big\{N^{\mathrm{inf}}_{u_{i_0}\to v}\big((T_\ell+\tfrac{1}{2},T_\ell+1]\big)\ge 1\Big\}\ \cap\
\Big\{N^{\mathrm{rec}}_{v}((T_\ell,T_\ell+1])=0\Big\}.
\end{eqnarray*}
Since $T_{\ell+1}-T_\ell\ge 3$, the time intervals involved in the events $\{E_\ell^{(v)}\}_{\ell=1}^{\lfloor n^{\varepsilon_2/2}\rfloor}$ are disjoint. Hence, by the independent increments of the Poisson processes in the modified Harris construction, the events $(E_\ell^{(v)})_{\ell=1}^{\lfloor n^{\varepsilon_2/2}\rfloor}$ are independent. Moreover, each of them has probability $c_{\mathrm{edge}}$, where $c_{\mathrm{edge}}$ is the constant from Lemma~\ref{lemma_edge_transmission_one}, and each implies that the vertex $v$ is infected at time $T_\ell+1$. Define
$$
A_v:=\bigcup_{\ell=1}^{\lfloor n^{\varepsilon_2/2}\rfloor}E_\ell^{(v)}.
$$
Then $A_v\in\mathcal{G}(HS(v))$. On the event $A_v\cap N$, let
$
\ell_*:=\inf\{\ell\in\{1,\dots,\lfloor n^{\varepsilon_2/2}\rfloor\}: E_\ell^{(v)} \text{ occurs}\}
$.
Since each $E_\ell^{(v)}$ is measurable with respect to $\mathcal{G}(HS(v))$, the index $\ell_*$ is measurable with respect to $\mathcal{G}(HS(v))$. Therefore the random time
$
T:=T_{\ell_*}+1
$
is also measurable with respect to $\mathcal{G}(HS(v))$, satisfies $T\le T_{\lfloor n^{\varepsilon_2/2}\rfloor}+1$, and the vertex $v$ is infected at time $T$. Thus $A_v\cap N$ implies item~(i). Furthermore, for all sufficiently large $n$,
$$
\mathbb{P}(A_v)=1-(1-c_{\mathrm{edge}})^{\lfloor n^{\varepsilon_2/2}\rfloor}\ge 1-\exp(-n^{\varepsilon_2/3}).
$$

Next consider the case $x=u_i$ for some $i\ne i_0$. For each $\ell=1,\dots,\lfloor n^{\varepsilon_2/2}\rfloor$, define
\begin{eqnarray*}
E_\ell^{(u_i)}:= E_{\ell}^{(v)}\cap \Big\{N^{\mathrm{rec}}_{v}((T_\ell+1,T_\ell+2])=0\Big\}\ \cap\
\Big\{N^{\mathrm{sus}}_{u_i}\big((T_\ell+1,T_\ell+\tfrac{3}{2}]\big)\ge 1\Big\}\ \cap \\ \
\Big\{N^{\mathrm{inf}}_{v\to u_i}\big((T_\ell+\tfrac{3}{2},T_\ell+2]\big)\ge 1\Big\}\ \cap\
\Big\{N^{\mathrm{rec}}_{u_i}((T_\ell+1,T_\ell+2])=0\Big\}.
\end{eqnarray*}
Again, since $T_{\ell+1}-T_\ell\ge 3$, these events are independent. Each has probability $c_{\mathrm{edge}}^2$, and each implies that the vertex $u_i$ is infected at time $T_\ell+2$. Define
$$
A_{u_i}:=\bigcup_{\ell=1}^{\lfloor n^{\varepsilon_2/2}\rfloor}E_\ell^{(u_i)}.
$$
Then $A_{u_i}\in\mathcal{G}(HS(v))$. On the event $A_{u_i}\cap N$, let
$
\ell_*:=\inf\{\ell\in\{1,\dots,\lfloor n^{\varepsilon_2/2}\rfloor\}: E_\ell^{(u_i)} \text{ occurs}\}
$.
Since each $E_\ell^{(u_i)}$ is measurable with respect to $\mathcal{G}(HS(v))$, the index $\ell_*$ is measurable with respect to $\mathcal{G}(HS(v))$. Therefore the random time
$
T:=T_{\ell_*}+2
$
is also measurable with respect to $\mathcal{G}(HS(v))$, satisfies $T\le T_{\lfloor n^{\varepsilon_2/2}\rfloor}+2$, and the vertex $u_i$ is infected at time $T$. Thus $A_{u_i}\cap N$ implies item~(i). Moreover, for all sufficiently large $n$,
$$
\mathbb{P}(A_{u_i})=1-(1-c_{\mathrm{edge}}^2)^{\lfloor n^{\varepsilon_2/2}\rfloor}\ge 1-\exp(-n^{\varepsilon_2/3}).
$$

\end{proof}

Lemma~\ref{lemma_hierarchicalstar_supermartingale} introduces an embedded process $(\tilde W_r)_{r\in\mathbb{N}}$, 
which tracks the number of first-layer vertices $u_i$ that are $n^{\varepsilon_1/24}$-lit among $\{w_{i,1},\dots,w_{i,\lceil n^{\varepsilon_1}\rceil}\}$ at time $t_r:=3rn^{\varepsilon_2}$. The key observation is that, by Proposition~\ref{prop:star_independent}, if a given $u_i$ is $n^{\varepsilon_1/24}$-lit at time $t$, then it remains $n^{\varepsilon_1/24}$-lit at time $t+3n^{\varepsilon_2}$ with probability at least $1-n^{-\varepsilon_2}$, and these good events can be realized on disjoint local graphical data. If $\tilde{W}_r\le n^{\varepsilon_3}$, with some $\varepsilon_3<\varepsilon_2$, then with probability at least $1-n^{-(\varepsilon_2-\varepsilon_3)}$, all currently lit first-layer vertices remain lit over one epoch. On the other hand, Lemma~\ref{lemma_edge_transmission_k} implies that the remaining first-layer vertices are infected during the same time interval with high probability, and hence become $n^{\varepsilon_1/24}$-lit. Combining these two effects yields that $\mathbb{E}[\exp(-\tilde W_{r+1})\mid\mathcal{F}_{3n^{\varepsilon_2}r} ]\le \exp(-\tilde{W}_r)$  whenever $\tilde{W}_r\in [\frac14 n^{\varepsilon_3}+1,\,\lceil n^{\varepsilon_3}\rceil]$. 

It is worth emphasizing that if we denote by $(W_t)_{t\ge 0}$ the number of vertices $u_i$ that are $n^{\varepsilon_1/24}$-lit among $\{w_{i,1},\dots,w_{i,\lceil n^{\varepsilon_1}\rceil}\}$ at time $t$, then $W_t, W_{t_r}$ are not measurable with respect to the local graphical $\sigma$-field $\mathcal G(HS(v))$. The reason is that the event that a given $u_i$ is lit among $\{w_{i,1},\dots,w_{i,n^{\varepsilon_1}}\}$ at time $t$ may depend on infection paths entering the hierarchical-star from outside. In Lemma~\ref{lemma_hierarchicalstar_supermartingale}, we therefore construct instead an embedded comparison process $\tilde W_r$ that is measurable with respect to $\mathcal G(HS(v))$ and that always provides a lower bound for $W_{t_r}$.

\begin{lemma}\label{lemma_hierarchicalstar_supermartingale}
Let $G=(V,E)$ be a graph containing 
\(
HS(v):=\mathrm{HStar}\!(v;\{u_i\}_{i\le \lceil n^{\varepsilon}\rceil};\{w_{i,j}\}_{i\le \lceil n^{\varepsilon}\rceil,\,j\le \lceil n^{\varepsilon_1}\rceil})
\)
as in Definition~\ref{def:hierarchical-star}, for some $\varepsilon\ge \varepsilon_1>0$. Let the initial configuration $\boldsymbol{\sigma}(0)$
be arbitrary on $V$, subject only to the condition that $v$ is $(k,n^{\varepsilon_1/24})$-HS-lit with respect to $HS(v)$
for some
\(
 k\in[\tfrac12 n^{\varepsilon_3}+1,\, n^{\varepsilon_3}], \ \varepsilon_2=\tfrac{1}{200}\min\{\rho,\lambda,1\}\varepsilon_1, \ \varepsilon_3=\tfrac14\varepsilon_2.
\)
Then, for all sufficiently large $n$, there exists an event $C^{(v)}\in\mathcal{G}(HS(v))$ such that
\[
\mathbb{P}\big(C^{(v)}\big)\ \ge\ 1-\exp(1-n^{\varepsilon_3}/4),
\]
and on $C^{(v)}$, one of the following two alternatives holds:
\begin{enumerate}[\normalfont (i)]
    \item the vertex $v$ becomes $(n^{\varepsilon_3},n^{\varepsilon_1/24})$-HS-lit with respect to $HS(v)$ at some time before it ceases to be $(\tfrac14 n^{\varepsilon_3},n^{\varepsilon_2})$-HS-lit with respect to $HS(v)$;
    \item the vertex $v$ never ceases to be $(\tfrac14 n^{\varepsilon_3},n^{\varepsilon_2})$-HS-lit with respect to $HS(v)$.
\end{enumerate}
\end{lemma}

\begin{proof}
Let $(\mathcal{F}_t)_{t\ge 0}$ be the natural filtration of the modified Harris construction, as defined in Subsection~\ref{ssec:harris}. For each $r\in\mathbb{N}$, set $t_r:=3rn^{\varepsilon_2}$, and let $\bar{\mathcal{F}}_r:=\sigma(\mathcal{H}|_{[0,t_r]})$ be the $\sigma$-field generated by the restriction of the modified Harris construction to the time interval $[0,t_r]$. Define
\[
L_r:=\Big\{u_i\in \{u_1,\dots, u_{\lceil n^{\varepsilon}\rceil}\}:\ u_i \text{ is $n^{\varepsilon_1/24}$-lit among }\{w_{i,1},\dots, w_{i,\lceil n^{\varepsilon_1}\rceil}\}\text{ at time }t_r\Big\}, \qquad W_r:=|L_r|.
\]
Then $W_0\ge k$.

Suppose that we can construct a process $(\tilde W_r)_{r\in\mathbb{N}}$ adapted to the filtration $(\bar{\mathcal{F}}_r)_{r\in \mathbb{N}}$ such that:
\begin{enumerate}
\item[(a)] $\tilde W_r$ is measurable with respect to $\mathcal{G}(HS(v))$;
\item[(b)] For every $r$ before the first time at which $\tilde W_r<\frac14 n^{\varepsilon_3}+1$, $\tilde W_r\le W_r$, and $v$ remains $(\tfrac14 n^{\varepsilon_3},n^{\varepsilon_2})$-HS-lit with respect to $HS(v)$ throughout time interval $[0, t_{r}]$;
\item[(c)] For every $r\ge0$, on the event
\(
\tilde W_r\in [\tfrac14 n^{\varepsilon_3}+1,\,\lceil n^{\varepsilon_3}\rceil]
\), 
we have
\[
\mathbb E\!\left[\exp(-\tilde W_{r+1})\mid \bar{\mathcal F}_r\right]
\le
\exp(-\tilde W_r).
\]
\end{enumerate}
Assuming such a process has been constructed, we first explain why it yields the conclusion of the lemma. Define
\[
\tau_*:=\inf\left\{r\ge0:\tilde W_r\notin
\left[\tfrac14 n^{\varepsilon_3}+1,\lceil n^{\varepsilon_3}\rceil\right]\right\},
\]
and let
\[
D_-:=\left\{\tau_*<\infty,\ \tilde W_{\tau_*}<\tfrac14 n^{\varepsilon_3}+1\right\}, \qquad C^{(v)}:=D_-^{\mathrm c}.
\]
Since each $\tilde W_r$ is measurable with respect to $\mathcal G(HS(v))$, the event $C^{(v)}$ also belongs to $\mathcal G(HS(v))$.

On $C^{(v)}$, the process does not exit the interval
\(
\left[\tfrac14 n^{\varepsilon_3}+1,\lceil n^{\varepsilon_3}\rceil\right]
\) through the lower endpoint. Hence either $\tau_*<\infty$ and the exit is through the upper endpoint, or $\tau_*=\infty$. In the first case, $W_{\tau_*}\ge \tilde W_{\tau_*}>\lceil n^{\varepsilon_3}\rceil$, and therefore $v$ becomes $(n^{\varepsilon_3},n^{\varepsilon_1/24})$-HS-lit before it ceases to be $(\tfrac14 n^{\varepsilon_3},n^{\varepsilon_2})$-HS-lit. In the second case, $\tilde W_r\ge \tfrac14 n^{\varepsilon_3}+1$ for all $r\ge0$, and hence part~(b) implies that $v$ never ceases to be $(\tfrac14 n^{\varepsilon_3},n^{\varepsilon_2})$-HS-lit. Thus $C^{(v)}$ implies one of the two alternatives in the statement of the lemma. The probability estimate in the statement will follow from part~(c), together with a standard optional-stopping argument for the stopped process $\exp(-\tilde W_{r\wedge\tau_*})$, as shown at the end of the proof.

We now proceed with the construction.

Define $\tilde{L}_0$ to be any subset of $L_0$ such that $|\tilde{L}_0|=\lfloor k\rfloor$, and set $\tilde{W}_0=|\tilde{L}_0|=\lfloor k\rfloor$. We then define $\tilde{L}_{r+1}$ inductively, given $\tilde{L}_r$, and set $\tilde{W}_{r+1}=|\tilde{L}_{r+1}|$.

Fix an epoch index $r\in\mathbb{N}$, and assume that $\tilde{L}_r$ has already been defined. For each $u_i\in \tilde{L}_r$, let $N_r^{(i)}$ be the event that $u_i$ is not $n^{\varepsilon_1/24}$-lit among $\{w_{i,1},\dots, w_{i,\lceil n^{\varepsilon_1}\rceil}\}$ at time $t_{r+1}$.

By Proposition~\ref{prop:star_independent}, for each $u_i\in \tilde{L}_r$ there exists an event
\(
\tilde{A}_r^{(u_i)}\in \mathcal{G}\big(u_i,\{w_{i,1},\dots, w_{i,\lceil n^{\varepsilon_1}\rceil}\}\big)
\)
such that, for all sufficiently large $n$, the following hold:
\begin{enumerate}
  \item[\normalfont (i)] On $\tilde{A}_r^{(u_i)}$, the vertex $u_i$ remains $n^{\varepsilon_1/96}$-lit among $\{w_{i,1},\dots, w_{i,\lceil n^{\varepsilon_1}\rceil}\}$ throughout the time interval $\big[t_r,t_r+\tfrac{\varepsilon_1}{48}\log n\big]$;
  \item[\normalfont (ii)] On $\tilde{A}_r^{(u_i)}$, the vertex $u_i$ remains $n^{\varepsilon_1/24}$-lit among $\{w_{i,1},\dots, w_{i,\lceil n^{\varepsilon_1}\rceil}\}$ throughout the time interval $[t_r+\frac{\varepsilon_1}{48}\log n,t_r+3n^{\varepsilon_2}]$;
  \item[\normalfont (iii)] On $\tilde{A}_r^{(u_i)}$, there exist random times $T_{\ell}^{(u_i)}\in [t_r,t_r+n^{\varepsilon_2}]$, $\ell=1,\dots,\lfloor n^{\varepsilon_2/2}\rfloor$, such that each $T_{\ell}^{(u_i)}$ is measurable with respect to $\mathcal{G}(u_i,\{w_{i,1},\dots, w_{i,\lceil n^{\varepsilon_1}\rceil}\})$, $T_{\ell+1}^{(u_i)}-T_{\ell}^{(u_i)}\ge 3$, and $u_i$ is infected at time $T_{\ell}^{(u_i)}$ for every $\ell=1,\dots,\lfloor n^{\varepsilon_2/2}\rfloor$;
  \item[\normalfont (iv)] $\mathbb{P}(\tilde{A}_r^{(u_i)}\mid \bar{\mathcal{F}}_r) \ge 1 - n^{-\varepsilon_2}$.
\end{enumerate}
Thus,
\[
N_r^{(i)}\subseteq (\tilde{A}_r^{(u_i)})^{\mathrm{c}}.
\]
Let
\[
N_r=\sum_{i\in \tilde{L}_r}\mathbf{1}\{N_r^{(i)}\},
\qquad
\tilde{N}_r=\sum_{i\in \tilde{L}_r}\mathbf{1}\{(\tilde{A}_r^{(u_i)})^{\mathrm{c}}\}.
\]
For $i\ne j$, the $\sigma$-fields $\mathcal{G}\big(u_i,\{w_{i,1},\dots, w_{i,\lceil n^{\varepsilon_1}\rceil}\}\big)$ and $\mathcal{G}\big(u_j,\{w_{j,1},\dots, w_{j,\lceil n^{\varepsilon_1}\rceil}\}\big)$ are independent. Hence the events $(\tilde A_r^{(u_i)})^{\mathrm{c}}$ and $(\tilde A_r^{(u_j)})^{\mathrm{c}}$ are conditionally independent given $\bar{\mathcal{F}}_r$. Therefore,
\[
N_r\le \tilde N_r \le_{\mathrm{st}} \mathrm{Binomial}(\tilde W_r,\,n^{-\varepsilon_2}).
\]
In particular, on the event $\{ \tilde{W}_r\le \lceil n^{\varepsilon_3}\rceil\}$, since $1+x\le \exp(x)$ for $x\ge 0$ and $\varepsilon_3=\frac{1}{4}\varepsilon_2$,
\begin{equation}\label{eq:supermartingaleNr}
  \mathbb{E}\!\left[e^{\tilde N_r}\mid \bar{\mathcal{F}}_r\right]
\le \big(1-n^{-\varepsilon_2}+e\,n^{-\varepsilon_2}\big)^{ \tilde{W}_r}
\le \exp\!\big((e-1)\lceil n^{\varepsilon_3}\rceil n^{-\varepsilon_2}\big)
\le 1+o(1),
\end{equation}
and
\begin{equation}\label{eq:supermartingaleNr2}
    \mathbb{P}(\tilde{N}_r=0\mid \bar{\mathcal{F}}_r) \ge (1-n^{-\varepsilon_2})^{\lceil n^{\varepsilon_3}\rceil}\ge 1-{\lceil n^{\varepsilon_3}\rceil}n^{-\varepsilon_2}\ge 1-o(1).
\end{equation}

Whenever $\tilde{W}_r\le \lceil n^{\varepsilon_3}\rceil$,  choose a vertex $u_j\notin \tilde{L}_r$, according to a fixed deterministic rule. By Lemma~\ref{lemma_edge_transmission_k}, there exists an event $\tilde{B}_r^{(1)}\in\mathcal{G}(HS(v))$ such that, for all sufficiently large $n$, the following hold:
\begin{enumerate}
    \item[\normalfont (i)] On $\tilde{B}_r^{(1)}\cap \{\tilde{N}_r=0\} \cap \{\tilde W_r\ge \frac{1}{4}n^{\varepsilon_3}+1\}$, there exists a random time $T_r\in [t_r,t_r+2n^{\varepsilon_2}]$ such that $T_r$ is measurable with respect to $\mathcal{G}(HS(v))$, and the vertex $u_j$ is infected at time $T_r$;
    \item[\normalfont (ii)]
    \begin{equation}\label{eq:supermartingaleAr1}
    \mathbb{P}(\tilde{B}_r^{(1)}\mid\bar{\mathcal{F}}_r)\ \ge\ 1-\exp(-n^{\varepsilon_2/3})\ge 1-o(1).
\end{equation}
\end{enumerate}

Since $T_r$ is measurable with respect to $\mathcal{G}(HS(v))$ and $u_j$ is infected at time $T_r$, Proposition~\ref{prop:star_independent}, applied to the star $(u_j,\{w_{j,1},\dots,w_{j,\lceil n^{\varepsilon_1}\rceil}\})$ at time $T_r$, yields an event $\tilde{B}_r^{(2)}\in \mathcal{G}(HS(v))$ such that
\(
\tilde{B}_r^{(2)} \cap \tilde{B}_r^{(1)} \cap \{\tilde{N}_r=0\} \cap \{\tilde W_r\ge \frac{1}{4}n^{\varepsilon_3}+1\}
\)
implies that $u_j$ is $n^{\varepsilon_1/24}$-lit among $\{w_{j,1},\dots,w_{j,\lceil n^{\varepsilon_1}\rceil}\}$ at time $t_{r+1}\in [T_r+\frac{\varepsilon_1}{48}\log n, T_r+3n^{\varepsilon_2}]$, and
\begin{equation}\label{eq:supermartingaleAr2}
    \mathbb{P}\big(\tilde B_r^{(2)}\mid {\mathcal{F}}_{T_r}\big)\ \ge \ 1-n^{-\varepsilon_2}\ge 1-o(1).
\end{equation}

Define
\[
\tilde L_{r+1}:=
\begin{cases}
 \tilde{L}_r \backslash \bigcup_{u_i\in \tilde{L}_r: (\tilde A_r^{(u_i)})^{\mathrm{c}}\text{ occurs}}\{u_i\} , & \text{if } \{\tilde N_r>0\} \text{ occurs},\\
 \tilde{L}_{r}\cup \{u_j\} , & \text{if } \tilde B_r^{(1)}\cap \tilde{B}_r^{(2)} \cap \{\tilde{N}_r=0 \}  \text{ occurs},\\
 \tilde L_r, & \text{otherwise}.
\end{cases}\ .
\]
We claim that if $\tilde L_r\subseteq L_r$ and $\tilde W_r\ge \frac{1}{4}n^{\varepsilon_3}+1$, then $\tilde L_{r+1}\subseteq L_{r+1}$.
\begin{itemize}
    \item If $\tilde N_r>0$ occurs, then for any $u_i\in \tilde{L}_r$ such that $(\tilde A_r^{(u_i)})^{\mathrm{c}}$ does not occur, the vertex $u_i$ is $n^{\varepsilon_1/24}$-lit among $\{w_{i,1},\dots, w_{i,\lceil n^{\varepsilon_1}\rceil}\}$ at time $t_{r+1}$, and hence $u_i\in L_{r+1}$. Therefore $\tilde{L}_{r+1}\subseteq L_{r+1}$.
    \item If $\tilde B_r^{(1)}\cap \tilde{B}_r^{(2)} \cap \{\tilde{N}_r=0 \}$ occurs, then $u_j$ is $n^{\varepsilon_1/24}$-lit among $\{w_{j,1},\dots,w_{j,\lceil n^{\varepsilon_1}\rceil}\}$ at time $t_{r+1}$, and hence $u_j\in L_{r+1}$. Since also $\tilde N_r=0$, every vertex of $\tilde L_r$ remains in $L_{r+1}$. Therefore $\tilde{L}_{r+1}\subseteq L_{r+1}$.
    \item Otherwise, we have $\tilde N_r=0$ and $\tilde B_r^{(1)}\cap \tilde{B}_r^{(2)}$ does not occur. In this case, every vertex $u_i\in \tilde L_r$ remains $n^{\varepsilon_1/24}$-lit among $\{w_{i,1},\dots,w_{i,\lceil n^{\varepsilon_1}\rceil}\}$ at time $t_{r+1}$, and hence belongs to $L_{r+1}$. Therefore $\tilde L_{r+1}=\tilde L_r\subseteq L_{r+1}$.
\end{itemize}
Since $\tilde L_0\subseteq L_0$, it follows by induction that $\tilde{L}_r\subseteq L_r$ for every $r$ before the first time at which $\tilde W_r<\frac14 n^{\varepsilon_3}+1$.

Now let
\[
\tilde W_{r+1}:=|\tilde{L}_{r+1}|=
\begin{cases}
 \tilde{W}_r -\tilde{N}_r , & \text{if } \{\tilde N_r>0\} \text{ occurs},\\
 \tilde{W}_{r}+1 , & \text{if } \tilde B_r^{(1)}\cap \tilde{B}_r^{(2)} \cap \{\tilde{N}_r=0 \}  \text{ occurs},\\
\tilde W_r, & \text{otherwise}.
\end{cases}\ .
\]
Then $\tilde{W}_r$ is measurable with respect to $\mathcal{G}(HS(v))$, establishing part~(a).

Moreover, $\tilde{W}_r\le W_r$ for every $r$ before the first time at which $\tilde W_r<\frac14 n^{\varepsilon_3}+1$. Furthermore, for each $u_i\in \tilde{L}_r$, the occurrence of $\tilde A_r^{(u_i)}$ implies that $u_i$ is $n^{\varepsilon_2}$-lit among $\{w_{i,1},\dots,w_{i,\lceil n^{\varepsilon_1}\rceil}\}$ throughout the time interval $[t_r,t_{r+1}]$, since $n^{\varepsilon_1/96}>n^{\varepsilon_2}$. We claim that for every $r$ before the first time at which $\tilde W_r<\frac14 n^{\varepsilon_3}+1$, the vertex $v$ remains $(\frac{1}{4}n^{\varepsilon_3},n^{\varepsilon_2})$-HS-lit with respect to $HS(v)$ throughout the time interval $[0,t_{r}]$, establishing part~(b). We prove this by induction on $r$. The case $r=0$ is trivial. Suppose the claim holds up to time $t_{r-1}$. It remains to prove it on $[t_{r-1},t_r]$.
\begin{itemize}
    \item If $\tilde{W}_{r-1}\ge \tilde{W}_{r}$, then all vertices $u_i\in \tilde{L}_{r}$ satisfy $\tilde{A}_{r-1}^{(u_i)}$, and hence $v$ remains $(\tilde{W}_{r},n^{\varepsilon_2})$-HS-lit with respect to $HS(v)$ throughout the time interval $[t_{r-1},t_{r}]$.
    \item If $\tilde{W}_{r-1}=\tilde{W}_{r}-1$, then all vertices $u_i\in \tilde{L}_{r-1}$ satisfy $\tilde{A}_{r-1}^{(u_i)}$, and hence $v$ remains $(\tilde{W}_{r-1},n^{\varepsilon_2})$-HS-lit with respect to $HS(v)$ throughout the time interval $[t_{r-1},t_{r}]$.
\end{itemize}

It remains to prove part~(c). Let $X_r:=\tilde{W}_{r+1}-\tilde{W}_r$. Conditioning on $\bar{\mathcal{F}}_r$, on the event
\(
\tilde W_r\in [\tfrac14 n^{\varepsilon_3}+1,\,\lceil n^{\varepsilon_3}\rceil]
\), we obtain
\begin{align*}
\mathbb{E}\!\left[e^{-X_r}\mid \bar{\mathcal{F}}_r\right]
&\le \sum_{m=1}^{\tilde W_r} e^{m}\,\mathbb{P}(\tilde{N}_r=m\mid \bar{\mathcal{F}}_r)
  +1\cdot \mathbb{P}((\tilde B_r^{(1)}\cap \tilde{B}_r^{(2)} \cap \{\tilde{N}_r=0 \})^{\mathsf c}\mid \bar{\mathcal{F}}_r) \\
&\qquad
  +e^{-1}\cdot \mathbb{P}(\tilde B_r^{(1)}\cap \tilde{B}_r^{(2)} \cap \{\tilde{N}_r=0 \}\mid \bar{\mathcal{F}}_r)\\
&=\mathbb{E}\!\left[e^{\tilde{N}_r}\mid \bar{\mathcal{F}}_r\right]
  -(1-e^{-1})\,\mathbb{P}(\tilde B_r^{(1)}\cap \tilde{B}_r^{(2)} \cap \{\tilde{N}_r=0 \}\mid \bar{\mathcal{F}}_r).
\end{align*}
Together with \eqref{eq:supermartingaleNr}, \eqref{eq:supermartingaleNr2}, \eqref{eq:supermartingaleAr1}, and \eqref{eq:supermartingaleAr2}, this shows that, for all sufficiently large $n$, on the event $\tilde W_r\in [\tfrac14 n^{\varepsilon_3}+1,\,\lceil n^{\varepsilon_3}\rceil]$
\[
\mathbb{E}\!\left[e^{-X_r}\mid \bar{\mathcal{F}}_r\right] < 1,
\]
This proves part~(c).

Define
\(
\tau_*:=\inf\{r\ge0:\tilde W_r\notin [\tfrac14 n^{\varepsilon_3}+1,\,\lceil n^{\varepsilon_3}\rceil]\},  D_-:=\{\tau_*<\infty,\ \tilde W_{\tau_*}<\tfrac14 n^{\varepsilon_3}+1\}, C^{(v)}:=D_-^{\mathrm c}
\). Then
\(
M_r:=\exp(-\tilde W_{r\wedge\tau_*})
\)
satisfies
\(
\mathbb E[M_{r+1}\mid \bar{\mathcal F}_r]\le M_r
\).
Indeed, on $\{r<\tau_*\}$ this follows from the above estimate, while on $\{r\ge\tau_*\}$ the process is constant. Hence, for every $R\ge1$,
\[
\mathbb E[M_R]\le M_0=\exp(-\lfloor k\rfloor).
\]
On the event $D_-\cap\{\tau_*\le R\}$, since $\tilde W_r$ is integer-valued and $\tilde W_{\tau_*}<\tfrac14 n^{\varepsilon_3}+1$, we have
\(
\tilde W_{\tau_*}\le \tfrac14 n^{\varepsilon_3}
\).
Therefore,
\[
M_R
=
\exp(-\tilde W_{\tau_*})
\ge
\exp(-n^{\varepsilon_3}/4)
\qquad\text{on }D_-\cap\{\tau_*\le R\}.
\]
It follows that
\[
\exp(-\lfloor k\rfloor)
\ge
\mathbb E[M_R]
\ge
\exp(-n^{\varepsilon_3}/4)\,
\mathbb P(D_-\cap\{\tau_*\le R\}).
\]
Letting $R\to\infty$ gives
\[
\mathbb P(D_-)
\le
\exp\!\left(\tfrac14 n^{\varepsilon_3}-\lfloor k\rfloor\right)
\le
\exp\!\left(1-\tfrac14 n^{\varepsilon_3}\right),\qquad \mathbb{P}(C^{(v)})\ge 1-\exp\!\left(1-\tfrac14 n^{\varepsilon_3}\right)
\]
where we used $k\ge \tfrac12 n^{\varepsilon_3}+1$.

\end{proof}

\begin{lemma}\label{lemma_hierarchicalstar_expsurvivaltime}
Let $G=(V,E)$ be a graph containing
\(
HS(v):=\mathrm{HStar}\!(v;\{u_i\}_{i\le \lceil n^{\varepsilon}\rceil};\{w_{i,j}\}_{i\le \lceil n^{\varepsilon}\rceil,\,j\le \lceil n^{\varepsilon_1}\rceil})
\)
as in Definition~\ref{def:hierarchical-star}, for some $\varepsilon\ge \varepsilon_1>0$. Let the initial configuration $\boldsymbol{\sigma}(0)$
be arbitrary on $V$, subject only to the condition that $v$ is $(n^{\varepsilon_3},n^{\varepsilon_1/24})$-HS-lit with respect to $HS(v)$
where
\(
\varepsilon_2=\tfrac{1}{200}\min\{\rho,\lambda,1\}\varepsilon_1, \ \varepsilon_3=\tfrac14\varepsilon_2.
\)
Then, for all sufficiently large $n$, there exists an event $\tilde{C}^{(v)}\in \mathcal{G}(HS(v))$ such that
\begin{enumerate}[\normalfont (i)]
    \item on $\tilde{C}^{(v)}$, the vertex $v$ remains $(\tfrac14 n^{\varepsilon_3},\,n^{\varepsilon_2})$-HS-lit with respect to $HS(v)$ throughout the time interval $[0, \exp(n^{\varepsilon_3/2})]$;
    \item $\mathbb{P}(\tilde{C}^{(v)})\ge 1-\exp(-n^{\varepsilon_3/2})$.
\end{enumerate}
\end{lemma}

\begin{proof}
The proof is based on iterating the one-step mechanism from Lemma~\ref{lemma_hierarchicalstar_supermartingale}. Starting from a configuration in which $v$ is $(n^{\varepsilon_3},n^{\varepsilon_1/24})$-HS-lit, we first choose a subset of size $\lceil\frac12 n^{\varepsilon_3}+1\rceil$ among the currently good first-layer vertices according to a fixed deterministic rule, and then run the comparison process $\tilde W_r$ constructed in Lemma~\ref{lemma_hierarchicalstar_supermartingale} on this subset. With high probability, the process $\tilde W_r$ does not exit the interval $[\tfrac14 n^{\varepsilon_3}+1,\lceil n^{\varepsilon_3}\rceil]$ through its lower endpoint. If it exits through the upper endpoint, then $v$ becomes $(n^{\varepsilon_3},n^{\varepsilon_1/24})$-HS-lit before it ceases to be $(\tfrac14 n^{\varepsilon_3},n^{\varepsilon_2})$-HS-lit, and we restart the same construction from that time. If it never exits the interval, then $v$ never ceases to be $(\tfrac14 n^{\varepsilon_3},n^{\varepsilon_2})$-HS-lit, and the desired conclusion follows immediately.

For $r\in \mathbb{N}$, let $L_{t_r}$ denote the set of vertices $u_i$ that are $n^{\varepsilon_1/24}$-lit among $\{w_{i,1},\dots, w_{i,\lceil n^{\varepsilon_1}\rceil}\}$ at time $t_r:=3rn^{\varepsilon_2}$, and let $W_{t_r}:=|L_{t_r}|$. Then $W_0\ge \lceil n^{\varepsilon_3}\rceil$.

Choose arbitrarily a subset $\tilde{L}^{(1)}_0\subseteq L_0$ such that $|\tilde{L}^{(1)}_0|=\lceil\frac{1}{2} n^{\varepsilon_3}\rceil+1$, according to a fixed deterministic rule, and let $\tilde{W}^{(1)}_0:=|\tilde{L}^{(1)}_0|$. By the same inductive construction as in Lemma~\ref{lemma_hierarchicalstar_supermartingale}, we obtain a process $\tilde{W}^{(1)}_r$ measurable with respect to $\mathcal{G}(HS(v))$. Let
$$
\tau^{(1)}:=\inf\{r\ge 0:\tilde{W}^{(1)}_r\notin [\tfrac{1}{4}n^{\varepsilon_3}+1,\,\lceil n^{\varepsilon_3}\rceil]\}.
$$
Then $\tau^{(1)}$ is also measurable with respect to $\mathcal{G}(HS(v))$. Define
$
D_-^{(1)}:=\{\tau^{(1)}<\infty,\tilde W^{(1)}_{\tau^{(1)}}<\tfrac14 n^{\varepsilon_3}+1\}
$
and
$
C^{(1)}:=(D_-^{(1)})^{\mathrm c}.
$
By Lemma~\ref{lemma_hierarchicalstar_supermartingale},
$$
\mathbb{P}(C^{(1)})\ge 1-\exp\!\Big(1-\frac{1}{4}n^{\varepsilon_3}\Big).
$$
On $C^{(1)}$, either $\tau^{(1)}=\infty$, in which case $v$ never ceases to be $(\tfrac14 n^{\varepsilon_3},n^{\varepsilon_2})$-HS-lit and the desired conclusion follows, or $\tau^{(1)}<\infty$ and the exit is through the upper endpoint. In the latter case,
$$
W_{t_{\tau^{(1)}}}\ge \tilde{W}^{(1)}_{\tau^{(1)}}>\lceil n^{\varepsilon_3}\rceil,
$$
and $v$ remains $(\tfrac{1}{4}n^{\varepsilon_3},\,n^{\varepsilon_2})$-HS-lit with respect to $HS(v)$ throughout the time interval $[0,t_{\tau^{(1)}}]$.

Let $M:=\lceil\exp(n^{\varepsilon_3/2})\rceil$. For each $i=1,\dots,M$, we define $\tilde{L}^{(i)}$, $\tilde{W}^{(i)}$, $\tau^{(i)}$, $D_-^{(i)}$, and $C^{(i)}$ inductively, stopping the construction if some $\tau^{(i)}=\infty$.

Assume that $\tilde{L}^{(j)}$, $\tilde{W}^{(j)}$, $\tau^{(j)}$, $D_-^{(j)}$, and $C^{(j)}$ have been defined for all $j=1,\dots,i$, and assume that $\cap_{j=1}^i C^{(j)}$ occurs and $\tau^{(j)}<\infty$ for all $j=1,\dots,i$. Then the preceding paragraph implies
$$
W_{t_{\sum_{j=1}^i\tau^{(j)}}}\ge \lceil n^{\varepsilon_3}\rceil.
$$
Choose $\tilde{L}^{(i+1)}_0$ to be a subset of $L_{t_{\sum_{j=1}^i\tau^{(j)}}}$ such that $|\tilde{L}^{(i+1)}_0|=\lceil \frac{1}{2}n^{\varepsilon_3}\rceil+1$, according to a fixed deterministic rule, and let $\tilde{W}^{(i+1)}_0:=|\tilde{L}^{(i+1)}_0|$. Define $\tilde{L}^{(i+1)}_r$ and $\tilde{W}^{(i+1)}_r$ by the same inductive construction as in Lemma~\ref{lemma_hierarchicalstar_supermartingale}. Since $\tau^{(1)},\dots,\tau^{(i)}$ are measurable with respect to $\mathcal{G}(HS(v))$, it follows that $\tilde{W}^{(i+1)}_r$ is also measurable with respect to $\mathcal{G}(HS(v))$.

Next let
$$
\tau^{(i+1)}:=\inf\{r\ge 0:\tilde{W}^{(i+1)}_r\notin [\tfrac{1}{4}n^{\varepsilon_3}+1,\,\lceil n^{\varepsilon_3}\rceil]\}.
$$
Then $\tau^{(i+1)}$ is measurable with respect to $\mathcal{G}(HS(v))$. Define
$
D_-^{(i+1)}:=\{\tau^{(i+1)}<\infty,\tilde W^{(i+1)}_{\tau^{(i+1)}}<\tfrac14 n^{\varepsilon_3}+1\}
$
and
$
C^{(i+1)}:=(D_-^{(i+1)})^{\mathrm c}.
$
On $\cap_{j=1}^{i+1}C^{(j)}$, if $\tau^{(i+1)}=\infty$, then $v$ never ceases to be $(\tfrac14 n^{\varepsilon_3},n^{\varepsilon_2})$-HS-lit from time $t_{\sum_{j=1}^i\tau^{(j)}}$ onward, and the desired conclusion follows. If $\tau^{(i+1)}<\infty$, then the exit is through the upper endpoint, and hence
$$
W_{t_{\sum_{j=1}^{i+1}\tau^{(j)}}}\ge \tilde{W}^{(i+1)}_{\tau^{(i+1)}}>\lceil n^{\varepsilon_3}\rceil,
$$
and $v$ remains $(\tfrac{1}{4}n^{\varepsilon_3},\,n^{\varepsilon_2})$-HS-lit with respect to $HS(v)$ throughout the time interval
$
[t_{\sum_{j=1}^{i}\tau^{(j)}},\,t_{\sum_{j=1}^{i+1}\tau^{(j)}}]
$.
Moreover, on $\bigcap_{j=1}^i C^{(j)}\cap\{\tau^{(1)}<\infty,\dots,\tau^{(i)}<\infty\}$,
$$
\mathbb{P}(C^{(i+1)}\mid \mathcal{F}_{t_{\sum_{j=1}^i \tau^{(j)}}})\ge 1-\exp\!\Big(1-\frac{1}{4}n^{\varepsilon_3}\Big).
$$

Now define
$$
\tilde{C}^{(v)}:=\bigcap_{i=1}^M C^{(i)}.
$$
Then $\tilde{C}^{(v)}$ is measurable with respect to $\mathcal{G}(HS(v))$. Furthermore, on $\tilde{C}^{(v)}$, either some $\tau^{(i)}=\infty$, in which case $v$ never ceases to be $(\tfrac14 n^{\varepsilon_3},\,n^{\varepsilon_2})$-HS-lit after the beginning of the $i$th trial, or else $\tau^{(i)}<\infty$ for all $i=1,\dots,M$, in which case $v$ remains $(\tfrac{1}{4}n^{\varepsilon_3},\,n^{\varepsilon_2})$-HS-lit with respect to $HS(v)$ throughout the time interval
$
[0,\,t_{\sum_{j=1}^{M}\tau^{(j)}}]
$.
In either case, using the conditional bounds above and a union bound, we obtain, for all sufficiently large $n$,
$$
\mathbb{P}(\tilde{C}^{(v)})\ge 1-M\exp\!\Big(1-\frac{1}{4}n^{\varepsilon_3}\Big)\ge 1-\exp(-n^{\varepsilon_3/2}).
$$

Finally, if $\tau^{(i)}<\infty$ for all $i=1,\dots,M$, then $\tau^{(i)}\ge 1$ for each $i$, since $\tilde{W}^{(i)}_0=\lceil \frac{1}{2}n^{\varepsilon_3}+1\rceil\in [\frac{1}{4}n^{\varepsilon_3}+1,\,\lceil n^{\varepsilon_3}\rceil]$. Therefore,
$$
t_{\sum_{j=1}^{M}\tau^{(j)}}\ge t_M=3Mn^{\varepsilon_2}\ge M\ge \exp(n^{\varepsilon_3/2}).
$$
We conclude that, on $\tilde{C}^{(v)}$, the vertex $v$ remains $(\tfrac{1}{4}n^{\varepsilon_3},\,n^{\varepsilon_2})$-HS-lit with respect to $HS(v)$ throughout the time interval $[0,\exp(n^{\varepsilon_3/2})]$.
\end{proof}

\medskip
Lemma~\ref{lemma_hierarchicalstar_expsurvivaltime} treats the stronger initial condition in which the center is already $(n^{\varepsilon_3},n^{\varepsilon_1/24})$-HS-lit, and under this assumption it yields survival for a time of order $\exp(n^{\varepsilon_3/2})$. By contrast, Proposition~\ref{hierarchicalstar_thm2} assumes only the weaker $(k,n^{\varepsilon_2})$-HS-lit condition, where $k\in[\frac14 n^{\varepsilon_3},\,n^{\varepsilon_3})$. Thus, in order to apply Lemma~\ref{lemma_hierarchicalstar_expsurvivaltime}, it remains to show that, starting from this weaker level, the process reaches the stronger $(n^{\varepsilon_3},n^{\varepsilon_1/24})$-HS-lit configuration with high probability.

Lemma~\ref{lemma_hierarchicalstar_infectedtoHSlit} below provides precisely this upgrade mechanism.

\begin{lemma}\label{lemma_hierarchicalstar_infectedtoHSlit}
Let $G=(V,E)$ be a graph containing 
$
HS(v):=\mathrm{HStar}\!(v;\{u_i\}_{i\le \lceil n^{\varepsilon}\rceil};\{w_{i,j}\}_{i\le \lceil n^{\varepsilon}\rceil,\,j\le \lceil n^{\varepsilon_1}\rceil})
$
as in Definition~\ref{def:hierarchical-star}, for some $\varepsilon\ge \varepsilon_1>0$.
If $\sigma_v(0)=I$, then for all sufficiently large $n$, there exists an event $C^{(v)}_{I\to \mathrm{HSlit}}\in \mathcal{G}(HS(v))$ such that
\[
\mathbb{P}(C^{(v)}_{I\to \mathrm{HSlit}})\ge 1-n^{-\varepsilon/4},
\]
and on $C^{(v)}_{I\to \mathrm{HSlit}}$ the vertex $v$ is $(n^{\varepsilon_3}, n^{\varepsilon_1/24})$-HS-lit with respect to $HS(v)$ at time $1$, where $\varepsilon_2=\tfrac{1}{200}\min\{\rho,\lambda,1\}\varepsilon_1$ and $\varepsilon_3=\tfrac14\varepsilon_2$.
\end{lemma}

\begin{proof}
Let $k:=\lceil n^{\varepsilon}\rceil$, so that $v$ has the $k$ neighbors $u_1,\dots,u_k$ inside $HS(v)$. Define
\[
E_0:=\{N_v^{\mathrm{rec}}([0,k^{-1/3}])=0\}.
\]
Then $\mathbb{P}(E_0)=\exp(-k^{-1/3})\ge 1-n^{-\varepsilon/3}$, and on $E_0$ we have $\sigma_v(t)=I$ for all $t\in[0,k^{-1/3}]$.

Apply Lemma~\ref{lemma_star_A2v} to the star $(v,\{u_1,\dots,u_k\})$ inside $HS(v)$. It yields an event $E_1\in \mathcal{G}(HS(v))$ such that
\[
\mathbb{P}(E_1)\ \ge\ 1-\exp(-k^{1/24})\ge 1-\exp\!\big(-n^{\varepsilon/24}\big),
\]
and on $E_1$ there are at least $k^{1/24}\ge n^{\varepsilon/24}$ indices $i$ with $u_i$ infected by time $k^{-1/3}$ and not recovering before
time $(\log k)/4\ge (\varepsilon/4)\log n$. In particular, for $n$ large, these $u_i$ are infected at time $t=1$.
Set $C^{(v)}_{I\to \mathrm{HSlit}}:=E_0\cap E_1$. Then $C^{(v)}_{I\to \mathrm{HSlit}}\in \mathcal{G}(HS(v))$ and
\[
\mathbb{P}\!\left(C^{(v)}_{I\to \mathrm{HSlit}}\right)
\ge 1-n^{-\varepsilon/3}-\exp\!\big(-n^{\varepsilon/24}\big)
\ge 1-n^{-\varepsilon/4}
\]
for all sufficiently large $n$. On $C^{(v)}_{I\to \mathrm{HSlit}}$, at time $t=1$ at least $n^{\varepsilon/24}$ vertices $u_i$ are infected, hence
each such $u_i$ is $n^{\varepsilon_1/24}$-lit among $\{w_{i,1},\dots,w_{i,\lceil n^{\varepsilon_1}\rceil}\}$, and therefore $v$ is
$(n^{\varepsilon/24},n^{\varepsilon_1/24})$-HS-lit at time $t=1$. Since $n^{\varepsilon/24}\ge \lceil n^{\varepsilon_3}\rceil$ for all sufficiently large $n$,
this implies that $v$ is $(n^{\varepsilon_3},n^{\varepsilon_1/24})$-HS-lit at time $1$, completing the proof.
\end{proof}

We now turn to the proof of Proposition~\ref{hierarchicalstar_thm2}. By the discussion above, it remains to establish the following. Starting only from a weaker $(k,n^{\varepsilon_2})$-HS-lit condition, where
$
k\in[\frac14 n^{\varepsilon_3},\,n^{\varepsilon_3}),
$
with high probability the center $v$ reaches the stronger $(n^{\varepsilon_3},n^{\varepsilon_1/24})$-HS-lit state before time $\exp(n^{\varepsilon_4})$, while remaining $(n^{\varepsilon_4},n^{\varepsilon_4})$-HS-lit throughout this time interval. Once this upgrade is established, Proposition~\ref{hierarchicalstar_thm2} follows from Lemma~\ref{lemma_hierarchicalstar_expsurvivaltime}.

The proof of this upgrade has two parts. The first part uses Lemma~\ref{lemma_hierarchicalstar_expsurvivaltime} again, but at a smaller parameter scale, to guarantee that the center does not lose $(n^{\varepsilon_4}, n^{\varepsilon_4})$-HS-lit. The second part combines Lemma~\ref{lemma_edge_transmission_k} with Lemma~\ref{lemma_hierarchicalstar_infectedtoHSlit} to show that, during this period of weaker persistence, enough first-layer vertices are reinfected so that the process reaches the stronger $(n^{\varepsilon_3},n^{\varepsilon_1/24})$-HS-lit configuration.

\begin{proof}[Proof of Proposition~\ref{hierarchicalstar_thm2}]
We introduce a smaller auxiliary scale, denoted by $\epsilon=\varepsilon$, $\epsilon_1=24\varepsilon_2$, $\epsilon_2=\frac{1}{200}\min\{\rho,\lambda,1\}\epsilon_1$, and $\epsilon_3=\frac{1}{4}\epsilon_2$, only for the first application of Lemma~\ref{lemma_hierarchicalstar_expsurvivaltime}. Then $\varepsilon_4=\frac{1}{4}\epsilon_3$.

Since $v$ is $(\frac{1}{4}n^{\varepsilon_3}, n^{\varepsilon_2})$-HS-lit with respect to $HS(v)$, and since $n^{\epsilon_3}\le \frac{1}{4}n^{\varepsilon_3}$ and $n^{\epsilon_1/24}= n^{\varepsilon_2}$, there exists a hierarchical-star
$\widetilde{HS}(v):=\mathrm{HStar}(v;\{{u}_i\}_{i\le \lceil n^{\epsilon}\rceil}, \{\tilde{w}_{i,j}\}_{i\le \lceil n^{\epsilon}\rceil, j\le \lceil n^{\epsilon_1}\rceil})\subseteq HS(v)$
such that $v$ is $(n^{\epsilon_3}, n^{\epsilon_1/24})$-HS-lit with respect to $\widetilde{HS}(v)$. By Lemma~\ref{lemma_hierarchicalstar_expsurvivaltime}, there exists an event $B_1^{(v)}\in\mathcal{G}(\widetilde{HS}(v))\subseteq \mathcal{G}(HS(v))$ such that
\[
\mathbb{P}(B_1^{(v)})\ge 1-\exp(-n^{\epsilon_3/2})\ge 1-\exp(-n^{\varepsilon_4}),
\]
and, on $B_1^{(v)}$, the vertex $v$ remains $(\frac{1}{4}n^{\epsilon_3}, n^{\epsilon_2})$-HS-lit with respect to $\widetilde{HS}(v)$ up to time $\exp(n^{\epsilon_3/2})\ge \exp(n^{\varepsilon_4})$. Since $\widetilde{HS}(v)\subseteq HS(v)$, $\frac{1}{4}n^{\epsilon_3}\ge n^{\varepsilon_4}$, and $n^{\epsilon_2}\ge n^{\varepsilon_4}$, the event $B_1^{(v)}$ also implies that the vertex $v$ remains $(n^{\varepsilon_4},n^{\varepsilon_4})$-HS-lit with respect to $HS(v)$ throughout the time interval $[0,\exp(n^{\varepsilon_4})]$, which proves part~(i) in Proposition~\ref{hierarchicalstar_thm2}.

On the event $B_1^{(v)}$, there exists a first-layer vertex $u_i$, with $i\le \lceil n^{\epsilon}\rceil$, and random times $T_1,\dots,T_{\lceil n^{\epsilon_2/2}\rceil}$ such that each $T_\ell$ is measurable with respect to $\mathcal{G}(u_i,\{\tilde{w}_{i,1},\dots,\tilde{w}_{i,\lceil n^{\epsilon_1}\rceil}\})$, $T_{\ell+1}-T_\ell\ge 3$, $T_{\lceil n^{\epsilon_2/2}\rceil}\le \lceil n^{\epsilon_2}\rceil$, and $u_i$ is infected at time $T_\ell$ for every $\ell=1,\dots,\lceil n^{\epsilon_2/2}\rceil$.

By Lemma~\ref{lemma_edge_transmission_k} and Lemma~\ref{lemma_hierarchicalstar_infectedtoHSlit}, for each $\ell=1,\dots,\lceil n^{\epsilon_2/2}\rceil$ there exist events
$$
\tilde{E}_\ell^{(v)}:=E_\ell^{(v)}\cap C_{I\to \mathrm{HSlit},\ell}^{(v)}\in\mathcal{G}(HS(v)),
$$
where $E_\ell^{(v)}$ is the shifted edge-transmission event from Lemma~\ref{lemma_edge_transmission_k}, and $C_{I\to \mathrm{HSlit},\ell}^{(v)}$ is the event from Lemma~\ref{lemma_hierarchicalstar_infectedtoHSlit} shifted to start at time $T_\ell+1$. On $E_\ell^{(v)}$, the vertex $v$ is infected at time $T_\ell+1$. On $\tilde{E}_\ell^{(v)}$, the vertex $v$ is $(n^{\varepsilon_3},n^{\varepsilon_1/24})$-HS-lit with respect to $HS(v)$ at time $T_\ell+2$.

Since $T_{\ell+1}-T_\ell\ge 3$, the time intervals supporting the events $\tilde{E}_\ell^{(v)}$ are pairwise disjoint. Conditional on $B_1^{(v)}$ and on the times $T_1,\dots,T_{\lceil n^{\epsilon_2/2}\rceil}$, these events depend on disjoint increments of the modified Harris construction. Hence they are conditionally independent. Moreover, since $E_\ell^{(v)}$ implies that $v$ is infected at time $T_\ell+1$, the strong Markov property at time $T_\ell+1$ gives
$$
\mathbb{P}\big(C_{I\to \mathrm{HSlit},\ell}^{(v)}\mid \mathcal{F}_{T_\ell+1}\big)\ge 1-n^{-\varepsilon/4}.
$$
Therefore, for all sufficiently large $n$,
$$
\mathbb{P}\big(\tilde{E}_\ell^{(v)}\mid B_1^{(v)},T_1,\dots,T_{\lceil n^{\epsilon_2/2}\rceil}\big)\ge \frac12 c_{\mathrm{edge}}.
$$

Define
$$
B_2^{(v)}:=\bigcup_{\ell=1}^{\lceil n^{\epsilon_2/2}\rceil}\tilde{E}_\ell^{(v)}.
$$
Then $B_2^{(v)}\in\mathcal{G}(HS(v))$. On $B_1^{(v)}\cap B_2^{(v)}$, let $\ell_*$ be the smallest index such that $\tilde{E}_{\ell_*}^{(v)}$ occurs, and define
$$
T_v:=T_{\ell_*}+2.
$$
Then $T_v$ is measurable with respect to $\mathcal{G}(HS(v))$, satisfies
$$
T_v\le T_{\lceil n^{\epsilon_2/2}\rceil}+2\le \lceil n^{\epsilon_2}\rceil+2\le \exp(n^{\varepsilon_4})
$$
for all sufficiently large $n$, and on $B_1^{(v)}\cap B_2^{(v)}$ the vertex $v$ is $(n^{\varepsilon_3},n^{\varepsilon_1/24})$-HS-lit with respect to $HS(v)$ at time $T_v$.

Using the conditional independence and the above lower bound, we obtain
$$
\mathbb{P}\big(B_2^{(v)}\mid B_1^{(v)},T_1,\dots,T_{\lceil n^{\epsilon_2/2}\rceil}\big)\ge 1-\Bigl(1-\frac12 c_{\mathrm{edge}}\Bigr)^{\lceil n^{\epsilon_2/2}\rceil}\ge 1-\exp(-n^{\varepsilon_4})
$$
for all sufficiently large $n$.

Applying the time-shifted version of Lemma~\ref{lemma_hierarchicalstar_expsurvivaltime} from time $T_v$, and using the strong Markov property at time $T_v$, we obtain an event $B_3^{(v)}\in\mathcal{G}(HS(v))$ such that
\begin{equation*}
\mathbb{P}\big(B_3^{(v)}\mid \mathcal{F}_{T_v}\big)\ge 1-\exp(-n^{\varepsilon_4})
\end{equation*}
for all sufficiently large $n$, and on $B_1^{(v)}\cap B_2^{(v)}\cap B_3^{(v)}$ the vertex $v$ remains $(\frac14 n^{\varepsilon_3},n^{\varepsilon_2})$-HS-lit with respect to $HS(v)$ throughout the time interval $[T_v,\; T_v+\exp(n^{\varepsilon_3/2})]$. Since $T_v\le \exp(n^{\varepsilon_4})$, it follows that part~(ii) in Proposition~\ref{hierarchicalstar_thm2} holds on $B_1^{(v)}\cap B_2^{(v)}\cap B_3^{(v)}$. Moreover, for all sufficiently large $n$,
$$\mathbb{P}(B_1^{(v)}\cap B_2^{(v)}\cap B_3^{(v)})\ge 1-3\exp(-n^{\varepsilon_4}).$$

Set
\(
L:=\exp(n^{\varepsilon_4}),
m:=2\left\lceil \exp(n^{\varepsilon_4/8})\right\rceil+3 
\), then $mL\le \frac{1}{2}\exp(n^{\varepsilon_3/2})$. For each $i=1,\dots,m$, define $\tilde{B}_{2,i}^{(v)}$ be the event that there exists a random time $\tilde{T}_i\in ((i-1)L,\; iL)$, such that $\tilde{T}_i$ is $\mathcal{G}(HS(v))$ measurable and $v$ is infected at $\tilde{T}_i$. 

Starting from $v$ being $(\frac14 n^{\varepsilon_3},n^{\varepsilon_2})$-HS-lit with respect to $HS(v)$, the construction of $B_1^{(v)}\cap B_2^{(v)}$ implies that $v$ is infected at some $\mathcal{G}(HS(v))$ measurable time $T_v$ before $L$. On the event $B_1^{(v)}\cap B_2^{(v)}\cap B_3^{(v)}$, the vertex $v$ remains $(\frac14 n^{\varepsilon_3},n^{\varepsilon_2})$-HS-lit with respect to $HS(v)$ throughout the time interval $[L,mL]$. Hence, by the same argument as in the construction of $B_1^{(v)}\cap B_2^{(v)}$, for every $i=1,\dots,m$, there exists an event $B_{2,i}^{(v)}\in \mathcal{G}(HS(v))$ such that $B_{2,i}^{(v)}\subseteq \tilde B_{2,i}^{(v)}$ and

\[
\mathbb P\bigl((B_{2,i}^{(v)})^c \mid \mathcal F_{(i-1)L}\bigr)
\le
2\exp(-n^{\varepsilon_4}),\quad \text{on the event $B_1^{(v)}\cap B_2^{(v)}\cap B_3^{(v)}$.}
\]
Consequently, by a union bound,
\[
\mathbb P\Bigl(
\bigcap_{i=1}^m B_{2,i}^{(v)}
\;\Big|\;
B_1^{(v)}\cap B_2^{(v)}\cap B_3^{(v)}
\Bigr)
\ge
1-2m\exp(-n^{\varepsilon_4}).
\]
Hence, defining
\[
B^{(v)}
:=
B_1^{(v)}\cap B_2^{(v)}\cap B_3^{(v)}
\cap
\bigcap_{i=1}^m B_{2,i}^{(v)}\in \mathcal{G}(HS(v)) ,
\]
we obtain
\[
\mathbb P(B^{(v)})
\ge
1-3\exp(-n^{\varepsilon_4})-2m\exp(-n^{\varepsilon_4})
\ge
1-\exp(-n^{\varepsilon_4/10})
\]
for all sufficiently large $n$.

On the event $B^{(v)}$, denote the corresponding $\tilde{T}_i$ for each $i=1,\dots,m$. Define
\[
T_i:=\tilde{T}_{2i+1},
\qquad
i=1,\dots,\left\lceil \exp(n^{\varepsilon_4/8})\right\rceil+1 .
\]
Since $\tilde{T}_{2i+1}\in (2iL,(2i+1)L)$ and $\tilde{T}_{2i-1}\in ((2i-2)L,(2i-1)L)$, we have
\[
 T_{i+1}- T_{i}\ge L=\exp(n^{\varepsilon_4}), T_i\le (2\lceil\exp(n^{\varepsilon_4/8})\rceil+3)L\le \frac{1}{2}\exp(n^{\varepsilon_3/2})
\]
for every $i=1,\dots,\lceil \exp(n^{\varepsilon_4/8})\rceil$, for all sufficiently large $n$. Moreover, $T_i$ is $\mathcal{G}(HS(v))$-measurable for every $i=1,\dots,\lceil \exp(n^{\varepsilon_4/8})\rceil+1$, since each corresponding $\tilde T_{2i+1}$ is $\mathcal{G}(HS(v))$-measurable. Hence part~(iii) of Proposition~\ref{hierarchicalstar_thm2} holds on $B^{(v)}$.

\end{proof}

We now argue that Proposition~\ref{hierarchicalstar_thm2} remains valid for the threshold-SIRS process. In this setting, the local graphical $\sigma$-field on a hierarchical-star $\mathrm{HStar}(v;\{u_i\}_{i\le k_1};\{w_{i,j}\}_{i\le k_1, j\le k_2})$ is defined by
$$
\mathcal{G}(\mathrm{HStar}(v;\{u_i\}_{i\le k_1};\{w_{i,j}\}_{i\le k_1, j\le k_2}))
:=
\sigma\bigl((N_x^{\mathrm{rec}},N_x^{\mathrm{sus}},N_x^{\mathrm{inf}})_{x\in \{v\}\cup\{u_i:i\le k_1\}\cup\{w_{i,j}:i\le k_1,\ j\le k_2\}}\bigr).
$$

After replacing the edge infection clock $N_{u\to v}^{\mathrm{inf}}$ by the vertex infection clock $N_v^{\mathrm{inf}}$, Lemma~\ref{lemma_edge_transmission_one} continues to hold for the threshold-SIRS process. Consequently, Lemma~\ref{lemma_edge_transmission_k} also remains valid. In Section~\ref{subsection_hierarchicalstar_2}, we have already explained that Proposition~\ref{prop:star_independent} holds for the threshold-SIRS process as well. Therefore, every ingredient used in the proof of Lemma~\ref{lemma_hierarchicalstar_supermartingale} remains available in the threshold setting, and so Lemma~\ref{lemma_hierarchicalstar_supermartingale} also holds for the threshold-SIRS process. The same conclusion then follows for Lemma~\ref{lemma_hierarchicalstar_expsurvivaltime}, Lemma~\ref{lemma_hierarchicalstar_infectedtoHSlit}, and finally Proposition~\ref{hierarchicalstar_thm2}.

\section{Hierarchical-Stars in Power Law Random Graph}\label{section_hierarchicalstarinpowerlaw}
The purpose of this section is twofold. First, we prove Proposition~\ref{prop:existsmanyhierarchicalstar}, which shows that, with high probability, the power-law random graph contains at least $\lceil n^{1-\delta}\rceil$ vertex-disjoint $(n^\varepsilon,n^\varepsilon)$ hierarchical-stars. Heuristically, fix $\delta\in(0,1)$. With high probability, there are at least $\lceil n^{1-\delta}\rceil$ vertices $v_1,\dots,v_{\lceil n^{1-\delta}\rceil}$ with degree at least $n^{2\tilde{\varepsilon}}$ for a suitable small $\tilde{\varepsilon}>0$. Starting from one such vertex $v_1$, if we expose its half-edges one by one, then a newly paired vertex $u$ falls into a degree window above $n^{\bar\varepsilon}$ with probability of order
\[
\frac{\sum_{u\in V:\,d_u\ge n^{\bar\varepsilon}}d_u}{\sum_{u\in V}d_u}\asymp n^{-\bar\varepsilon(\tau-2)}.
\]
Although this probability changes slightly during the sequential exposure, it remains of the same order as long as only a negligible fraction of the relevant half-edges has been used. Thus, by choosing $\tilde{\varepsilon}$ sufficiently larger than $\bar\varepsilon$, one expects to find at least $n^{\tilde{\varepsilon}}$ distinct neighbors of $v_1$ whose degrees are at least $n^{\bar\varepsilon}$, with exponentially high probability. Repeating this exposure for $v_2,\dots,v_{\lceil n^{1-\delta}\rceil}$ still works because the half-edges used in earlier rounds form only a negligible portion of the available pool. Hence, heuristically, all these vertices simultaneously acquire many disjoint first-layer neighbors of degree at least $n^{\bar\varepsilon}$. Finally, applying the same idea once more to the first-layer vertices produces the second layer and hence the desired collection of disjoint hierarchical-stars, with $\varepsilon\ll \bar\varepsilon\ll \tilde{\varepsilon}$. Lemmas~\ref{lemma_degree_event_D}, \ref{lemma_hierarchicalstarv11layer}, \ref{lemma_hierarchicalstarvi1layer}, \ref{lemma_hierarchicalstar_second_layer} make this informal argument precise.

For the proof of Theorem~\ref{thm_main1}, we will need two geometric properties of the graph. The first is precisely the existence of many disjoint hierarchical-stars, established here. The second is the logarithmic diameter bound, namely that $\mathrm{diam}(G)\le c\log n$ for some constant $c$, which will be introduced in Section~\ref{section_proofofmain}.

Second, we prove Lemmas~\ref{lemma_Mhierarchicalstar} and~\ref{lemma_logMline}, which will be used in the proof of Theorem~\ref{thm_main2}. Roughly speaking, Lemma~\ref{lemma_Mhierarchicalstar} shows that any vertex of sufficiently large constant degree is, with high probability, the center of a hierarchical-star whose size is polynomial in its degree, while Lemma~\ref{lemma_logMline} shows that, starting from such a large constant degree vertex $v_1$, one can find a path leading to a vertex $v_k$ whose degree is at least $n^{\min\{1/\tau, (\tau-2)/(\tau-1)\}/32}$. In the proof of Theorem~\ref{thm_main2}, we will show that the infection spreads from $v_1$ to $v_k$ with probability bounded below by a positive constant.

\begin{proposition}\label{prop:existsmanyhierarchicalstar}
Fix $\delta\in(0,1)$ and set
$
\varepsilon:=\frac{\delta}{400\tau^3}.
$
Let $G=G(n,\tau)$ be a power-law random graph on $n$ vertices with exponent $\tau>2$. Then there exists a constant $c>0$ such that, for all sufficiently large $n$,
\begin{equation}
\mathbb P\Bigl(
\text{$G$ contains at least $\lceil n^{1-\delta}\rceil$ vertex-disjoint $(n^{\varepsilon},n^{\varepsilon})$-hierarchical-stars}
\Bigr)
\ge
1-c(\log n)^{\mathbf{1}\{\tau=3\}}\,n^{-\min\{\tau-2,1\}}.
\end{equation}
\end{proposition}

Note that the probability bound in Proposition~\ref{prop:existsmanyhierarchicalstar} depends on the value of $\tau$. This difference ultimately comes from the estimate on the event
\(
\left\{\sum_{v\in V} d_v \in \left[\frac{1}{2}n\mu,\frac{3}{2}n\mu\right]\right\}\), with \(\mu=\mathbb{E}[d_v]
\),
which is established in Lemma~\ref{lemma_degree_event_D}. When $\tau>3$, the degree distribution has finite variance, and the desired bound follows directly from Chebyshev's inequality. In the borderline case $\tau=3$ and in the infinite-variance regime $\tau\in(2,3)$, one must use different arguments to control this event, and this is the reason the final probability estimates take different forms in the three regimes.\\

We begin with a regularity statement for the degree sequence. Lemma~\ref{lemma_degree_event_D} collects the quantitative degree estimates that will be used throughout this subsection.

\begin{lemma}\label{lemma_degree_event_D}
Let $\bm d :=(d_i)_{i=1}^n$ be i.i.d.\ random variables with distribution
$
\mathbb{P}(d_i=k)=C_\tau\,k^{-\tau}\mathbf{1}\{k\ge 3\}
$, and let $F$ denote the corresponding distribution function. Write
$
d^{(1)}\ge d^{(2)}\ge \cdots \ge d^{(n)}
$
for the order statistics of $d_1,\dots,d_n$. Fix any
$
\varepsilon\in \left(0,\frac{1}{2\tau-1}\right),
$
and define
$c_1:=\left(\frac{C_\tau}{\tau-1}\right)^{\!\frac{1}{\tau-1}}5^{-\frac{1}{\tau-1}}$
, $c_2:=\left(\frac{C_\tau}{\tau-1}\right)^{\!\frac{1}{\tau-1}}8^{\frac{1}{\tau-1}}
$.
Let
$
\mu:=\mathbb E[d_1]
$,
and define the event
\begin{equation}\label{eq_D}
D_\varepsilon:=
\biggl\{
\forall i\in[\lceil n^\varepsilon\rceil ,n],\
c_1\left(\frac{n}{i}\right)^{\!\frac{1}{\tau-1}}
\le d^{(i)}\le
c_2\left(\frac{n}{i}\right)^{\!\frac{1}{\tau-1}}
\biggr\}\cap
\biggl\{
\biggl|\sum_{i=1}^n d^{(i)}-n\mu\biggr|\le \frac12 n\mu
\biggr\}.
\end{equation}
Then there exists a constant $c'>0$, depending only on $\tau$, such that, for all sufficiently large $n$,
\begin{equation}\label{eq_probofD}
\mathbb P(D_\varepsilon)\ge
\begin{cases}
1-c'n^{-(\tau-2)}, & \tau\in(2,3),\\[1mm]
1-c'(\log n)\,n^{-1}, & \tau=3,\\[1mm]
1-c'n^{-1}, & \tau>3.
\end{cases}
\end{equation}
Moreover, on the event $D_\varepsilon$, for every
$
i\in \bigl[\lfloor n^{1-2\varepsilon(\tau-1)}\rfloor,\,\lfloor n^{1-1.1\varepsilon(\tau-1)}\rfloor\bigr]
$,
one has
\begin{equation}\label{eq_d_i_between}
n^\varepsilon\le d^{(i)}\le n^{3\varepsilon}
\end{equation}
and there exist constants $C_1,C_2>0$, depending only on $\tau$, such that
\begin{equation}\label{eq_sumH_lower}
C_1 n^{1-1.1\varepsilon(\tau-2)}
\le
\sum_{i=\lfloor n^{1-2\varepsilon(\tau-1)}\rfloor}^{\lfloor n^{1-1.1\varepsilon(\tau-1)}\rfloor} d^{(i)}
\le
C_2 n^{1-1.1\varepsilon(\tau-2)}.
\end{equation}
\end{lemma}

\begin{proof}

The proof consists of three steps. We first derive uniform bounds on the order statistics $(d^{(i)})_{i=1}^n$ from the quantile representation and the power-law tail asymptotics. We then estimate the deviation of the total degree sum from its mean. The argument is split into the cases $\tau>3$, $\tau=3$, and $\tau\in(2,3)$, reflecting the fact that the second moment is finite only in the first regime and is critical at $\tau=3$. Finally, we combine these estimates to obtain the event $D_\varepsilon$, and then deduce \eqref{eq_d_i_between} and \eqref{eq_sumH_lower} from the bounds defining $D_\varepsilon$.

By  \cite[Eq.~(2.19)]{DharaEtAl2019}, there exists an i.i.d.\ sequence $(E_i)_{i\ge 1}$ of unit-rate exponential random variables such that, defining
$
\Gamma_i:=\sum_{j=1}^i E_j,\ \bar d_i:=(1-F)^{-1}\!\left(\frac{\Gamma_i}{\Gamma_{n+1}}\right)$ for $i=1,\dots,n
$,
we have
\begin{equation}\label{eq_orderstat_representation}
(d^{(1)},d^{(2)},\dots,d^{(n)})\overset{d}= (\bar d_1,\bar d_2,\dots,\bar d_n).
\end{equation}
We first derive bounds on the tail of the degree distribution. For any $t\ge 0$,
\begin{align*}
\mathbb P(d_1>t)
&=
\sum_{k=\lfloor t\rfloor+1}^{\infty}\mathbb P(d_1=k)
=
C_\tau\sum_{k=\lfloor t\rfloor+1}^{\infty}k^{-\tau}.
\end{align*}
Using the integral test, we obtain
\begin{align*}
\mathbb P(d_1>t)
&\ge
C_\tau\int_{\lfloor t\rfloor+1}^{\infty}(x+1)^{-\tau}\,dx
=
\frac{C_\tau}{\tau-1}(\lfloor t\rfloor+2)^{1-\tau},
\\
\mathbb P(d_1>t)
&\le
C_\tau\int_{\lfloor t\rfloor+1}^{\infty}x^{-\tau}\,dx
=
\frac{C_\tau}{\tau-1}(\lfloor t\rfloor+1)^{1-\tau}.
\end{align*}
It follows that, for every $x\in(0,1)$,
\begin{equation}\label{eq_tail_inverse}
\left(\frac{C_\tau}{\tau-1}\right)^{\!\frac{1}{\tau-1}}x^{-\frac{1}{\tau-1}}-3
\le
(1-F)^{-1}(x)
\le
\left(\frac{C_\tau}{\tau-1}\right)^{\!\frac{1}{\tau-1}}x^{-\frac{1}{\tau-1}}.
\end{equation}
Next, fix $i\in[n^\varepsilon,n+1]$. A standard Chernoff bound yields, for all sufficiently large $n$,
\begin{equation}\label{eq_gamma_chernoff}
\mathbb P\bigl(|\Gamma_i-i|\ge i\bigr)\le e^{-i/2}\le e^{-n^\varepsilon/2}.
\end{equation}
Hence, by a union bound,
\begin{equation}\label{eq_gamma_union}
\mathbb P\Bigl(\exists\, i\in[n^\varepsilon,n+1]: |\Gamma_i-i|\ge i\Bigr)
\le e^{-n^{\varepsilon/3}}
\end{equation}
for all sufficiently large $n$.

On the event that $|\Gamma_i-i|<i$ for all $i\in[n^\varepsilon,n+1]$, we have
$
\frac{i}{2}\le \Gamma_i\le 2i,\ \frac{n+1}{2}\le \Gamma_{n+1}\le 2(n+1)
$,
so that
\begin{equation}\label{eq_ratio_bounds}
\frac{1}{4}\frac{i}{n+1}
\le
\frac{\Gamma_i}{\Gamma_{n+1}}
\le
4\frac{i}{n+1}
\end{equation}
for every $i\in[n^\varepsilon,n+1]$.
Combining \eqref{eq_tail_inverse}, \eqref{eq_gamma_union}, and \eqref{eq_ratio_bounds}, we obtain that, for all sufficiently large $n$, with probability at least $1-\exp(-n^{\varepsilon/3})$, for every $i\in[n^\varepsilon,n]$,
\begin{equation}\label{eq_quantile_bounds}
\left(\frac{C_\tau}{\tau-1}\right)^{\!\frac{1}{\tau-1}}
\left(\frac{1}{4}\right)^{\!\frac{1}{\tau-1}}
\left(\frac{n+1}{i}\right)^{\!\frac{1}{\tau-1}}
-3
\le
(1-F)^{-1}\!\left(\frac{\Gamma_i}{\Gamma_{n+1}}\right)
\le
\left(\frac{C_\tau}{\tau-1}\right)^{\!\frac{1}{\tau-1}}
4^{\frac{1}{\tau-1}}
\left(\frac{n+1}{i}\right)^{\!\frac{1}{\tau-1}}.
\end{equation}
For all sufficiently large $n$ and all $i\in[n^\varepsilon,n]$, we have
$
\left(\frac{1}{4}\right)^{\!\frac{1}{\tau-1}}
\left(\frac{n+1}{i}\right)^{\!\frac{1}{\tau-1}}
-3
\ge
5^{-\frac{1}{\tau-1}}\left(\frac{n}{i}\right)^{\!\frac{1}{\tau-1}}
$,
and
$
4^{\frac{1}{\tau-1}}
\left(\frac{n+1}{i}\right)^{\!\frac{1}{\tau-1}}
\le
8^{\frac{1}{\tau-1}}
\left(\frac{n}{i}\right)^{\!\frac{1}{\tau-1}}
$.
Therefore, by \eqref{eq_orderstat_representation} and \eqref{eq_quantile_bounds},
\begin{equation}\label{eq_orderstat_bounds}
\mathbb P\biggl(
\forall i\in[n^\varepsilon,n],\
c_1\left(\frac{n}{i}\right)^{\!\frac{1}{\tau-1}}
\le d^{(i)}\le
c_2\left(\frac{n}{i}\right)^{\!\frac{1}{\tau-1}}
\biggr)
\ge
1-\exp(-n^{\varepsilon/3}).
\end{equation}
When $\tau>3$, the second moment of $d_1$ is finite; write $\sigma^2:=\mathrm{Var}(d_1)$. By Chebyshev's inequality,
\begin{equation}\label{eq_chebyshev_sum}
\mathbb P\biggl(
\biggl|\sum_{i=1}^n d^{(i)}-n\mu\biggr|>\frac12 n\mu
\biggr)
\le
\frac{\mathrm{Var}\!\left(\sum_{i=1}^n d^{(i)}\right)}{(n\mu/2)^2}
=
\frac{4\sigma^2}{\mu^2 n}.
\end{equation}
When $\tau\in(2,3]$, define
$
Y_i:=d_i\mathbf 1\{d_i\le n\},\ T_n:=\sum_{i=1}^n Y_i
$.
Then
\begin{align*}
\mathbb P\biggl(
\biggl|\sum_{i=1}^n d^{(i)}-n\mu\biggr|>\frac12 n\mu
\biggr)
&\le
\mathbb P(d^{(1)}>n)
+
\mathbb P\bigl(|T_n-n\mu|>\tfrac12 n\mu\bigr).
\end{align*}
By the upper tail estimate above,
\begin{equation}\label{eq_dmax_n_bound_heavytail}
\mathbb P(d^{(1)}>n)
\le
n\,\mathbb P(d_1>n)
\le
\frac{C_\tau}{\tau-1}n^{-(\tau-2)}.
\end{equation}
Moreover,
\begin{align*}
\mu-\mathbb E[Y_1]
&=
\mathbb E[d_1\mathbf 1\{d_1>n\}]
=
C_\tau\sum_{k=n+1}^{\infty}k^{-(\tau-1)}
\le
2C_\tau n^{-(\tau-2)}.
\end{align*}
Hence, for all sufficiently large $n$,
\begin{equation}\label{eq_mean_shift_heavytail}
n\bigl(\mu-\mathbb E[Y_1]\bigr)\le \frac14 n\mu,
\end{equation}
so that
\begin{equation}\label{eq_centering_heavytail}
\bigl\{|T_n-n\mu|>\tfrac12 n\mu\bigr\}
\subseteq
\bigl\{|T_n-n\mathbb E[Y_1]|>\tfrac14 n\mu\bigr\}.
\end{equation}
We now estimate the variance term separately in the cases $\tau\in(2,3)$ and $\tau=3$.
If $\tau\in(2,3)$, then
\begin{align*}
\mathrm{Var}(Y_1)
&\le
\mathbb E[Y_1^2]
=
C_\tau\sum_{k=3}^{n}k^{2-\tau}
\le
2C_\tau n^{3-\tau},
\end{align*}
for all sufficiently large $n$, and thus
\begin{equation}\label{eq_var_Tn_subcritical}
\mathrm{Var}(T_n)=n\,\mathrm{Var}(Y_1)\le 2C_\tau n^{4-\tau}.
\end{equation}
By Chebyshev's inequality,
\begin{align}
\mathbb P\bigl(|T_n-n\mathbb E[Y_1]|>\tfrac14 n\mu\bigr)
&\le
\frac{16\,\mathrm{Var}(T_n)}{\mu^2 n^2}
\le
2C_\tau n^{2-\tau}
=
2C_\tau n^{-(\tau-2)}.
\label{eq_sum_bound_tau_in_23}
\end{align}
Combining \eqref{eq_dmax_n_bound_heavytail}, \eqref{eq_centering_heavytail}, and \eqref{eq_sum_bound_tau_in_23}, we obtain
\begin{equation}\label{eq_sum_bound_heavytail_subcritical}
\mathbb P\biggl(
\biggl|\sum_{i=1}^n d^{(i)}-n\mu\biggr|>\frac12 n\mu
\biggr)
\le
2C_\tau n^{-(\tau-2)}.
\end{equation}
If $\tau=3$, then
\begin{align*}
\mathrm{Var}(Y_1)
&\le
\mathbb E[Y_1^2]
=
C_\tau\sum_{k=3}^{n}k^{-1}
\le
2C_\tau\log n,
\end{align*}
for all sufficiently large $n$, and thus
\begin{equation}\label{eq_var_Tn_tau3}
\mathrm{Var}(T_n)=n\,\mathrm{Var}(Y_1)\le 2C_\tau n\log n.
\end{equation}
By Chebyshev's inequality,
\begin{align}
\mathbb P\bigl(|T_n-n\mathbb E[Y_1]|>\tfrac14 n\mu\bigr)
&\le
\frac{16\,\mathrm{Var}(T_n)}{\mu^2 n^2}
\le
2C_\tau\frac{\log n}{n}.
\label{eq_sum_bound_tau3_centered}
\end{align}
Combining \eqref{eq_dmax_n_bound_heavytail}, \eqref{eq_centering_heavytail}, and \eqref{eq_sum_bound_tau3_centered}, we obtain
\begin{equation}\label{eq_sum_bound_tau3}
\mathbb P\biggl(
\biggl|\sum_{i=1}^n d^{(i)}-n\mu\biggr|>\frac12 n\mu
\biggr)
\le
2C_\tau\frac{\log n}{n}.
\end{equation}
The bound \eqref{eq_probofD} now follows from \eqref{eq_orderstat_bounds}, \eqref{eq_chebyshev_sum}, \eqref{eq_sum_bound_heavytail_subcritical}, and \eqref{eq_sum_bound_tau3}. Indeed, for all sufficiently large~$n$,
\begin{align*}
\mathbb P(D_\varepsilon^{\mathrm{c}})
&\le
\exp(-n^{\varepsilon/3})
+
\begin{cases}
2C_\tau n^{-(\tau-2)}, & \tau\in(2,3),\\[1mm]
2C_\tau(\log n)\,n^{-1}, & \tau=3,\\[1mm]
\frac{4\sigma^2}{\mu^2 n}, & \tau>3.
\end{cases}
\end{align*}
Choosing $c'$ sufficiently large yields \eqref{eq_probofD}.

We now derive the remaining consequences on the event $D_\varepsilon$. Let
$$
a_n:=\lfloor n^{1-2\varepsilon(\tau-1)}\rfloor,
\qquad
b_n:=\lfloor n^{1-1.1\varepsilon(\tau-1)}\rfloor.
$$
If
$i\in \bigl[n^{1-2\varepsilon(\tau-1)},\,n^{1-1.1\varepsilon(\tau-1)}\bigr],$
then $1-2\varepsilon(\tau-1)>\varepsilon$, and hence, by the definition of $D_\varepsilon$,
\begin{equation}\label{eq_di_window}
c_1\left(\frac{n}{i}\right)^{\!\frac{1}{\tau-1}}
\le d^{(i)}\le
c_2\left(\frac{n}{i}\right)^{\!\frac{1}{\tau-1}}.
\end{equation}
Summing the lower bound in \eqref{eq_di_window} from $a_n$ to $b_n$, we obtain
\begin{align*}
\sum_{i=a_n}^{b_n} d^{(i)}
&\ge
c_1\sum_{i=a_n}^{b_n}\left(\frac{n}{i}\right)^{\!\frac{1}{\tau-1}}
\ge
c_1 n^{\frac{1}{\tau-1}}
\int_{a_n}^{b_n+1} x^{-\frac{1}{\tau-1}}\,dx.
\end{align*}
Since $\tau>2$,
\begin{align*}
\int_{a_n}^{b_n+1} x^{-\frac{1}{\tau-1}}\,dx
&=
\frac{\tau-1}{\tau-2}
\Bigl((b_n+1)^{1-\frac{1}{\tau-1}}-a_n^{1-\frac{1}{\tau-1}}\Bigr).
\end{align*}
Moreover,
\begin{align*}
\frac{a_n^{1-\frac{1}{\tau-1}}}{(b_n+1)^{1-\frac{1}{\tau-1}}}
\le
\frac{\bigl(n^{1-2\varepsilon(\tau-1)}\bigr)^{1-\frac{1}{\tau-1}}}
{\bigl(n^{1-1.1\varepsilon(\tau-1)}\bigr)^{1-\frac{1}{\tau-1}}}
=
n^{-0.9\varepsilon(\tau-2)},
\end{align*}
which tends to $0$ as $n\to\infty$. Therefore, for all sufficiently large $n$,
\begin{equation}\label{eq_integral_lower}
\int_{a_n}^{b_n+1} x^{-\frac{1}{\tau-1}}\,dx
\ge
c\, (b_n+1)^{1-\frac{1}{\tau-1}}
\end{equation}
for some constant $c>0$ depending only on $\tau$. Consequently,
\begin{align*}
\sum_{i=a_n}^{b_n} d^{(i)}
\ge
C_1 n^{\frac{1}{\tau-1}}
\left(n^{1-1.1\varepsilon(\tau-1)}\right)^{1-\frac{1}{\tau-1}}
=
C_1 n^{1-1.1\varepsilon(\tau-2)}
\end{align*}
for some constant $C_1>0$ depending only on $\tau$.
Similarly, using the upper bound in \eqref{eq_di_window},
\begin{align*}
\sum_{i=a_n}^{b_n} d^{(i)}
\le
c_2\sum_{i=a_n}^{b_n}\left(\frac{n}{i}\right)^{\!\frac{1}{\tau-1}}
\le
c_2 n^{\frac{1}{\tau-1}}
\int_{a_n-1}^{b_n} x^{-\frac{1}{\tau-1}}\,dx
\le
C_2 n^{1-1.1\varepsilon(\tau-2)}
\end{align*}
for some constant $C_2>0$ depending only on $\tau$. This proves \eqref{eq_sumH_lower}.
Finally, for every
$
i\in [a_n, b_n]
$,
we have
\begin{align*}
d^{(i)}
&\ge
d^{(b_n)}
\ge
c_1\left(\frac{n}{n^{1-1.1\varepsilon(\tau-1)}}\right)^{\!\frac{1}{\tau-1}}
=
c_1 n^{1.1\varepsilon}
>
n^\varepsilon,
\\
d^{(i)}
&\le
d^{(a_n)}
\le
c_2\left(\frac{n}{n^{1-2\varepsilon(\tau-1)}}\right)^{\!\frac{1}{\tau-1}}
=
c_2 n^{2\varepsilon}
<
n^{3\varepsilon}
\end{align*}
for all sufficiently large $n$. This proves \eqref{eq_d_i_between}.
\end{proof}

\begin{lemma}\label{lemma_hierarchicalstarv11layer}
Fix any $\tilde{\varepsilon}\in \left(0,\frac{1}{4\tau}\right)$, and set
$
\bar\varepsilon=\frac{\tilde{\varepsilon}}{4.4\tau}
$.
Let $G=G(n,\tau)$ be a power-law random graph on $n$ vertices with exponent $\tau>2$, and let
$
\bm d=(d_1,\dots,d_n)
$
denote its degree sequence. Suppose that $v$ is a vertex of degree at least $n^{2\tilde{\varepsilon}}$. Let $A_v$ denote the event that $v$ has at least $\lceil n^{\tilde{\varepsilon}}\rceil$ distinct neighbors whose degrees all lie in the interval $(n^{\bar\varepsilon},n^{3\bar\varepsilon})$. Then, for all sufficiently large $n$,
\begin{equation}\label{eq_lemma_hierarchicalstarv11layer}
\mathbb{P}(A_v\mid \bm d)\ge 1-\exp(-n^{\bar\varepsilon(\tau-2)}), \qquad \text{on the event $D_{\bar\varepsilon}$ defined in \eqref{eq_D}. }
\end{equation}
\end{lemma}

\begin{proof}
We work throughout conditional on the degree sequence $\bm d$, and on the event $D_{\bar\varepsilon}$. Conditional on $\bm d$, the remaining randomness is the uniform pairing of the half-edges in the configuration model. We will expose the first $\lceil n^{2\tilde\varepsilon}\rceil$ half-edges of $v$, namely
$
(v,1),(v,2),\dots,(v,\lceil n^{2\tilde\varepsilon}\rceil)
$,
one by one, and show that with high probability at least $n^{\tilde\varepsilon}$ of them are paired to half-edges attached to distinct vertices whose degrees lie in $(n^{\bar\varepsilon},n^{3\bar\varepsilon})$.

For each vertex $u\in V$, let
$
\Omega_u:=\{(u,1),\dots,(u,d_u)\}
$
be the set of half-edges incident to $u$. Fix a deterministic tie-breaking rule and write $v^{(1)},\dots,v^{(n)}$ for the vertices ordered so that $d_{v^{(1)}}\ge d_{v^{(2)}}\ge \cdots \ge d_{v^{(n)}}$. Define
$$
\mathcal{K}_{\bar\varepsilon}:=\Bigl\{v^{(i)}:\ i\in
\bigl[\lfloor n^{1-2\bar\varepsilon(\tau-1)}\rfloor,\,
\lfloor n^{1-1.1\bar\varepsilon(\tau-1)}\rfloor\bigr]
\Bigr\}.
$$

By Lemma~\ref{lemma_degree_event_D}, on the event $D_{\bar\varepsilon}$, every vertex in $\mathcal{K}_{\bar\varepsilon}$ has degree in $(n^{\bar\varepsilon},n^{3\bar\varepsilon})$, hence $v\notin \mathcal{K}_{\bar\varepsilon}$, and
\begin{equation}\label{eq_sumK_hierarchicalstar}
C_1 n^{1-1.1\bar\varepsilon(\tau-2)}
\le
\sum_{u\in\mathcal{K}_{\bar\varepsilon}} d_u
\le
C_2 n^{1-1.1\bar\varepsilon(\tau-2)}
\end{equation}
for some constants $C_1,C_2>0$ depending only on $\tau$.

For $1\le i\le \lceil n^{2\tilde\varepsilon}\rceil$, let $e_i$ denote the partner half-edge paired to $(v,i)$, and let $u_i$ be the vertex incident to $e_i$, that is, $e_i\in \Omega_{u_i}$. Since we are working in the multigraph configuration model, the same vertex may appear several times among $u_1,u_2,\dots$, and we do not forbid this. We only require that the half-edges paired at different steps are distinct, which is automatic under the pairing procedure.
Let
$U_{i-1}:=\{u_1,\dots,u_{i-1}\}$ be the set of distinct vertices already hit after the first $i-1$ exposures. Let $\bar{\mathcal{F}}_0$ be the trivial $\sigma$-algebra at time $0$ with no pairing information and let
$$
\bar{\mathcal F}_i:=\sigma(e_1,\dots,e_i).
$$
For $i=1,\dots,\lceil n^{2\tilde\varepsilon}\rceil$, define
$$
B_i:=
\{u_i\in \mathcal{K}_{\bar\varepsilon}\setminus U_{i-1}\}.
$$
Thus $B_i$ is the event that, at step $i$, the newly exposed half-edge of $v$ is paired to a half-edge incident to a vertex in $\mathcal{K}_{\bar\varepsilon}$ that has not been seen before. Therefore,
\begin{equation}\label{eq_Av_contains_Bi_new}
A_v\supseteq
\biggl\{
\sum_{i=1}^{\lceil n^{2\tilde\varepsilon}\rceil}\mathbf 1_{B_i}\ge n^{\tilde\varepsilon}
\biggr\}.
\end{equation}
We first compute the probability of $B_1$. For the first pairing, the admissible half-edge pool is all half-edges except $(v,1)$. Hence,
\begin{equation}\label{eq_probofB1_multigraph}
\mathbb P(B_1 )
=
\frac{\sum_{u\in\mathcal{K}_{\bar\varepsilon}} d_u}{\sum_{u\in V} d_u-1}.
\end{equation}
Since on $D_{\bar\varepsilon}$,
$
\frac12 \mu n\le \sum_{u\in V} d_u\le \frac32 \mu n
$, it follows from \eqref{eq_sumK_hierarchicalstar} that for all sufficiently large $n$,
\begin{equation}\label{eq_probofB1_bounds_multigraph}
c_3 n^{-1.1\bar\varepsilon(\tau-2)}
\le
\mathbb P(B_1)
\le
c_4 n^{-1.1\bar\varepsilon(\tau-2)},
\end{equation}
where
$
c_3:=\frac{2C_1}{3\mu}, c_4:=\frac{4C_2}{\mu}
$ are constant.

For $2\le i\le \lceil n^{2\tilde\varepsilon}\rceil$, after the first $i-1$ pairings, exactly $2(i-1)$ half-edges have been removed from the pool, namely
$$
(v,1),\dots,(v,i-1),\ e_1,\dots,e_{i-1}.
$$
Therefore the partner of $(v,i)$ is uniformly distributed over a set of cardinality
$
\sum_{u\in V}d_u-(2i-1)
$.
Moreover, the event $B_i$ occurs precisely when this partner belongs to a half-edge attached to some vertex in $\mathcal{K}_{\bar\varepsilon}\setminus U_{i-1}$. Hence
\begin{equation}\label{eq_condprob_Bi_multigraph}
\mathbb P(B_i\mid \bar{\mathcal F}_{i-1})
=
\frac{\sum_{u\in \mathcal{K}_{\bar\varepsilon}\setminus U_{i-1}} d_u}{\sum_{u\in V} d_u-(2i-1)},
\qquad \text{a.s.}
\end{equation}
We now bound these conditional probabilities from below and above. On $D_{\bar\varepsilon}$, every vertex in $\mathcal{K}_{\bar\varepsilon}$ has degree at most $n^{3\bar\varepsilon}$, so
$$
\sum_{u\in U_{i-1}\cap\mathcal{K}_{\bar\varepsilon}} d_u
\le
(i-1)n^{3\bar\varepsilon}.
$$
Using \eqref{eq_sumK_hierarchicalstar}, we obtain
\begin{align}
\mathbb P(B_i\mid \bar{\mathcal F}_{i-1})
&\ge
\frac{\sum_{u\in\mathcal{K}_{\bar\varepsilon}} d_u-\sum_{u\in U_{i-1}\cap\mathcal{K}_{\bar\varepsilon}} d_u}
{\sum_{u\in V} d_u}
\ge
\frac{C_1 n^{1-1.1\bar\varepsilon(\tau-2)}-(i-1)n^{3\bar\varepsilon}}
{\frac32\mu n},
\qquad \text{a.s.}
\label{eq_condprob_lower_multigraph}
\end{align}
Similarly,
\begin{align}
\mathbb P(B_i\mid \bar{\mathcal F}_{i-1})
&\le
\frac{\sum_{u\in\mathcal{K}_{\bar\varepsilon}} d_u}{\sum_{u\in V} d_u-(2i-1)}
\le
\frac{C_2 n^{1-1.1\bar\varepsilon(\tau-2)}}{\frac12\mu n-(2i-1)},
\qquad \text{a.s.}
\label{eq_condprob_upper_multigraph}
\end{align}
Since $i\le \lceil n^{2\tilde\varepsilon}\rceil$ and
$
2\tilde\varepsilon+3\bar\varepsilon
<
1-1.1\bar\varepsilon(\tau-2)
$, there exist constants $c_5,c_6>0$, depending only on $\tau$, such that for all sufficiently large $n$,
\begin{equation}\label{eq_condBi_bounds_multigraph}
c_5 n^{-1.1\bar\varepsilon(\tau-2)}
\le
\mathbb P(B_i\mid \bar{\mathcal F}_{i-1})
\le
c_6 n^{-1.1\bar\varepsilon(\tau-2)},
\qquad \text{a.s.}
\end{equation}
Now define
$$
p_i:=\mathbb P(B_i\mid \bar{\mathcal F}_{i-1}),
\quad
\xi_i:=\mathbf 1_{B_i}-p_i,
\quad
M_k:=\sum_{i=1}^k \xi_i, \quad V_k:=\sum_{i=1}^k \mathbb E(\xi_i^2\mid \bar{\mathcal F}_{i-1}),
\quad
1\le k\le \lceil n^{2\tilde\varepsilon}\rceil.
$$
Then $(M_k)_{k\ge 0}$ is a martingale with respect to $(\bar{\mathcal F}_k)$, and 
$
V_k
$ is its predictable quadratic variation. Since
$
\mathbb E(\xi_i^2\mid \bar{\mathcal F}_{i-1})
=
p_i(1-p_i)\le p_i
$,
we have
$$
V_k\le \sum_{i=1}^k p_i
\le
c_6 \lceil n^{2\tilde\varepsilon}\rceil n^{-1.1\bar\varepsilon(\tau-2)}
\qquad\text{a.s.}
$$
for every $1\le k\le \lceil n^{2\tilde\varepsilon}\rceil$, and also $|\xi_i|\le 1$ a.s.
We apply Freedman's inequality to the martingale $(-M_k)$; see \cite[Theorem~1.6]{Freedman1975} or \cite[Theorem~1.1]{Tropp2011Freedman}. For any $t,\sigma^2>0$,
\begin{equation}\label{eq_freedman_general_multigraph}
\mathbb P\Bigl(\exists\,k\le \lceil n^{2\tilde\varepsilon}\rceil: -M_k\ge t,\ V_k\le \sigma^2\Bigr)
\le
\exp\biggl(-\frac{t^2/2}{\sigma^2+t/3}\biggr).
\end{equation}
Taking
$
t=n^{\tilde\varepsilon},\sigma^2=c_6 \lceil n^{2\tilde\varepsilon}\rceil n^{-1.1\bar\varepsilon(\tau-2)},
$
we obtain, for all sufficiently large $n$,
\begin{equation}\label{eq_freedman_multigraph}
\mathbb P\bigl(M_{\lceil n^{2\tilde\varepsilon}\rceil}\le -n^{\tilde\varepsilon}\bigr)
\le
\exp\biggl(
-\frac{n^{2\tilde\varepsilon}/2}
{c_6 \lceil n^{2\tilde\varepsilon}\rceil n^{-1.1\bar\varepsilon(\tau-2)}+n^{\tilde\varepsilon}/3}
\biggr)
\le
\exp(-n^{\bar\varepsilon(\tau-2)}).
\end{equation}
On the event $\{M_{\lceil n^{2\tilde\varepsilon}\rceil}>-n^{\tilde\varepsilon}\}$,
\begin{align*}
\sum_{i=1}^{\lceil n^{2\tilde\varepsilon}\rceil}\mathbf 1_{B_i}
=
\sum_{i=1}^{\lceil n^{2\tilde\varepsilon}\rceil} p_i + M_{\lceil n^{2\tilde\varepsilon}\rceil}
\ge
\sum_{i=1}^{\lceil n^{2\tilde\varepsilon}\rceil} p_i - n^{\tilde\varepsilon}
\ge
c_5 \lceil n^{2\tilde\varepsilon}\rceil n^{-1.1\bar\varepsilon(\tau-2)}-n^{\tilde\varepsilon}.
\end{align*}
Because
$
2\tilde\varepsilon-1.1\bar\varepsilon(\tau-2)>\tilde\varepsilon
$, by the choice $\bar\varepsilon=\tilde\varepsilon/(4.4\tau)$, the right-hand side is at least $n^{\tilde\varepsilon}$ for all sufficiently large $n$. Therefore,
$$
\mathbb P\bigl(
\sum_{i=1}^{\lceil n^{2\tilde\varepsilon}\rceil}\mathbf 1_{B_i}\ge n^{\tilde\varepsilon}
\bigr)
\ge
1-\exp(-n^{\bar\varepsilon(\tau-2)}).
$$
Since we work throughout conditional on the degree sequence $\bm d$, and on the event $D_{\bar\varepsilon}$, combining this with \eqref{eq_Av_contains_Bi_new}, we conclude that
$$
\mathbb P(A_v\mid \bm d)\ge 1-\exp(-n^{\bar\varepsilon (\tau-2)})
$$
on the event $D_{\bar\varepsilon}$, which proves \eqref{eq_lemma_hierarchicalstarv11layer}.
\end{proof}

\begin{lemma}\label{lemma_hierarchicalstarvi1layer}
Fix any $\delta\in (0,1)$, and set
$
\tilde{\varepsilon}=\frac{\delta}{20\tau},
\bar\varepsilon=\frac{\tilde{\varepsilon}}{4.4\tau}
$.
Let $G=G(n,\tau)$ be a power-law random graph on $n$ vertices with exponent $\tau>2$, and let
$
\bm d=(d_1,\dots,d_n)
$
denote its degree sequence. Suppose that $v_1,v_2,\dots,v_{\lceil n^{1-\delta}\rceil}$ are vertices with degrees greater than $n^{2\tilde{\varepsilon}}$. Let $A$ denote the event that, for each $j=1,\dots,\lceil n^{1-\delta}\rceil$, the vertex $v_j$ has at least $\lceil n^{\tilde{\varepsilon}}\rceil$ distinct neighbors
$
u_1^{(j)},\dots,u_{\lceil n^{\tilde{\varepsilon}}\rceil}^{(j)}
$
such that all vertices in the collection
$
\{
u_i^{(j)}:
1\le j\le \lceil n^{1-\delta}\rceil,\,
1\le i\le \lceil n^{\tilde{\varepsilon}}\rceil
\}$
are distinct and have degrees between $n^{\bar\varepsilon}$ and $n^{3\bar\varepsilon}$. Then, for all sufficiently large $n$,
\begin{equation}
\mathbb{P}(A\mid \bm d)\ge 1-\exp(-n^{\bar\varepsilon (\tau-2)/2}), \qquad \text{on the event $D_{\bar\varepsilon}$ defined in \eqref{eq_D}. }
\end{equation}
\end{lemma}

\begin{proof}
We work throughout conditional on the degree sequence $\bm d$, and on the event $D_{\bar\varepsilon}$. Conditional on $\bm d$, the remaining randomness comes from the uniform pairing of half-edges in the configuration model.

For each vertex $u\in V$, let $\Omega_u:=\{(u,1),\dots,(u,d_u)\}$ be the set of half-edges incident to $u$. Fix a deterministic tie-breaking rule and write $x^{(1)},\dots,x^{(n)}$ for the vertices ordered so that $d_{x^{(1)}}\ge d_{x^{(2)}}\ge \cdots \ge d_{x^{(n)}}$. Define
$$
\mathcal K_{\bar\varepsilon}:=\Bigl\{x^{(i)}:\ i\in\bigl[\lfloor n^{1-2\bar\varepsilon(\tau-1)}\rfloor,\lfloor n^{1-1.1\bar\varepsilon(\tau-1)}\rfloor\bigr]\Bigr\}.
$$
By Lemma~\ref{lemma_degree_event_D}, on the event $D_{\bar\varepsilon}$, every vertex in $\mathcal K_{\bar\varepsilon}$ has degree in $(n^{\bar\varepsilon},n^{3\bar\varepsilon})$, and
\begin{equation}\label{eq_sumKeps_vi1layer_new}
C_1 n^{1-1.1\bar\varepsilon(\tau-2)}
\le
\sum_{u\in\mathcal K_{\bar\varepsilon}}d_u
\le
C_2 n^{1-1.1\bar\varepsilon(\tau-2)}
\end{equation}
for some constants $C_1,C_2>0$ depending only on $\tau$. Since each $v_\ell$ has degree greater than $n^{2\tilde\varepsilon}$ and $2\tilde\varepsilon>3\bar\varepsilon$, we also have $v_\ell\notin\mathcal K_{\bar\varepsilon}$ for every $\ell=1,\dots,\lceil n^{1-\delta}\rceil$, for all sufficiently large $n$.

For each $\ell=1,\dots,\lceil n^{1-\delta}\rceil$, we expose the first $\lceil n^{2\tilde\varepsilon}\rceil$ half-edges of $v_\ell$, namely 
$
(v_\ell,1),\dots,(v_\ell,\lceil n^{2\tilde\varepsilon}\rceil)
$.
Since conditional on $\bm d$ the pairing is uniform, the order in which we expose the half-edges is irrelevant. Hence we may first expose the first $\lceil n^{2\tilde\varepsilon}\rceil$ half-edges of $v_1$, then the first $\lceil n^{2\tilde\varepsilon}\rceil$ half-edges of $v_2$, and so on, up to $v_{\lceil n^{1-\delta}\rceil}$.

For each $\ell=1,\dots,\lceil n^{1-\delta}\rceil$, let $e_1^{(\ell)},\dots,e_{\lceil n^{2\tilde\varepsilon}\rceil}^{(\ell)}$ denote the partner half-edges paired to \\
$(v_\ell,1),\dots,(v_\ell,\lceil n^{2\tilde\varepsilon}\rceil)$, and let $u_i^{(\ell)}$ be the vertex incident to $e_i^{(\ell)}$, that is, $e_i^{(\ell)}\in\Omega_{u_i^{(\ell)}}$. Define
$$
U_{i-1}^{(\ell)}:=\{u_1^{(\ell)},\dots,u_{i-1}^{(\ell)}\},\qquad 1\le i\le \lceil n^{2\tilde\varepsilon}\rceil.
$$
For each $\ell\ge2$, let
$$
\mathcal K_{\bar\varepsilon}^{(\ell)}:=\mathcal K_{\bar\varepsilon}\setminus\Bigl\{u_i^{(k)}:1\le k\le \ell-1,\ 1\le i\le \lceil n^{2\tilde\varepsilon}\rceil\Bigr\},
$$
and set $\mathcal K_{\bar\varepsilon}^{(1)}:=\mathcal K_{\bar\varepsilon}$.
Let
$$
E^{(<\ell)}:=\bigcup_{k=1}^{\ell-1}\Bigl(\{(v_k,1),\dots,(v_k,\lceil n^{2\tilde\varepsilon}\rceil)\}\cup\{e_1^{(k)},\dots,e_{\lceil n^{2\tilde\varepsilon}\rceil}^{(k)}\}\Bigr)
$$
be the set of all half-edges already used before stage $\ell$. Then
\begin{equation}\label{eq_size_Eless_ell_multigraph_new}
|E^{(<\ell)}|=2(\ell-1)\lceil n^{2\tilde\varepsilon}\rceil.
\end{equation}
For $i=0,1,\dots,\lceil n^{2\tilde\varepsilon}\rceil$, define the filtration
$$
\bar{\mathcal F}_0^{(\ell)}:=\sigma(E^{(<\ell)}),\qquad \bar{\mathcal F}_i^{(\ell)}:=\sigma(E^{(<\ell)},e_1^{(\ell)},\dots,e_i^{(\ell)}).
$$
For $1\le i\le \lceil n^{2\tilde\varepsilon}\rceil$, let
$$
B_i^{(\ell)}:=\{u_i^{(\ell)}\in\mathcal K_{\bar\varepsilon}^{(\ell)}\setminus U_{i-1}^{(\ell)}\}.
$$
Then
\begin{equation}\label{eq_Avell_contains_Bi_multigraph_new}
 \bigcap_{\ell=1}^{\lceil n^{1-\delta}\rceil}\biggl\{\sum_{i=1}^{\lceil n^{2\tilde\varepsilon}\rceil}\mathbf 1_{B_i^{(\ell)}}\ge n^{\tilde\varepsilon}\biggr\}\subseteq A .
\end{equation}
Fix $\ell\in\{1,\dots,\lceil n^{1-\delta}\rceil\}$. Conditional on $\bar{\mathcal F}_0^{(\ell)}$, the half-edge $(v_\ell,i)$ is paired uniformly to one of the remaining half-edges at step $i$. Since before step $i$ at stage $\ell$ exactly $|E^{(<\ell)}|+2(i-1)$ half-edges have already been used, we have
\begin{equation}\label{eq_condprob_Biell_multigraph_new}
\mathbb P\bigl(B_i^{(\ell)}\mid \bar{\mathcal F}_{i-1}^{(\ell)}\bigr)
=
\frac{\sum_{u\in \mathcal K_{\bar\varepsilon}^{(\ell)}\setminus U_{i-1}^{(\ell)}} d_u}
{\sum_{u\in V}d_u-|E^{(<\ell)}|-(2i-1)},
\qquad \text{a.s.}
\end{equation}
We now derive uniform lower and upper bounds on these conditional probabilities. Since all vertices in $\mathcal K_{\bar\varepsilon}$ have degree at most $n^{3\bar\varepsilon}$ on $D_{\bar\varepsilon}$, we have
\begin{equation}\label{eq_removed_degree_mass_multigraph_new}
\sum_{u\in\mathcal K_{\bar\varepsilon}\setminus\mathcal K_{\bar\varepsilon}^{(\ell)}}d_u\le(\ell-1)\lceil n^{2\tilde\varepsilon}\rceil n^{3\bar\varepsilon},\qquad \sum_{u\in U_{i-1}^{(\ell)}\cap\mathcal K_{\bar\varepsilon}^{(\ell)}}d_u\le(i-1)n^{3\bar\varepsilon}.
\end{equation}
Hence, by \eqref{eq_sumKeps_vi1layer_new},
\begin{align}
\sum_{u\in \mathcal K_{\bar\varepsilon}^{(\ell)}\setminus U_{i-1}^{(\ell)}}d_u
&\ge C_1 n^{1-1.1\bar\varepsilon(\tau-2)}-(\ell-1)\lceil n^{2\tilde\varepsilon}\rceil n^{3\bar\varepsilon}-(i-1)n^{3\bar\varepsilon}.
\label{eq_numerator_lower_multigraph_new}
\end{align}
Also, on $D_{\bar\varepsilon}$,
$
\frac12\mu n\le \sum_{u\in V}d_u\le \frac32\mu n
$.
Combining this with \eqref{eq_size_Eless_ell_multigraph_new}, we obtain
\begin{equation}\label{eq_denominator_bounds_multigraph_new}
\frac12\mu n-2(\ell-1)\lceil n^{2\tilde\varepsilon}\rceil-(2i-1)\le \sum_{u\in V}d_u-|E^{(<\ell)}|-(2i-1)\le \frac32\mu n.
\end{equation}
Therefore, by \eqref{eq_condprob_Biell_multigraph_new}, \eqref{eq_numerator_lower_multigraph_new}, and \eqref{eq_denominator_bounds_multigraph_new},
\begin{align}
\mathbb P\bigl(B_i^{(\ell)}\mid \bar{\mathcal F}_{i-1}^{(\ell)}\bigr)
&\ge
\frac{C_1 n^{1-1.1\bar\varepsilon(\tau-2)}-(\ell-1)\lceil n^{2\tilde\varepsilon}\rceil n^{3\bar\varepsilon}-(i-1)n^{3\bar\varepsilon}}{\frac32\mu n},
\qquad \text{a.s.},
\label{eq_condprob_lower_Biell_multigraph_new}
\end{align}
and similarly,
\begin{align}
\mathbb P\bigl(B_i^{(\ell)}\mid \bar{\mathcal F}_{i-1}^{(\ell)}\bigr)
&\le
\frac{C_2 n^{1-1.1\bar\varepsilon(\tau-2)}}{\frac12\mu n-2(\ell-1)\lceil n^{2\tilde\varepsilon}\rceil-(2i-1)},
\qquad \text{a.s.}
\label{eq_condprob_upper_Biell_multigraph_new}
\end{align}
Since $\ell\le \lceil n^{1-\delta}\rceil$, $i\le \lceil n^{2\tilde\varepsilon}\rceil$, $1-\delta+2\tilde\varepsilon+3\bar\varepsilon<1-1.1\bar\varepsilon(\tau-2)$ and $2\tilde\varepsilon<1$, there exist constants $c_5,c_6>0$, depending only on $\tau$, such that for all sufficiently large $n$,
\begin{equation}\label{eq_condprob_uniform_Biell_multigraph_new}
c_5 n^{-1.1\bar\varepsilon(\tau-2)}\le \mathbb P\bigl(B_i^{(\ell)}\mid \bar{\mathcal F}_{i-1}^{(\ell)}\bigr)\le c_6 n^{-1.1\bar\varepsilon(\tau-2)},\qquad \text{a.s.}
\end{equation}
Define
$$
p_i^{(\ell)}:=\mathbb P\bigl(B_i^{(\ell)}\mid \bar{\mathcal F}_{i-1}^{(\ell)}\bigr),\qquad \xi_i^{(\ell)}:=\mathbf 1_{B_i^{(\ell)}}-p_i^{(\ell)},
$$
and
$$
M_k^{(\ell)}:=\sum_{i=1}^k\xi_i^{(\ell)},\qquad V_k^{(\ell)}:=\sum_{i=1}^k\mathbb E\bigl((\xi_i^{(\ell)})^2\mid \bar{\mathcal F}_{i-1}^{(\ell)}\bigr),\qquad 1\le k\le \lceil n^{2\tilde\varepsilon}\rceil.
$$
Then $(M_k^{(\ell)})$ is a martingale with respect to $(\bar{\mathcal F}_k^{(\ell)})$, and
$
\mathbb E\bigl((\xi_i^{(\ell)})^2\mid \bar{\mathcal F}_{i-1}^{(\ell)}\bigr)=p_i^{(\ell)}(1-p_i^{(\ell)})\le p_i^{(\ell)}
$.
Hence,
$$
V_k^{(\ell)}\le \sum_{i=1}^k p_i^{(\ell)}\le c_6\lceil n^{2\tilde\varepsilon}\rceil n^{-1.1\bar\varepsilon(\tau-2)}\qquad\text{a.s.}
$$
for every $1\le k\le \lceil n^{2\tilde\varepsilon}\rceil$, and also $|\xi_i^{(\ell)}|\le 1$ a.s.
Applying Freedman's inequality to $(-M_k^{(\ell)})$ with $t=n^{\tilde\varepsilon}$ and $\sigma^2=c_6\lceil n^{2\tilde\varepsilon}\rceil n^{-1.1\bar\varepsilon(\tau-2)}$, see \cite[Theorem~1.6]{Freedman1975} or \cite[Theorem~1.1]{Tropp2011Freedman}, we obtain
\begin{equation}\label{eq_freedman_Biell_multigraph_new}
\mathbb P\Bigl(M_{\lceil n^{2\tilde\varepsilon}\rceil}^{(\ell)}\le -n^{\tilde\varepsilon}\Bigm| \bar{\mathcal F}_0^{(\ell)}\Bigr)\le \exp\biggl(-\frac{t^2/2}{\sigma^2+t/3}\biggr)  \le  \exp(-n^{\bar\varepsilon(\tau-2)})
\end{equation}
for all sufficiently large $n$.
On the event $\{M_{\lceil n^{2\tilde\varepsilon}\rceil}^{(\ell)}>-n^{\tilde\varepsilon}\}$,
\begin{align*}
\sum_{i=1}^{\lceil n^{2\tilde\varepsilon}\rceil}\mathbf 1_{B_i^{(\ell)}}=\sum_{i=1}^{\lceil n^{2\tilde\varepsilon}\rceil}p_i^{(\ell)}+M_{\lceil n^{2\tilde\varepsilon}\rceil}^{(\ell)}
\ge \sum_{i=1}^{\lceil n^{2\tilde\varepsilon}\rceil}p_i^{(\ell)}-n^{\tilde\varepsilon}
\ge c_5\lceil n^{2\tilde\varepsilon}\rceil n^{-1.1\bar\varepsilon(\tau-2)}-n^{\tilde\varepsilon}.
\end{align*}
Since $2\tilde\varepsilon-1.1\bar\varepsilon(\tau-2)>\tilde\varepsilon$, the right-hand side is at least $n^{\tilde\varepsilon}$ for all sufficiently large $n$. Therefore,
\begin{equation}\label{eq_stage_success_multigraph_new}
\mathbb P\biggl(\sum_{i=1}^{\lceil n^{2\tilde\varepsilon}\rceil}\mathbf 1_{B_i^{(\ell)}}\ge n^{\tilde\varepsilon}\Bigm|\bar{\mathcal F}_0^{(\ell)}\biggr)\ge 1-\exp(-n^{\bar\varepsilon(\tau-2)}).
\end{equation}
Finally, combining \eqref{eq_Avell_contains_Bi_multigraph_new} and \eqref{eq_stage_success_multigraph_new}, and recalling that we work conditional on $\bm d$ and on $D_{\bar\varepsilon}$, we obtain
\[
\mathbb P(A\mid \bm d)\ge 1-\lceil n^{1-\delta}\rceil\exp(-n^{\bar\varepsilon(\tau-2)})\ge 1-\exp(-n^{\bar\varepsilon(\tau-2)/2})
\]
for all sufficiently large $n$. This completes the proof.
\end{proof}

Lemma~\ref{lemma_hierarchicalstar_second_layer} upgrades the first-layer construction to the second-layer. This yields the full two-layer geometry needed for Proposition~\ref{prop:existsmanyhierarchicalstar}.

\begin{lemma}\label{lemma_hierarchicalstar_second_layer}
Fix any $\delta\in(0,1)$, and set
$
\tilde{\varepsilon}:=\frac{\delta}{20\tau},
\bar\varepsilon:=\frac{\tilde{\varepsilon}}{4.4\tau},
\varepsilon:=\frac{\bar\varepsilon}{4\tau}
$.
Let $G=G(n,\tau)$ be a power-law random graph on $n$ vertices with exponent $\tau>2$, and let
$
\bm d=(d_1,\dots,d_n)
$
denote its degree sequence. Suppose that $v_1,\dots,v_{\lceil n^{1-\delta}\rceil}$ are vertices with degrees greater than $n^{2\tilde\varepsilon}$. Let $\tilde A$ denote the event that, for each $a=1,\dots,\lceil n^{1-\delta}\rceil$, the vertex $v_a$ has at least $\lceil n^{\tilde\varepsilon}\rceil$ pairwise distinct neighbors
$
u_1^{(a)},\dots,u_{\lceil n^{\tilde\varepsilon}\rceil}^{(a)}
$
such that all vertices in the collection
$
\{u_b^{(a)}:1\le a\le \lceil n^{1-\delta}\rceil ,\ 1\le b\le \lceil n^{\tilde\varepsilon}\rceil\}
$
are distinct, and have degrees between $n^{\bar\varepsilon}$ and $n^{3\bar\varepsilon}
$. Moreover, for each $a=1,\dots,\lceil n^{1-\delta}\rceil$ and each $b=1,\dots, \lceil n^{\tilde\varepsilon}\rceil$, the vertex $u_b^{(a)}$ has at least $ \lceil n^{\varepsilon}\rceil$ distinct neighbors
$
w_{a,b,1},\dots,w_{a,b,\lceil n^{\varepsilon}\rceil}
$
such that all vertices in the collection
$
\{w_{a,b,m}:1\le a\le\lceil n^{1-\delta}\rceil,\ 1\le b\le \lceil n^{\tilde\varepsilon}\rceil,\ 1\le m\le\lceil n^{\varepsilon}\rceil\}
$
are distinct and have degrees between
$n^{\varepsilon}$ and $n^{3\varepsilon}$. Then, for all sufficiently large $n$,
\begin{equation}
\mathbb P(\tilde A\mid \bm d)\ge 1-\exp(-n^{\bar\varepsilon(\tau-2)/2})-\exp(-n^{\varepsilon/2}), \qquad \text{on the event $D_{\varepsilon}$ defined in \eqref{eq_D}. }
\end{equation}
\end{lemma}

\begin{proof}
We work throughout conditional on the degree sequence $\bm d$ and on the event $D_{\varepsilon}$. Since $\varepsilon<\bar\varepsilon$, the event $D_{\varepsilon}$ implies the corresponding degree regularity estimates needed at scale $\bar\varepsilon$. Hence Lemma~\ref{lemma_hierarchicalstarvi1layer} applies. Let $A$ be the first-layer event from Lemma~\ref{lemma_hierarchicalstarvi1layer}, namely the event that, for each $a=1,\dots,\lceil n^{1-\delta}\rceil$, the vertex $v_a$ has at least $\lceil n^{\tilde\varepsilon}\rceil$ pairwise distinct neighbors $u_1^{(a)},\dots,u_{\lceil n^{\tilde\varepsilon}\rceil}^{(a)}$ such that all vertices in the collection $\{u_b^{(a)}:1\le a\le \lceil n^{1-\delta}\rceil,\ 1\le b\le \lceil n^{\tilde\varepsilon}\rceil\}$ are distinct and have degrees between $n^{\bar\varepsilon}$ and $n^{3\bar\varepsilon}$. By Lemma~\ref{lemma_hierarchicalstarvi1layer},
\begin{equation}\label{eq_first_layer_success_second}
\mathbb P(A\mid \bm d)\ge 1-\exp(-n^{\bar\varepsilon(\tau-2)/2})
\end{equation}
on the event $D_{\varepsilon}$.

On $A$, choose the vertices $u_b^{(a)}$ according to a fixed deterministic rule. We now expose the second layer. Let $F^{(1)}$ denote the set of half-edges already used in the first-layer exploration. Since in the proof of Lemma~\ref{lemma_hierarchicalstarvi1layer} we expose $\lceil n^{2\tilde\varepsilon}\rceil$ pairings for each of the $\lceil n^{1-\delta}\rceil$ vertices $v_a$, we have
\begin{equation}\label{eq_first_layer_used_second_new}
|F^{(1)}|=2\lceil n^{1-\delta}\rceil\lceil n^{2\tilde\varepsilon}\rceil.
\end{equation}

Fix a deterministic tie-breaking rule and write $x^{(1)},\dots,x^{(n)}$ for the vertices ordered so that $d_{x^{(1)}}\ge d_{x^{(2)}}\ge \cdots \ge d_{x^{(n)}}$. Define
$$
\mathcal K_{\varepsilon}:=\Bigl\{x^{(i)}:\ i\in\bigl[\lfloor n^{1-2\varepsilon(\tau-1)}\rfloor,\lfloor n^{1-1.1\varepsilon(\tau-1)}\rfloor\bigr]\Bigr\}.
$$
By Lemma~\ref{lemma_degree_event_D}, on the event $D_{\varepsilon}$, every vertex in $\mathcal K_{\varepsilon}$ has degree in $(n^{\varepsilon},n^{3\varepsilon})$. Since each $v_a$ has degree greater than $n^{2\tilde\varepsilon}$ and each $u_b^{(a)}$ has degree greater than $n^{\bar\varepsilon}$, while $2\tilde\varepsilon>3\varepsilon$ and $\bar\varepsilon>3\varepsilon$, we have $v_a\notin\mathcal K_{\varepsilon}$ and $u_b^{(a)}\notin\mathcal K_{\varepsilon}$ for all relevant $a,b$, for all sufficiently large $n$. Moreover,
\begin{equation}\label{eq_sumKbar_new}
C_1 n^{1-1.1\varepsilon(\tau-2)}
\le
\sum_{x\in\mathcal K_{\varepsilon}}d_x
\le
C_2 n^{1-1.1\varepsilon(\tau-2)}
\end{equation}
for some constants $C_1,C_2>0$ depending only on $\tau$.

Order the pairs $(a,b)$ lexicographically and write
$$
(a_1,b_1),\dots,(a_N,b_N),\qquad N:=\lceil n^{1-\delta}\rceil\lceil n^{\tilde\varepsilon}\rceil.
$$
For each $r=1,\dots,N$, let $\hat u_r:=u_{b_r}^{(a_r)}$. Since $\hat u_r$ has degree greater than $n^{\bar\varepsilon}$ and one of its half-edges has already been paired to $v_{a_r}$, there are at least $\lfloor n^{\bar\varepsilon}\rfloor$ remaining half-edges of $\hat u_r$. We expose $\lfloor n^{\bar\varepsilon}\rfloor$ such half-edges, in a fixed deterministic order, and denote them by $(\hat u_r,1),\dots,(\hat u_r,\lfloor n^{\bar\varepsilon}\rfloor)$. For stage $r$, let $e_s^{(r)}$ be the partner half-edge paired to $(\hat u_r,s)$, and let $y_s^{(r)}$ be the vertex incident to $e_s^{(r)}$, that is, $e_s^{(r)}\in\Omega_{y_s^{(r)}}$. Define
$$
U_{s-1}^{(r)}:=\{y_1^{(r)},\dots,y_{s-1}^{(r)}\},\qquad 1\le s\le \lfloor n^{\bar\varepsilon}\rfloor.
$$
Also let
$$
Z_{r-1}:=\{y_t^{(q)}:1\le q\le r-1,\ 1\le t\le \lfloor n^{\bar\varepsilon}\rfloor,\ y_t^{(q)}\in\mathcal K_{\varepsilon}\}
$$
be the set of vertices in $\mathcal K_{\varepsilon}$ already hit in previous second-layer stages. Let
$$
E^{(<r)}:=\bigcup_{q=1}^{r-1}\Bigl(\{(\hat u_q,1),\dots,(\hat u_q,\lfloor n^{\bar\varepsilon}\rfloor)\}\cup\{e_1^{(q)},\dots,e_{\lfloor n^{\bar\varepsilon}\rfloor}^{(q)}\}\Bigr)
$$
be the set of all half-edges already used before stage $r$. Then
\begin{equation}\label{eq_prev_second_halfedges_new}
|E^{(<r)}|=2(r-1)\lfloor n^{\bar\varepsilon}\rfloor.
\end{equation}
Define the filtration
$$
\bar{\mathcal F}_0^{(r)}:=\sigma(F^{(1)},E^{(<r)}),\qquad \bar{\mathcal F}_s^{(r)}:=\sigma(F^{(1)},E^{(<r)},e_1^{(r)},\dots,e_s^{(r)}),\qquad 1\le s\le \lfloor n^{\bar\varepsilon}\rfloor.
$$
For each $s=1,\dots,\lfloor n^{\bar\varepsilon}\rfloor$, define
$$
B_s^{(r)}:=\bigl\{y_s^{(r)}\in\mathcal K_{\varepsilon}\setminus(Z_{r-1}\cup U_{s-1}^{(r)})\bigr\}.
$$
Thus $B_s^{(r)}$ is the event that, at step $s$ of stage $r$, the newly exposed partner half-edge belongs to a fresh vertex in $\mathcal K_{\varepsilon}$. Therefore,
\begin{equation}\label{eq_tildeA_contains_B}
A\cap\bigcap_{r=1}^{N}\biggl\{\sum_{s=1}^{\lfloor n^{\bar\varepsilon}\rfloor}\mathbf 1\{B_s^{(r)}\}\ge n^{\varepsilon}\biggr\}\subseteq \tilde A.
\end{equation}
Conditional on $\bar{\mathcal F}_{s-1}^{(r)}$, the half-edge $(\hat u_r,s)$ is paired uniformly to one of the remaining half-edges. Therefore,
\begin{equation}\label{eq_second_layer_condprob_new}
\mathbb P(B_s^{(r)}\mid\bar{\mathcal F}_{s-1}^{(r)})
=
\frac{\sum_{x\in\mathcal K_{\varepsilon}\setminus(Z_{r-1}\cup U_{s-1}^{(r)})}d_x}
{\sum_{x\in V}d_x-|F^{(1)}|-|E^{(<r)}|-(2s-1)},
\qquad \text{a.s.}
\end{equation}
We now bound the numerator from below. Since every vertex in $\mathcal K_{\varepsilon}$ has degree at most $n^{3\varepsilon}$ on $D_{\varepsilon}$, we have
$$
\sum_{x\in Z_{r-1}}d_x\le (r-1)\lfloor n^{\bar\varepsilon}\rfloor n^{3\varepsilon},\qquad \sum_{x\in U_{s-1}^{(r)}\cap\mathcal K_{\varepsilon}}d_x\le (s-1)n^{3\varepsilon}.
$$
Hence, by \eqref{eq_sumKbar_new},
\begin{align}
\sum_{x\in\mathcal K_{\varepsilon}\setminus(Z_{r-1}\cup U_{s-1}^{(r)})}d_x
&\ge C_1 n^{1-1.1\varepsilon(\tau-2)}-(r-1)\lfloor n^{\bar\varepsilon}\rfloor n^{3\varepsilon}-(s-1)n^{3\varepsilon},
\label{eq_second_num_lower_new}
\end{align}
and also this numerator is at most $C_2 n^{1-1.1\varepsilon(\tau-2)}$. On $D_{\varepsilon}$, $\frac12\mu n\le \sum_{x\in V}d_x\le \frac32\mu n$. Combining this with \eqref{eq_first_layer_used_second_new} and \eqref{eq_prev_second_halfedges_new}, we get
\begin{equation}\label{eq_second_den_bounds_new}
\frac12\mu n-2\lceil n^{1-\delta}\rceil\lceil n^{2\tilde\varepsilon}\rceil-2(r-1)\lfloor n^{\bar\varepsilon}\rfloor-(2s-1)
\le
\sum_{x\in V}d_x-|F^{(1)}|-|E^{(<r)}|-(2s-1)
\le
\frac32\mu n.
\end{equation}
Since $r\le N=\lceil n^{1-\delta}\rceil\lceil n^{\tilde\varepsilon}\rceil$, $s\le \lfloor n^{\bar\varepsilon}\rfloor$, and
$
1-\delta+\tilde\varepsilon+\bar\varepsilon+3\varepsilon<1-1.1\varepsilon(\tau-2), 1-\delta+2\tilde\varepsilon<1
$,
there exist constants $c_5,c_6>0$, depending only on $\tau$, such that for all sufficiently large $n$,
\begin{equation}\label{eq_second_uniform_p_new}
c_5 n^{-1.1\varepsilon(\tau-2)}
\le
\mathbb P(B_s^{(r)}\mid \bar{\mathcal F}_{s-1}^{(r)})
\le
c_6 n^{-1.1\varepsilon(\tau-2)},
\qquad \text{a.s.}
\end{equation}
Now define
$$
p_s^{(r)}:=\mathbb P(B_s^{(r)}\mid\bar{\mathcal F}_{s-1}^{(r)}),\qquad \xi_s^{(r)}:=\mathbf 1_{B_s^{(r)}}-p_s^{(r)},
$$
and
$$
M_k^{(r)}:=\sum_{s=1}^k\xi_s^{(r)},\qquad V_k^{(r)}:=\sum_{s=1}^k\mathbb E\bigl((\xi_s^{(r)})^2\mid\bar{\mathcal F}_{s-1}^{(r)}\bigr),\qquad 1\le k\le \lfloor n^{\bar\varepsilon}\rfloor.
$$
Then $(M_k^{(r)})$ is a martingale with respect to $(\bar{\mathcal F}_k^{(r)})$, and
$
V_k^{(r)}\le \sum_{s=1}^k p_s^{(r)}\le c_6\lfloor n^{\bar\varepsilon}\rfloor n^{-1.1\varepsilon(\tau-2)}
$a.s.
Let
$
\mu_r:=\sum_{s=1}^{\lfloor n^{\bar\varepsilon}\rfloor}p_s^{(r)}
$.
By \eqref{eq_second_uniform_p_new},
$
\mu_r\in\Big(c_5\lfloor n^{\bar\varepsilon}\rfloor n^{-1.1\varepsilon(\tau-2)},\ c_6\lfloor n^{\bar\varepsilon}\rfloor n^{-1.1\varepsilon(\tau-2)}\Big)
$.
We apply Freedman's inequality to the martingale $(-M_k^{(r)})$, see \cite[Theorem~1.6]{Freedman1975} or \cite[Theorem~1.1]{Tropp2011Freedman}. Since $|\xi_s^{(r)}|\le1$ and $V_{\lfloor n^{\bar\varepsilon}\rfloor}^{(r)}\le c_6\lfloor n^{\bar\varepsilon}\rfloor n^{-1.1\varepsilon(\tau-2)}$ a.s., taking
$
t:=\frac12\mu_r, \sigma^2:=c_6\lfloor n^{\bar\varepsilon}\rfloor n^{-1.1\varepsilon(\tau-2)}
$,
gives
\begin{align}
\mathbb P\Bigl(M_{\lfloor n^{\bar\varepsilon}\rfloor}^{(r)}\le -\tfrac12\mu_r\Bigm|\bar{\mathcal F}_0^{(r)}\Bigr)
&\le
\exp\biggl(-\frac{t^2/2}{\sigma^2+t/3}\biggr)
\le \exp(-n^{\varepsilon})
\label{eq_second_stage_fail_new}
\end{align}
for all sufficiently large $n$, because $\bar\varepsilon-1.1\varepsilon(\tau-2)>\varepsilon$. On the complement of the event in \eqref{eq_second_stage_fail_new},
\begin{align*}
\sum_{s=1}^{\lfloor n^{\bar\varepsilon}\rfloor}\mathbf 1\{B_s^{(r)}\}
&=
\mu_r+M_{\lfloor n^{\bar\varepsilon}\rfloor}^{(r)}
\ge
\frac12\mu_r
\ge
n^{\varepsilon}
\end{align*}
for all sufficiently large $n$. Therefore, for each stage $r$,
\begin{equation}\label{eq_second_stage_success_new}
\mathbb P\Bigl(\sum_{s=1}^{\lfloor n^{\bar\varepsilon}\rfloor}\mathbf 1\{B_s^{(r)}\}\ge n^{\varepsilon}\ \Bigm|\bar{\mathcal F}_0^{(r)}\Bigr)\ge 1-\exp(-n^{\varepsilon}).
\end{equation}
Finally, taking a union bound over the $N=\lceil n^{1-\delta}\rceil\lceil n^{\tilde\varepsilon}\rceil$ stages, and recalling that we work conditional on $\bm d$ and on $D_{\bar\varepsilon}$, we obtain
\begin{align}
\mathbb P(\tilde A\mid A,\bm d)
&\ge
\mathbb P\Bigl(\bigcap_{r=1}^{N}\Bigl\{\sum_{s=1}^{\lfloor n^{\bar\varepsilon}\rfloor}\mathbf 1\{B_s^{(r)}\}\ge n^{\varepsilon}\Bigr\}\ \Bigm|\ A,\bm d\Bigr)
\notag\\
&\ge
1-N\exp(-n^{\varepsilon})
\ge
1-\exp(-n^{\varepsilon/2})
\label{eq_second_all_success_new}
\end{align}
for all sufficiently large $n$. Combining \eqref{eq_first_layer_success_second} with \eqref{eq_second_all_success_new}, and recalling that we work throughout conditional on $\bm d$ and on $D_{\varepsilon}$, we conclude that
$$
\mathbb P(\tilde A\mid\bm d)\ge 1-\exp(-n^{\bar\varepsilon(\tau-2)/2})-\exp(-n^{\varepsilon/2})
$$
for all sufficiently large $n$. This completes the proof.
\end{proof}

\begin{proof}[Proof of Proposition \ref{prop:existsmanyhierarchicalstar}]
Define
$
\tilde{\varepsilon}:=\frac{\delta}{20\tau}, \bar{\varepsilon}:=\frac{\tilde{\varepsilon}}{4.4\tau}$. 
Then $
\varepsilon<\frac{\bar{\varepsilon}}{4\tau}
$.
It therefore suffices to prove the proposition with the above choice of
$\tilde{\varepsilon},\bar{\varepsilon},\varepsilon$, since any $(n^{\bar\varepsilon},n^{\bar{\varepsilon}/(4\tau)})$ hierarchical-star is in particular an $(n^\varepsilon,n^\varepsilon)$ hierarchical-star.
Moreover, with these choices of parameters, all assumptions of Lemma~\ref{lemma_hierarchicalstar_second_layer} are satisfied. We now work on the event $D_{\varepsilon}$ from Lemma~\ref{lemma_degree_event_D}.

We first show that, on the event $D_{\varepsilon}$, there exist at least $\lceil n^{1-\delta}\rceil$ vertices of degree greater than $n^{2\tilde\varepsilon}$. Since $\lceil n^{1-\delta}\rceil\in [n^\varepsilon,n]$, by Lemma~\ref{lemma_degree_event_D}, $d^{(\lceil n^{1-\delta}\rceil)}
\ge
c_1\left(\frac{n}{n^{1-\delta}}\right)^{\frac{1}{\tau-1}}
=
c_1 n^{\frac{\delta}{\tau-1}}$
Since
$
\frac{\delta}{\tau-1}>2\tilde{\varepsilon}
$,
it follows that
$$
d^{(\lceil n^{1-\delta}\rceil)}>n^{2\tilde{\varepsilon}}
$$
for all sufficiently large $n$. Hence, on the event $D_{\varepsilon}$, there exist at least $\lceil n^{1-\delta}\rceil$ vertices with degree greater than $n^{2\tilde{\varepsilon}}$. Choose and fix any such collection, and denote it by
$
v_1,\dots,v_{\lceil n^{1-\delta}\rceil}
$.

Now we apply Lemma~\ref{lemma_hierarchicalstar_second_layer}. On the event $D_{\varepsilon}$, with conditional probability at least
$
1-\exp(-n^{\bar\varepsilon(\tau-2)/2})-\exp(-n^{\varepsilon/2})
$,
for each $a=1,\dots,\lceil n^{1-\delta}\rceil$, the vertex $v_a$ has at least $n^{\tilde{\varepsilon}}$ distinct neighbors
$
u_1^{(a)},\dots,u_{n^{\tilde{\varepsilon}}}^{(a)}
$,
such that all vertices in the collection
$
\{u_b^{(a)}:1\le a\le \lceil n^{1-\delta}\rceil,\ 1\le b\le \lceil n^{\tilde{\varepsilon}}\rceil\}
$
are distinct and have degrees between $n^{\bar\varepsilon}$ and $n^{3\bar\varepsilon}$. Moreover, for each $a=1,\dots,\lceil n^{1-\delta}\rceil$ and each $b=1,\dots,\lceil n^{\tilde{\varepsilon}}\rceil$, the vertex $u_b^{(a)}$ has at least $n^{\varepsilon}$ distinct neighbors
$
w_{a,b,1},\dots,w_{a,b,\lceil n^{\varepsilon}\rceil}
$,
such that all vertices in the collection
$$
\{w_{a,b,m}:1\le a\le \lceil n^{1-\delta}\rceil,\ 1\le b\le \lceil n^{\tilde{\varepsilon}}\rceil,\ 1\le m\le \lceil n^{\varepsilon}\rceil\}
$$
are distinct.

Hence, on this event, for each $a=1,\dots,\lceil n^{1-\delta}\rceil$, the vertex $v_a$ is the center of a hierarchical-star with first layer
$
\{u_1^{(a)},\dots,u_{\lceil n^{\tilde{\varepsilon}}\rceil}^{(a)}\}
$
and second layer
$
\{w_{a,b,m}:1\le b\le\lceil n^{\tilde{\varepsilon}}\rceil,\ 1\le m\le \lceil n^{\varepsilon}\rceil\}
$.
Moreover, by construction, all these hierarchical-stars are vertex-disjoint. Since $\tilde{\varepsilon}>\varepsilon$, each such hierarchical-star is in particular an $(n^\varepsilon,n^\varepsilon)$ hierarchical-star.

Therefore,
\begin{align*}
\mathbb P\bigl(\text{$G$ contains at least $\lceil n^{1-\delta}\rceil$ vertex-disjoint $(n^\varepsilon,n^\varepsilon)$ hierarchical-stars}\bigr)\\
\ge
\mathbb P(D_{\varepsilon})\Bigl(1-\exp(-n^{\bar\varepsilon (\tau-2)})-\exp(-n^{\varepsilon/2})\Bigr).
\end{align*}
By Lemma~\ref{lemma_degree_event_D},
$
\mathbb P(D_{\varepsilon})\ge 1-c'(\log n)^{\mathbf{1}\{\tau=3\}}n^{-\min\{\tau-2,1\}}
$
for all sufficiently large $n$. Hence, 
$$
\mathbb P\bigl(\text{$G$ contains at least $\lceil n^{1-\delta}\rceil$ vertex-disjoint $(n^\varepsilon,n^\varepsilon)$ hierarchical-stars}\bigr)
\ge
1-2c'(\log n)^{\mathbf{1}\{\tau=3\}}\,n^{-\min\{\tau-2,1\}}
$$
for all sufficiently large $n$. This completes the proof.
\end{proof}

We next turn to two auxiliary lemmas that will be used later in the proof of Theorem~\ref{thm_main2}.

\begin{lemma}\label{lemma_Mhierarchicalstar}
Let $G=G(n,\tau)$ be a power-law random graph on $n$ vertices with exponent $\tau>2$. Then there exist constants
$
M^{(0)},
\alpha_1,\alpha_2>0
$
depending only on $\tau$ such that the following holds. For any $M\ge M^{(0)}$, if $v$ is a vertex with degree at least $M$, then, for all sufficiently large $n$,
\[
\mathbb P\Bigl(v \text{ is the center of a hierarchical-star with parameters }(M^{\alpha_1},M^{\alpha_1})\Bigr)
\ge 1-\exp(-M^{\alpha_2}).
\]
\end{lemma}

\begin{proof}
We fix
\(
\alpha_1:=\frac{1}{4(\tau-1)},\alpha_2:=\frac{1}{8(\tau-1)}.
\)
Let
\[
I_M:=[2M^{\alpha_1},\,4M^{\alpha_1}],
\qquad
\mathcal K_M:=\{x\in V:d_x\in I_M\}.
\]
For each $x\in V$, define
\[
Y_M^{(x)}:=d_x\,\mathbf 1\{d_x\in I_M\},\quad  q_M:=\mathbb E[Y_M^{(x)}]
=
\sum_{k=\lceil 2M^{\alpha_1}\rceil}^{\lfloor 4M^{\alpha_1}\rfloor} ck^{1-\tau},  \quad S_M:=\sum_{x\in V} Y_M^{(x)}
=
\sum_{x\in V} d_x\,\mathbf 1\{d_x\in I_M\}.
\]
Then there exist constants $c_1,c_2>0$, depending only on $\tau$, such that for all sufficiently large~$M$,
\begin{equation}\label{eq:qM_bounds_fixed_new}
c_1 M^{-\alpha_1(\tau-2)}
\le
q_M
\le
c_2 M^{-\alpha_1(\tau-2)}.
\end{equation}
Since the random variables $(Y_M^{(x)})_{x\in V}$ are i.i.d.\ and uniformly bounded by $4M^{\alpha_1}$, for every fixed~$M$, by Hoeffding inequality,
\begin{equation}\label{eq:SM_concentration_fixed_new}
\mathbb P\Bigl(\bigl|S_M-q_M n\bigr|>\tfrac12 q_M n\Bigr)\le 2\exp(-\frac{q_M^2}{32M^{2\alpha_1}}n)\to 0
\qquad\text{as }n\to\infty.
\end{equation}
By Lemma~\ref{lemma_degree_event_D},
\begin{equation}\label{eq:degree_sum_event_fixed_new}
\mathbb P(\{\frac12 \mu n\le \sum_{x\in V} d_x\le \frac32 \mu n\})\to 1
\qquad\text{as }n\to\infty,
\end{equation}
By a union bound and the power-law tail estimate,
\begin{align}
\mathbb P\Bigl(\max_{x\in V} d_x>n^{1/(\tau-1)}\log n\Bigr)
&\le
n\,\mathbb P(d_1>n^{1/(\tau-1)}\log n)
\notag\\
&\le
C(\log n)^{-(\tau-1)}
\to 0
\qquad\text{as }n\to\infty.
\label{eq:maxdeg_fixed_new}
\end{align}
Let
\[
E_M:=
\Bigl\{
\tfrac12 q_M n\le S_M\le \tfrac32 q_M n
\Bigr\}
\cap \{\frac12 \mu n\le \sum_{x\in V} d_x\le \frac32 \mu n\}
\cap
\Bigl\{
\max_{x\in V} d_x\le n^{1/(\tau-1)}\log n
\Bigr\}.
\]
By \eqref{eq:SM_concentration_fixed_new}, \eqref{eq:degree_sum_event_fixed_new}, and \eqref{eq:maxdeg_fixed_new},
\begin{equation}\label{eq:EM_high_prob_fixed_new}
\mathbb P(E_M)\to 1
\qquad\text{as }n\to\infty
\end{equation}
for every fixed $M$. From now on, we work under the event $E_M$.

We begin by constructing the first layer. Expose the first $M$ pairings of the half-edges incident to $v$. Since $d_v\ge M$, this is always possible. For $1\le i\le M$, let $e_i$ denote the partner half-edge paired to $(v,i)$, and let $u_i$ be the vertex incident to $e_i$, that is, $e_i\in\Omega_{u_i}$. Define
\[
U_{i-1}:=\{u_1,\dots,u_{i-1}\},
\qquad
\bar{\mathcal F}_i:=\sigma(e_1,\dots,e_i).
\]
For $i=1,\dots,M$, define
\[
B_i:=\{u_i\in \mathcal K_M\setminus U_{i-1}\}.
\]
Thus $B_i$ is the event that, at step $i$, the newly exposed half-edge of $v$ is paired to a half-edge incident to a fresh vertex in $\mathcal K_M$.

On the event $E_M$, conditional on $\bar{\mathcal F}_{i-1}$, we have
\begin{equation}\label{eq:Bi_fixed_cond}
\mathbb P(B_i\mid \bar{\mathcal F}_{i-1})
=
\frac{\sum_{u\in \mathcal K_M\setminus U_{i-1}} d_u}{\sum_{u\in V} d_u-(2i-1)}.
\end{equation}
Since every vertex in $\mathcal K_M$ has degree at most $4M^{\alpha_1}$, on $E_M$,
\(
\sum_{u\in U_{i-1}\cap \mathcal K_M} d_u
\le
4(i-1)M^{\alpha_1}
\le
4M^{1+\alpha_1}.
\)
Therefore, using \eqref{eq:qM_bounds_fixed_new} and the definition of $E_M$, for all sufficiently large $n$,
\begin{align}
\mathbb P(B_i\mid \bar{\mathcal F}_{i-1},E_M)
\ge
\frac{\frac12 q_M n-4M^{1+\alpha_1}}{\frac32\mu n}
\ge
\frac{q_M}{4\mu}
\ge
c_3 M^{-\alpha_1(\tau-2)}
\label{eq:Bi_fixed_lower}
\end{align}
\begin{equation}\label{eq:Bi_fixed_upper}
\mathbb P(B_i\mid \bar{\mathcal F}_{i-1},E_M)
\le
c_4 M^{-\alpha_1(\tau-2)}
\end{equation}
for some constants $c_3,c_4>0$ depending only on $\tau$.
Now define
\[
p_i:=\mathbb P(B_i\mid \bar{\mathcal F}_{i-1},E_M),
\qquad
\xi_i:=\mathbf 1_{B_i}-p_i,
\]
and
\[
M_k:=\sum_{i=1}^k \xi_i,
\qquad
V_k:=\sum_{i=1}^k \mathbb E(\xi_i^2\mid \bar{\mathcal F}_{i-1},E_M),
\qquad
1\le k\le M.
\]
Then $(M_k)$ is a martingale with respect to $(\bar{\mathcal F}_k)$, and
\(
V_k\le \sum_{i=1}^k p_i
\le
c_4 M^{1-\alpha_1(\tau-2)}
\) a.s. Let
\(
\mu_M:=\sum_{i=1}^M p_i
\).
By \eqref{eq:Bi_fixed_lower},
\begin{equation}\label{eq:muM_fixed_lower}
\mu_M\ge c_3 M^{1-\alpha_1(\tau-2)}.
\end{equation}
We now apply Freedman's inequality to the martingale $(-M_k)$, see \cite[Theorem~1.6]{Freedman1975} or \cite[Theorem~1.1]{Tropp2011Freedman}. Take
\(
t:=\frac12\mu_M, \sigma^2:=c_4 M^{1-\alpha_1(\tau-2)}
\).
Then
\begin{align}
\mathbb P\Bigl(M_M\le -\tfrac12\mu_M\Bigm|E_M\Bigr)
=
\mathbb P\Bigl(-M_M\ge t\Bigm|E_M\Bigr)
\le
\exp\biggl(
-\frac{t^2/2}{\sigma^2+t/3}
\biggr)
\le
\frac12\exp(-M^{\alpha_2})
\label{eq:first_layer_success_fixed_new}
\end{align}
for all sufficiently large $M$, since
\(
1-\alpha_1(\tau-2)>\alpha_2
\).
On the complement of the event in \eqref{eq:first_layer_success_fixed_new},
\[
\sum_{i=1}^M \mathbf 1_{B_i}
=
\mu_M+M_M
\ge
\frac12\mu_M.
\]
Since
\(
1-\alpha_1(\tau-2)>\alpha_1
\),
it follows from \eqref{eq:muM_fixed_lower} that, for all sufficiently large $M$,
\(
\frac12\mu_M\ge M^{\alpha_1}
\).
Hence, conditional on $E_M$, with probability at least
\(
1-\frac12\exp(-M^{\alpha_2})
\),
the vertex $v$ has at least $\lceil M^{\alpha_1}\rceil$ distinct neighbors
\(
u_1,\dots,u_{\lceil M^{\alpha_1}\rceil}
\)
whose degrees all belong to $I_M$.

We now construct the second layer. Since each $u_i$ has degree at least $2M^{\alpha_1}$, after removing the unique half-edge already matched to $v$, there remain at least $M^{\alpha_1}$ unused half-edges incident to $u_i$. For each $i=1,\dots,\lceil M^{\alpha_1}\rceil$, expose the first $ \lceil M^{\alpha_1}\rceil$ such half-edges. The total number of second-layer pairings is at most
\(
\lceil M^{\alpha_1}\rceil^2.
\)
Let $\mathcal B$ denote the set of vertices that have already appeared during the second-layer construction. Then
\(
|\mathcal B|\le 1+\lceil M^{\alpha_1}\rceil+\lceil M^{\alpha_1}\rceil^2
\).
On the event $E_M$, every vertex in $\mathcal B$ has degree at most $n^{1/(\tau-1)}\log n$. Therefore the total half-edge mass attached to vertices in $\mathcal B$ is at most
\[
(1+\lceil M^{\alpha_1}\rceil+\lceil M^{\alpha_1}\rceil^2)\,n^{1/(\tau-1)}\log n.
\]
Since on $E_M$ the total half-edge mass is at least $\frac12\mu n$, it follows that, conditional on the previous exposure history and on $E_M$, the probability that the next second-layer pairing hits a previously used vertex is at most
\[
\frac{(1+\lceil M^{\alpha_1}\rceil+\lceil M^{\alpha_1}\rceil^2)\,n^{1/(\tau-1)}\log n}{\frac12\mu n}
=
\frac{2}{\mu}(1+\lceil M^{\alpha_1}\rceil+\lceil M^{\alpha_1}\rceil^2)\,n^{-\frac{\tau-2}{\tau-1}}\log n
=
o(1).
\]
Since the total number of second-layer pairings is at most $\lceil M^{\alpha_1}\rceil^2$, a union bound shows that
\begin{equation}\label{eq:collision_prob_fixed_new}
\mathbb P\bigl(\text{a previously used vertex is reused in the second layer}\mid E_M\bigr)
\to 0
\qquad\text{as }n\to\infty.
\end{equation}
Combining \eqref{eq:first_layer_success_fixed_new} and \eqref{eq:collision_prob_fixed_new}, for all sufficiently large $n$,
\[
\mathbb P\Bigl(
v \text{ is the center of a hierarchical-star with parameters }(M^{\alpha_1},M^{\alpha_1})
\Bigm| E_M
\Bigr)
\ge
1-\exp(-M^{\alpha_2}).
\]
Using \eqref{eq:EM_high_prob_fixed_new}, for all sufficiently large $n$,
\[
\mathbb P\Bigl(
v \text{ is the center of a hierarchical-star with parameters }(M^{\alpha_1},M^{\alpha_1})
\Bigr)
\ge
1-\exp(-M^{\alpha_2}).
\]
This completes the proof.
\end{proof}

Lemma~\ref{lemma_logMline} provides the graph-theoretic mechanism needed to climb from a fixed degree scale to the degree scale $n^c$. The underlying proof strategy is adapted from Lemmas~3.1 and~3.2 in Chatterjee and Durrett~\cite{ChatterjeeDurrett2009}, where the corresponding argument is carried out for the case $\tau>3$ and for the simple graph model $G(n,\tau)\mid \{G(n,\tau)\text{ is simple}\}$. Here we extend that argument to the full range $\tau>2$ and to the graph $G(n,\tau)$ itself. Although the setting is more general, the core idea of the proof is the same.

\begin{lemma}\label{lemma_logMline}
Let $G=G(n,\tau)$ be a power-law random graph on $n$ vertices with exponent $\tau>2$. Then there exists a constant
$
M^{(0)}\ge 2
$
such that the following holds. For any $M\ge M^{(0)}$, if $v_1$ is a vertex with degree at least $M$, and
\(
k:=\left\lfloor \frac{\log n}{\log M}\cdot \min \left\{\frac{1}{32\tau}, \frac{\tau-2}{32(\tau-1)} \right\} \right\rfloor
\), and if $E$ is the event that there exist vertices
$v_2,\dots,v_k$ (not necessarily distinct)
such that, for every $i=1,\dots,k-1$,
\[
d(v_i)\ge M^i,
\qquad
\mathrm{dist}(v_i,v_{i+1})\le 4\tau(i+1)\log M.
\]
Then, for all sufficiently large $n$,
\[
\mathbb P(E)\ge \frac13 .
\]
\end{lemma}

\begin{proof}
Fix $M\ge M^{(0)}$, where $M^{(0)}$ will be chosen sufficiently large at the end of the proof. For each $i\ge 1$, define
\[
I_i:=[M^{i+1},\,2M^{i+1}],
\qquad
\mathcal K_i:=\{x\in V:d_x\in I_i\},
\qquad
Y_i^{(x)}:=d_x\,\mathbf 1\{d_x\in I_i\},
\qquad
S_i:=\sum_{x\in V}Y_i^{(x)}.
\]
Let
\[
H_n:=
\Bigl\{\tfrac12\mu n\le \sum_{x\in V}d_x\le \tfrac32\mu n \Bigr\}
\cap
\Bigl\{\max_{x\in V} d_x\le n^{1/(\tau-1)}\log n\Bigr\}
\cap
\bigcap_{i=1}^k
\Bigl\{
\tfrac12\mathbb E[S_i]\le S_i\le \tfrac32\mathbb E[S_i]
\Bigr\}.
\]
By \eqref{eq:SM_concentration_fixed_new}, \eqref{eq:degree_sum_event_fixed_new}, and \eqref{eq:maxdeg_fixed_new} in Lemma~\ref{lemma_Mhierarchicalstar},
\begin{equation}\label{eq:Hn_whp_logMline}
\mathbb P(H_n)\to 1
\qquad\text{as }n\to\infty.
\end{equation}

We now work on the event $H_n$. We recursively construct vertices $v_2,\dots,v_k$.
Fix $i\in\{1,\dots,k-1\}$ and suppose that $v_i$ has already been constructed and satisfies
\(
d(v_i)\ge M^i
\).
Set
\(
L_i:=\lfloor M^{4\tau(i+1)}\rfloor.
\) Then
\begin{equation}\label{eq:Libound}
  L_i\le L_{k-1}\le M^{4\tau k}\le n^{\min\{\frac{1}{8}, \frac{1}{8}(1-\frac{1}{\tau-1}) \}}  
\end{equation}
Starting from $v_i$, we perform a breadth-first exploration as follows. We expose pairings of half-edges incident to the currently discovered vertices in breadth-first order, and we stop after discovering $L_i$ vertices. Let $A_i$ denote the event that during this exploration we discover a vertex $v_{i+1}$ of degree in $I_i=[M^{i+1},2M^{i+1}]$.

Let $y_1^{(i)},\dots,y_{L_i}^{(i)}$ denote the first $L_i$ explored vertices, listed in exploration order, starting from $v_i$. For each $\ell=1,\dots,L_i$, define
\(
B_\ell^{(i)}:=\{d(y_\ell^{(i)})\in I_i\}
\).
Then
\(
\cup_{\ell=1}^{L_i}B_\ell^{(i)}\subseteq A_i.
\)

We now estimate the probability of $A_i^{\mathrm c}$ on the event $H_n$. Conditional on the event
\(
H_n\cap \bigcap_{m=1}^{\ell-1}(B_m^{(i)})^{\mathrm c},
\)
the first $\ell-1$ explored vertices all have degrees outside the interval $I_i$. Hence the total degree mass already removed from $\mathcal K_i$ is at most $(\ell-1)\max_{x\in V}d_x\le L_in^{1/(\tau-1)}\log n $.
Since the next explored vertex is obtained by pairing a frontier half-edge to a uniformly chosen remaining half-edge outside the already exposed history, it follows that
\[
\mathbb P\Bigl(B_\ell^{(i)}\Bigm| H_n\cap \bigcap_{m=1}^{\ell-1}(B_m^{(i)})^{\mathrm c}\Bigr)
\ge
\frac{S_i-L_in^{1/(\tau-1)}\log n}{\sum_{x\in V}d_x}.
\]
By the definition of $H_n$,
\(
S_i\ge \frac12\mathbb E[S_i],\ \sum_{x\in V}d_x\le \frac32\mu n
\).
Also, there exists a constant $c_0>0$, depending only on $\tau$, such that for all sufficiently large $n$,
\(
\mathbb E[S_i]
=
n\sum_{m=\lceil M^{i+1}\rceil}^{\lfloor 2M^{i+1}\rfloor} cm^{1-\tau}
\ge
c_0 nM^{-(i+1)(\tau-2)}
\).
By \eqref{eq:Libound},
\(
L_in^{1/(\tau-1)}\log n=o(n)
\). Hence there exists a constant $c_1>0$, depending only on $\tau$, such that for all sufficiently large $n$,
\begin{equation}\label{eq:Bi_lower_logMline}
\mathbb P\Bigl(B_\ell^{(i)}\Bigm| H_n\cap \bigcap_{m=1}^{\ell-1}(B_m^{(i)})^{\mathrm c}\Bigr)
\ge
c_1 M^{-(i+1)(\tau-2)}
\qquad\text{for all }\ell\le L_i.
\end{equation}
Hence, on the event $H_n$,
\begin{align*}
\mathbb P(A_i^{\mathrm c}\mid H_n)
&\le
\mathbb P\Bigl(\bigcap_{\ell=1}^{L_i}(B_\ell^{(i)})^{\mathrm c}\Bigm|H_n\Bigr)
=
\prod_{\ell=1}^{L_i}
\mathbb P\Bigl((B_\ell^{(i)})^{\mathrm c}\Bigm|H_n\cap \bigcap_{m=1}^{\ell-1}(B_m^{(i)})^{\mathrm c}\Bigr)\\
&\le
\bigl(1-c_1 M^{-(i+1)(\tau-2)}\bigr)^{L_i}
\le
\exp\bigl(-c_1L_iM^{-(i+1)(\tau-2)}\bigr)\\
&=
\exp\bigl(-c_1M^{(i+1)(3\tau+2)}\bigr)
\le
\exp(-M^i)
\end{align*}
for all sufficiently large $M$.
Choose $M^{(0)}$ sufficiently large so that
\[
\sum_{i=1}^{\infty}\exp(-M^i)\le \frac16
\qquad\text{for all }M\ge M^{(0)}.
\]
Let
\(
A:=\bigcap_{i=1}^{k-1}A_i
\).
Then, for all sufficiently large $n$,
\begin{equation}\label{eq:Aprob_logMline}
\mathbb P(A\mid H_n)\ge \frac56.
\end{equation}
Next let $R$ denote the event that, during the entire collection of the $k-1$ breadth-first explorations above, no previously discovered vertex is ever hit again. On the event $A\cap H_n$, the total number of explored vertices is at most
\(
\sum_{i=1}^{k-1}L_i\le kM^{4\tau k}
\).
Therefore the total number of half-edges incident to explored vertices is at most
\(
kM^{4\tau k}n^{1/(\tau-1)}\log n\in o(n)
\).
At each exposure step, conditional on the past, the probability that the next paired half-edge leads back to a previously discovered vertex is at most
\(
\frac{kM^{4\tau k}n^{1/(\tau-1)}\log n}{\frac12\mu n}
\).
Since the total number of exposure steps is also at most $kM^{4\tau k}$, a union bound yields
\[
\mathbb P(R\mid A\cap H_n)
\ge
1-kM^{4\tau k}\frac{kM^{4\tau k}n^{1/(\tau-2)}\log n}{\frac12\mu n}.
\]
By \eqref{eq:Libound},
\begin{equation}\label{eq:Rprob_logMline}
\mathbb P(R\mid A\cap H_n)\to 1, \qquad\text{as }n\to\infty .
\end{equation}

On the event $A\cap R$, the $i$th breadth-first exploration reveals at most $L_i$ distinct vertices and never hits a previously discovered vertex. Since every vertex has degree at least $3$, the number of vertices within graph distance $m$ of $v_i$ is at least $2^m$ as long as no collision occurs. Hence every discovered vertex in the $i$th exploration lies within graph distance at most
\(
\log_2L_i\le 4\tau(i+1)\log M
\)
of $v_i$. In particular,
\[
\mathrm{dist}(v_i,v_{i+1})\le 4\tau(i+1)\log M
\qquad\text{for all }i=1,\dots,k-1.
\]
Since also $d(v_i)\ge M^i$ by construction, it follows that
\(
A\cap R\subseteq E
\).
Finally, combining \eqref{eq:Hn_whp_logMline}, \eqref{eq:Aprob_logMline}, and \eqref{eq:Rprob_logMline}, we obtain, for all sufficiently large $n$,
\[
\mathbb P(E)
\ge
\mathbb P(H_n)\,\mathbb P(A\mid H_n)\,\mathbb P(R\mid A\cap H_n)
\ge \frac13.
\]
This completes the proof.
\end{proof}

\section{Proof of the Main Theorems}\label{section_proofofmain}
In this section we prove the two main theorems of the paper. The argument combines the geometric results established in Section~\ref{section_hierarchicalstarinpowerlaw} with the dynamical analysis of the SIRS process on a single hierarchical-star developed in Proposition~\ref{hierarchicalstar_thm2} of Section~\ref{subsection_hierarchicalstar_1}.

Before turning to the proofs of the main theorems, we introduce two auxiliary lemmas. Lemma~\ref{lemma_diameter_log_n} shows that, with high probability, the diameter of the power-law random graph is bounded by $c\log n$ for some constant $c$. Lemma~\ref{lemma_infectionspreadthroughhierarchicalstars} shows that once at least one $(n^{\varepsilon},n^{\varepsilon})$ hierarchical-star is already in a HS-lit state, then with high probability all $(n^{\varepsilon},n^{\varepsilon})$ hierarchical-stars reach the same level of HS-lit state at some later time.

\begin{lemma}\label{lemma_diameter_log_n}
Let $G=G(n,\tau)$ be the power-law random graph with exponent $\tau>2$, and let
$
\mathrm{diam}(G):=\max_{u,v\in V} d_G(u,v)
$
denote the graph diameter, where $d_G(u,v)$ is the graph distance between $u$ and $v$. Then there exists constants $C>0$ such that
\[
\mathbb P\big(\mathrm{diam}(G)\le C\log n\big)\to 1\qquad\text{as}\quad n\to\infty.
\]
\end{lemma}

\begin{proof}
This is a direct consequence of the diameter bounds proved for the (power-law) configuration model in
\cite[Theorem~7.19]{vanDerHofstad2024rgcn2}.
\end{proof}

\begin{lemma}\label{lemma_infectionspreadthroughhierarchicalstars}
Let $G=(V,E)$ be a graph with
$
\mathrm{diam}(G)\le c\log n
$, $\delta\in (0,1), \varepsilon\ge \varepsilon_1>0$,
and suppose that $G$ contains $\lceil n^{1-\delta'}\rceil$ pairwise vertex-disjoint hierarchical-stars
\[
HS(v^{(\ell)}):=\mathrm{HStar}\bigl(v^{(\ell)},\{u_i^{(\ell)}\}_{i\le \lceil n^\varepsilon\rceil },\{w_{i,j}^{(\ell)}\}_{i\le \lceil n^\varepsilon\rceil ,\ j\le  \lceil n^{\varepsilon_1}\rceil}\bigr),
\qquad
\ell=1,\dots,\lceil n^{1-\delta'}\rceil.
\]
Fix any non-empty set
$
\mathcal I\subseteq \{1,\dots,\lceil n^{1-\delta'}\rceil \}
$.
Assume that at time $0$, for every $\ell\in\mathcal I$, the center $v^{(\ell)}$ is $(\frac14 n^{\varepsilon_3},n^{\varepsilon_2})$-HS-lit with respect to $HS(v^{(\ell)})$, where
\(
\varepsilon_2=\frac{1}{200}\min\{\lambda,\rho,1\}\min\{\varepsilon,\varepsilon_1\},
\varepsilon_3=\frac14\varepsilon_2,
\varepsilon_4=\tfrac{1}{16\cdot 200}\min\{\rho, \lambda, 1\}\cdot24\varepsilon_2.
\)
Then there exists an event $A$ such that $A$ implies that, for every
$
\ell=1,\dots,\lceil n^{1-\delta'}\rceil,
$
the center $v^{(\ell)}$ is $(\frac14 n^{\varepsilon_3},n^{\varepsilon_2})$-HS-lit with respect to $HS(v^{(\ell)})$ at time $\exp(n^{\varepsilon_3/2})$, and, for all sufficiently large $n$,
\[
\mathbb P(A)\ge
1-\exp(-n^{\varepsilon_4/20}).
\]
\end{lemma}

\begin{proof}
Fix an index $\ell_0\in\mathcal I$. For each $\ell\in\mathcal I$, Proposition~\ref{hierarchicalstar_thm2} yields an event $B(v^{(\ell)})\in\mathcal G(HS(v^{(\ell)}))$ such that, on $B(v^{(\ell)})$, the center $v^{(\ell)}$ remains $(\frac14 n^{\varepsilon_3},n^{\varepsilon_2})$-HS-lit with respect to $HS(v^{(\ell)})$ throughout the time interval $[\exp(n^{\varepsilon_4}),\exp(n^{\varepsilon_3/2})]$, and there exist random times $T_1,\dots,T_M$, with $M:=\lceil\exp(n^{\varepsilon_4/8})\rceil+1$, such that $T_{r+1}-T_r\ge n^{\varepsilon_2}$, $T_r\le \frac12\exp(n^{\varepsilon_3/2})$, and $v^{(\ell)}$ is infected at time $T_r$ for every $r=1,\dots,M-1$. Moreover,
\[
\mathbb P(B(v^{(\ell)}))\ge 1-\exp(-n^{\varepsilon_4/10}).
\]
By a union bound,
\begin{equation}\label{eq:initial_centers_survive_new}
\mathbb P\Big(\bigcap_{\ell\in\mathcal I}B(v^{(\ell)})\Big)\ge 1-\lceil n^{1-\delta'}\rceil\exp(-n^{\varepsilon_4/10}).
\end{equation}

Now fix $\ell\notin\mathcal I$. Since $\mathrm{diam}(G)\le c\log n$, choose a path
\(\gamma^{(\ell)}:=(z_0^{(\ell)},z_1^{(\ell)},\dots,z_{q_\ell}^{(\ell)})
\)
from $z_0^{(\ell)}=v^{(\ell_0)}$ to $z_{q_\ell}^{(\ell)}=v^{(\ell)}$ with $q_\ell\le c\log n$. On the event $B(v^{(\ell_0)})$, the vertex $v^{(\ell_0)}$ is infected at each time $T_r$, $r=1,\dots,M-1$. For each such $r$, define $E_r^{(\ell)}$ to be the event that the infection is transmitted successively along the path $\gamma^{(\ell)}$ during the time intervals $[T_r,T_r+q_\ell]$, so that $v^{(\ell)}$ is infected at time $T_r+q_\ell$. More explicitly, $E_r^{(\ell)}$ is the intersection of the edge-transmission events from Lemma~\ref{lemma_edge_transmission_one} along the oriented edges $z_{a-1}^{(\ell)}\to z_a^{(\ell)}$, $a=1,\dots,q_\ell$, shifted to the intervals $[T_r+a-1,T_r+a]$. Since $q_\ell\le c\log n$ and $T_{r+1}-T_r\ge n^{\varepsilon_2}$, these time intervals are pairwise disjoint for distinct $r$, for all sufficiently large $n$. Therefore, conditional on $B(v^{(\ell_0)})$ and on the times $T_1,\dots,T_M$, the events $E_r^{(\ell)}$ are independent. Moreover,
\[
\mathbb P(E_r^{(\ell)}\mid B(v^{(\ell_0)}),T_1,\dots,T_M)\ge c_{\mathrm{edge}}^{q_\ell}\ge n^{-c\log(1/c_{\mathrm{edge}})}.
\]

On $E_r^{(\ell)}$, the center $v^{(\ell)}$ is infected at time $T_r+q_\ell$. Applying the time-shifted version of Lemma~\ref{lemma_hierarchicalstar_infectedtoHSlit} from time $T_r+q_\ell$, and using the strong Markov property at that time, we obtain an event $F_r^{(\ell)}$ such that, on $E_r^{(\ell)}\cap F_r^{(\ell)}$, the center $v^{(\ell)}$ becomes $(\frac14 n^{\varepsilon_3},n^{\varepsilon_2})$-HS-lit with respect to $HS(v^{(\ell)})$ by time $T_r+q_\ell+1$, and
\[
\mathbb P(F_r^{(\ell)}\mid \mathcal F_{T_r+q_\ell})\ge 1-n^{-\varepsilon/4}\ge \frac12
\]
for all sufficiently large $n$. Hence, setting $G_r^{(\ell)}:=E_r^{(\ell)}\cap F_r^{(\ell)}$, we have
\[
\mathbb P(G_r^{(\ell)}\mid B(v^{(\ell_0)}),T_1,\dots,T_M)\ge \frac12 n^{-c\log(1/c_{\mathrm{edge}})}.
\]
The events $G_r^{(\ell)}$ are conditionally independent over $r$ after conditioning on $B(v^{(\ell_0)})$ and the times $T_1,\dots,T_M$, because they are supported on disjoint time intervals of the modified Harris construction. Therefore,
\begin{align}
\mathbb P\Big(\bigcap_{r=1}^{M-1}(G_r^{(\ell)})^c\ \Bigm|\ B(v^{(\ell_0)}),T_1,\dots,T_M\Big)
&\le \left(1-\frac12 n^{-c\log(1/c_{\mathrm{edge}})}\right)^{M-1}\notag\\
&\le \exp\left(-\frac14\exp(n^{\varepsilon_4/8})\,n^{-c\log(1/c_{\mathrm{edge}})}\right)
\label{eq:dormant_activation_failure_new}
\end{align}
for all sufficiently large $n$.

Let $S_\ell$ be the first time $v^{(\ell)}$ is $( n^{\varepsilon_3},n^{\varepsilon_1/24})$-HS-lit with respect to $HS(v^{(\ell)})$. If some $G_r^{(\ell)}$ occurs , then $S_\ell \le 
 T_{M-1}+c\log n+1\le \frac12\exp(n^{\varepsilon_3/2})+c\log n+1\le \frac23\exp(n^{\varepsilon_3/2})$
for all sufficiently large $n$. Starting from time $S_\ell$, apply the time-shifted version of Proposition~\ref{hierarchicalstar_thm2}. There exists an event $B(S_\ell, v^{(\ell)})$ such that $\mathbb{P}(B(S_\ell, v^{(\ell)})\mid \mathcal{F}_{S_\ell}, B(v^{(\ell_0)}), T_1,\dots, T_M) \ge 1-\exp(-n^{\varepsilon_4/10})$, and on event $B(S_\ell, v^{(\ell)})$ the center $v^{(\ell)}$ remains $(\frac14 n^{\varepsilon_3},n^{\varepsilon_2})$-HS-lit throughout time interval $[S_\ell+\exp(n^{\varepsilon_4}), S_\ell+\exp(n^{\varepsilon_3/2})]$. Since $S_\ell+\exp(n^{\varepsilon_4})\le \exp(n^{\varepsilon_3/2})$ for all sufficiently large $n$, this event implies that $v^{(\ell)}$ is $(\frac14 n^{\varepsilon_3},n^{\varepsilon_2})$-HS-lit with respect to $HS(v^{(\ell)})$ at time $\exp(n^{\varepsilon_3/2})$. Therefore, for every $\ell\notin\mathcal I$,
\begin{align}
\mathbb P\Big(\big(\bigcup_{r=1}^{M-1}G_r^{(\ell)}\big)\cap B(S_\ell, v^{(\ell)}) \Bigm|\ B(v^{(\ell_0)}),T_1,\dots,T_M\Big)  \notag\\
\ge1- \exp(-n^{\varepsilon_4/10})
-\exp\left(-\frac14\exp(n^{\varepsilon_4/8})\,n^{-c\log(1/c_{\mathrm{edge}})}\right).
\label{eq:dormant_final_failure_new}
\end{align}

Define
\[
A:=\bigcap_{\ell\in \mathcal I}B(v^{(\ell)})\cap \bigcap_{\ell\notin \mathcal I}\left(\left(\bigcup_{r=1}^{M-1}G_r^{(\ell)}\right)\cap B(S_\ell, v^{(\ell)})\right).
\]
Then on the event $A$, every $v^{(\ell)}$ is $(\frac14 n^{\varepsilon_3},n^{\varepsilon_2})$-HS-lit with respect to $HS(v^{(\ell)})$ at time $\exp(n^{\varepsilon_3/2})$, for $\ell=1,\dots,\lceil n^{1-\delta'}\rceil$. Combining \eqref{eq:initial_centers_survive_new} and \eqref{eq:dormant_final_failure_new}, and taking a union bound over $\ell\notin\mathcal I$, we get
\[
\mathbb P(A)\ge 1-\lceil n^{1-\delta'}\rceil\exp(-n^{\varepsilon_4/10})-\lceil n^{1-\delta'}\rceil\left(\exp(-\frac14\exp(n^{\varepsilon_4/8})\,n^{-c\log(1/c_{\mathrm{edge}})})+ \exp(-n^{\varepsilon_4/10}) \right).
\]
Consequently, for all sufficiently large $n$,
\[
\mathbb P(A)\ge 1-\exp(-n^{\varepsilon_4/20}).
\]
\end{proof}

We now turn to the proof of Theorem~\ref{thm_main1}. 

\begin{proof}[Proof of Theorem \ref{thm_main1}]
    Let 
    $$\delta':=\delta/2, \quad \varepsilon:=\frac{\delta'}{400\tau^3}\, \quad\varepsilon_1=\varepsilon, \quad
\varepsilon_2:=\frac{1}{200}\varepsilon_1\min\{\lambda,\rho,1\},$$
$$
\varepsilon_3:=\frac{1}{4}\varepsilon_2, \quad
\varepsilon_4=\tfrac{1}{16\cdot 200}\min\{\rho, \lambda, 1\}\cdot24\varepsilon_2. $$
By Proposition~\ref{prop:existsmanyhierarchicalstar}, there exists a constant $c_1>0$ such that, for all sufficiently large $n$, with probability at least $1-c_1n^{-\min\{\tau-2, 1\}}(\log n)^{\mathbf{1}\{\tau=3\}}$, the graph $G$ contains at least $\lceil n^{1-\delta'}\rceil$ vertex-disjoint $(n^\varepsilon,n^\varepsilon)$ hierarchical-stars. By Lemma~\ref{lemma_diameter_log_n}, there exists a constant $c_2>0$ such that $\mathbb{P}(\mathrm{diam}(G)\le c_2\log n)\to 1$. We condition throughout on the intersection of these two events, denoted by $F_n$.

On this event $F_n$, fix a family of $\lceil n^{1-\delta'}\rceil$ vertex-disjoint $(n^{\varepsilon}, n^{\varepsilon})$ hierarchical-stars, and denote their centers by
$
v^{(1)},\dots,v^{(\lceil n^{1-\delta'}\rceil)}
$. For each $\ell=1,\dots,\lceil n^{1-\delta'}\rceil$, write
$$
HS(v^{(\ell)}):=\mathrm{HStar}\bigl(v^{(\ell)},\{u_i^{(\ell)}\}_{i\le \lceil n^\varepsilon\rceil},\{w_{i,j}^{(\ell)}\}_{i\le  \lceil n^\varepsilon\rceil,\ j\le \lceil n^{\varepsilon_1}\rceil}\bigr).
$$
Let $(\mathcal{F}_t)_{t\ge 0}$ be the natural filtration of the modified Harris construction, as defined in Subsection~\ref{ssec:harris}. For each $r\in \mathbb{N}$, set $t_r:=r\exp(n^{\varepsilon_3/2})$, and let $\bar{\mathcal{F}}_r:=\sigma(\mathcal{H}|_{[0,t_r]})$ denote the $\sigma$-field generated by the restriction of the modified Harris construction to the time interval $[0,t_r]$. Define
$$
L_r:=\Bigl\{v^{(\ell)}\in \{v^{(1)},\dots, v^{(\lceil n^{1-\delta'}\rceil)} \}:\text{$v^{(\ell)}$ is $(\frac14 n^{\varepsilon_3},n^{\varepsilon_2})$-HS-lit with respect to $HS(v^{(\ell)})$ at time $t_r$}\Bigr\},
$$
$$
W_r:=|L_r|.
$$
Then $W_0=\lceil n^{1-\delta'}\rceil$. To prove that $T>\exp(n^{1-\delta})$, it suffices to show that $W_r>0$ for some $r$ with $t_r>\exp(n^{1-\delta})$, since $W_r>0$ means that at least one hierarchical-star center is still HS-lit at time $t_r$, and in particular the infection has not gone extinct by that time.

Choose $\tilde{L}^{(1)}_0$ to be a subset of $L_0$ such that $|\tilde{L}^{(1)}_0|=\lceil\frac{1}{2} n^{1-\delta'}\rceil+1$, according to a fixed deterministic rule, and let $\tilde{W}^{(1)}_0=|\tilde{L}^{(1)}_0|=\lceil\frac{1}{2}n^{1-\delta'}\rceil+1$. We then define $\tilde{L}^{(1)}_{r+1}$ inductively, given $\tilde{L}^{(1)}_r$, and set $\tilde{W}^{(1)}_{r+1}=|\tilde{L}^{(1)}_{r+1}|$.

Fix an epoch index $r\in\mathbb{N}$, and assume that $\tilde{L}^{(1)}_r$ has already been defined. For each $v^{(\ell)}\in \tilde{L}^{(1)}_r$, let $N_r^{(\ell)}$ be the event that $v^{(\ell)}$ is not $(\frac14 n^{\varepsilon_3},n^{\varepsilon_2})$-HS-lit with respect to $HS(v^{(\ell)})$ at time $t_{r+1}$.

By Proposition~\ref{hierarchicalstar_thm2} applied from time $t_r$, for each $v^{(\ell)}\in \tilde{L}^{(1)}_r$ there exists an event $\tilde{B}_r(v^{(\ell)})\in \mathcal{G}(HS(v^{(\ell)}))$ such that for all sufficiently large $n$,
$$
\mathbb{P}(\tilde{B}_r(v^{(\ell)})\mid \bar{\mathcal F}_r)\ge 1-\exp(-n^{\varepsilon_4/10}),
$$
and $\tilde{B}_r(v^{(\ell)})$ implies that $v^{(\ell)}$ is $(\frac14 n^{\varepsilon_3},n^{\varepsilon_2})$-HS-lit with respect to $HS(v^{(\ell)})$ at time $t_{r+1}$. Therefore,
$$
N_r^{(\ell)}\subseteq (\tilde{B}_r(v^{(\ell)}))^{\mathrm{c}}.
$$
Let
\[
N_r=\sum_{v^{(\ell)}\in \tilde{L}^{(1)}_r}\mathbf{1}\{N_r^{(\ell)}\},
\qquad
\tilde{N}_r=\sum_{v^{(\ell)}\in \tilde{L}^{(1)}_r}\mathbf{1}\{(\tilde{B}_r(v^{(\ell)}))^{\mathrm{c}}\}.
\]
For $\ell_1\ne \ell_2$, the $\sigma$-fields $\mathcal{G}(HS(v^{(\ell_1)}))$ and $\mathcal{G}(HS(v^{(\ell_2)}))$ are independent. Hence the events $(\tilde{B}_r(v^{(\ell_1)}))^{\mathrm{c}}$ and $(\tilde{B}_r(v^{(\ell_2)}))^{\mathrm{c}}$ are conditionally independent given $\bar{\mathcal F}_r$. Therefore,
\[
N_r\le \tilde N_r \le_{\mathrm{st}} \mathrm{Binomial}(\tilde W^{(1)}_r,\,\exp(-n^{\varepsilon_4/10})).
\]
Since $\tilde{W}^{(1)}_r\le \lceil n^{1-\delta'}\rceil$ for all $r$,
\begin{equation}\label{eq:thm1_supermartingaleNr}
  \mathbb{E}\!\left[e^{\tilde N_r}\mid \bar{\mathcal{F}}_r\right]
\le \big(1-\exp(-n^{\varepsilon_4/10})+e\exp(-n^{\varepsilon_4/10})\big)^{\lceil n^{1-\delta'}\rceil}
\le 1+o(1),
\end{equation}
and
\begin{equation}\label{eq:thm1_supermartingaleNr2}
    \mathbb{P}(\tilde{N}_r=0\mid \bar{\mathcal{F}}_r) \ge (1-\exp(-n^{\varepsilon_4/10}))^{\lceil n^{1-\delta'}\rceil}\ge 1-\lceil n^{1-\delta'}\rceil\exp(-n^{\varepsilon_4/10})\ge 1-o(1).
\end{equation}

Choose any vertex \(v^{(x)}\notin  \tilde{L}^{(1)}_r\) according to a fixed deterministic rule, whenever $\tilde{W}_r<\lceil n^{1-\delta'}\rceil$. By Lemma~\ref{lemma_infectionspreadthroughhierarchicalstars}, there exists an event $\tilde A_r$ such that, for all sufficiently large $n$, the following hold:
\begin{enumerate}
    \item[\normalfont (i)] On $\tilde A_r \cap \{\tilde W^{(1)}_r\ge 1\}$, the vertex $v^{(x)}$ is $(\frac14 n^{\varepsilon_3},n^{\varepsilon_2})$-HS-lit with respect to $HS(v^{(x)})$ at time~$t_{r+1}$;
    \item[\normalfont (ii)]
    We have
    \begin{equation}\label{eq:thm1_supermartingaleAr1}
    \mathbb{P}(\tilde A_r\mid\bar{\mathcal{F}}_r)\ \ge\ 1-\exp(-n^{\varepsilon_4/10})\ge 1-o(1).
\end{equation}
\end{enumerate}
Define
\[
\tilde L^{(1)}_{r+1}:=
\begin{cases}
 \tilde{L}^{(1)}_r \backslash \bigcup_{v^{(\ell)}\in \tilde{L}^{(1)}_r: (\tilde{B}_r(v^{(\ell)}))^{\mathrm{c}}\text{ occurs}}\{v^{(\ell)}\} , & \text{if } \{\tilde N_r>0\} \text{ occurs},\\
 \tilde{L}^{(1)}_{r}\cup \{v^{(x)}\} , & \text{if } \tilde A_r\cap \{\tilde{N}_r=0 \} \text{ occurs},\\
\tilde L^{(1)}_r, & \text{otherwise}.
\end{cases}\ .
\]
We claim that if $\tilde L^{(1)}_r\subseteq L_r$ and $\tilde W^{(1)}_r\ge 1$, then $\tilde L^{(1)}_{r+1}\subseteq L_{r+1}$.
\begin{itemize}
    \item If $\tilde N_r>0$ occurs, then for any $v^{(\ell)}\in \tilde{L}^{(1)}_r$ such that $(\tilde{B}_r(v^{(\ell)}))^{\mathrm{c}}$ does not occur, we have that $v^{(\ell)}$ is $(\frac14 n^{\varepsilon_3},n^{\varepsilon_2})$-HS-lit with respect to $HS(v^{(\ell)})$ at time $t_{r+1}$, hence $v^{(\ell)}\in L_{r+1}$. Therefore $\tilde{L}^{(1)}_{r+1}\subseteq L_{r+1}$.
    \item If $\tilde A_r \cap \{\tilde{N}_r=0 \}$ occurs, then $v^{(x)}$ is $(\frac14 n^{\varepsilon_3},n^{\varepsilon_2})$-HS-lit with respect to $HS(v^{(x)})$ at time $t_{r+1}$, hence $v^{(x)}\in L_{r+1}$. Since also $\tilde N_r=0$, every vertex of $\tilde L^{(1)}_r$ remains in $L_{r+1}$. Therefore $\tilde{L}^{(1)}_{r+1}\subseteq L_{r+1}$.
    \item Otherwise, we have $\tilde N_r=0$ and $\tilde A_r$ does not occur. In this case, every vertex $v^{(\ell)}\in \tilde L^{(1)}_r$ remains $(\frac14 n^{\varepsilon_3},n^{\varepsilon_2})$-HS-lit with respect to $HS(v^{(\ell)})$ at time $t_{r+1}$, hence belongs to $L_{r+1}$. Therefore $\tilde L^{(1)}_{r+1}=\tilde L^{(1)}_r\subseteq L_{r+1}$.
\end{itemize}
Since $\tilde L^{(1)}_0\subseteq L_0$, it follows by induction that $\tilde{L}^{(1)}_r\subseteq L_r$ for all $r$ before the first time at which $\tilde W^{(1)}_r=0$. Now let
\[
\tilde W^{(1)}_{r+1}:=|\tilde{L}^{(1)}_{r+1}|=
\begin{cases}
 \tilde{W}^{(1)}_r -\tilde{N}_r , & \text{if } \{\tilde N_r>0\} \text{ occurs},\\
 \tilde{W}^{(1)}_{r}+1 , & \text{if } \tilde A_r\cap \{\tilde{N}_r=0 \} \text{ occurs},\\
\tilde W^{(1)}_r, & \text{otherwise}.
\end{cases}\ .
\]
Then $\tilde{W}^{(1)}_r\le W_r$ for all $r$ before the first time at which $\tilde W^{(1)}_r=0$.
Let $X_r:=\tilde{W}^{(1)}_{r+1}-\tilde{W}^{(1)}_r$. Conditioning on $\bar{\mathcal{F}}_r$, on the event $\tilde W^{(1)}_r\in [1, \lceil n^{1-\delta}\rceil)$, we obtain
\begin{align*}
\mathbb{E}\!\left[e^{-X_r}\mid \bar{\mathcal{F}}_r\right]
&\le \sum_{m=1}^{\tilde W^{(1)}_r} e^{m}\,\mathbb{P}(\tilde{N}_r=m\mid \bar{\mathcal{F}}_r)
  +1\cdot \mathbb{P}((\tilde A_r \cap \{\tilde{N}_r=0 \})^{\mathsf c}\mid \bar{\mathcal{F}}_r) \\
&\qquad
  +e^{-1}\cdot \mathbb{P}(\tilde A_r \cap \{\tilde{N}_r=0 \}\mid \bar{\mathcal{F}}_r)\\
&=\mathbb{E}\!\left[e^{\tilde{N}_r}\mid \bar{\mathcal{F}}_r\right]
  -(1-e^{-1})\,\mathbb{P}(\tilde A_r \cap \{\tilde{N}_r=0 \}\mid \bar{\mathcal{F}}_r).
\end{align*}
Together with \eqref{eq:thm1_supermartingaleNr}, \eqref{eq:thm1_supermartingaleNr2}, and \eqref{eq:thm1_supermartingaleAr1}, this shows that for all sufficiently large $n$, on the event $\tilde W^{(1)}_r\ge 1$,
\[
\mathbb{E}\!\left[e^{-X_r}\mid \bar{\mathcal{F}}_r\right] < 1.
\]
Since $\tilde W^{(1)}_r$ is $\bar{\mathcal F}_r$-measurable and $\exp(-\tilde W^{(1)}_{r+1})=\exp(-\tilde W^{(1)}_r)\exp(-X_r)$, we have
\[
\mathbb E\!\left[\exp(-\tilde W^{(1)}_{r+1})\mid \bar{\mathcal F}_r\right]\le \exp(-\tilde W^{(1)}_r)
\]
on the event $\{\tilde W^{(1)}_r\in[1,\tfrac34 n^{1-\delta'}]\}$.
Define
\[
\tau_*(1):=\inf\{r\ge 0:\ \tilde W^{(1)}_r\notin [1,\,\tfrac{3}{4} n^{1-\delta'}]\},\qquad D_-^{(1)}:=\{\tau_*(1)<\infty,\ \tilde W^{(1)}_{\tau_*(1)}=0\}.
\]
Set $Y^{(1)}_r:=\exp(-\tilde W^{(1)}_{r\wedge\tau_*(1)})$. The preceding local estimate implies $\mathbb E[Y^{(1)}_{r+1}\mid \bar{\mathcal F}_r]\le Y^{(1)}_r$. Hence, for every $R\ge1$, $\mathbb E[Y^{(1)}_R]\le Y^{(1)}_0=\exp(-\tilde W^{(1)}_0)$. On $D_-^{(1)}\cap\{\tau_*(1)\le R\}$, we have $Y^{(1)}_R=1$. Therefore,
\[
\mathbb P(D_-^{(1)}\cap\{\tau_*(1)\le R\})\le \exp(-\tilde W^{(1)}_0)\le \exp(-\lceil\tfrac12 n^{1-\delta'}\rceil).
\]
Letting $R\to\infty$ gives
\[
\mathbb P(D_-^{(1)})\le \exp(-\lceil\tfrac12 n^{1-\delta'}\rceil).
\]
Thus, with probability at least $1-\exp(-\lceil\tfrac12 n^{1-\delta'}\rceil)$, either $\tau_*(1)=\infty$ or $\tau_*(1)<\infty$ and $\tilde W^{(1)}_{\tau_*(1)}\ge \frac{3}{4}n^{1-\delta'}$.

We now iterate the same construction. If $\tau_*(1)=\infty$, then $\tilde W^{(1)}_r\ge1$ for all $r\ge0$, hence $W_r>0$ for all $r\ge0$. In this case, the infection persists for arbitrarily large times, and in particular $T>\exp(n^{1-\delta})$. 

Otherwise, on the event $(D_-^{(1)})^{\mathrm c}$, let $\tilde{L}^{(2)}_0$ be a subset of $\tilde L^{(1)}_{\tau_*(1)}$ such that $|\tilde{L}^{(2)}_0|=\lceil \frac{1}2{}n^{1-\delta'}\rceil+1$, according to a fixed deterministic rule, and let $\tilde{W}^{(2)}_0:=|\tilde{L}^{(2)}_0|$. Then define $\tilde{L}^{(2)}_r,\tilde{W}^{(2)}_r$ by the same inductive construction above. Since $\tilde{L}^{(1)}_{\tau_*(1)}\subseteq L_{\tau_*(1)}$, we have $\tilde{L}^{(2)}_r\subseteq L_{\tau_*(1)+r}$ for all $r$ before the first time at which $\tilde W^{(2)}_r=0$.
Set
\[
\tau_*(2):=\inf\{ r\ge 0: \tilde{W}^{(2)}_r\notin [1,\tfrac{3}{4}n^{1-\delta'}]\},
\qquad
D_-^{(2)}:=\{\tau_*(2)<\infty,\ \tilde{W}^{(2)}_{\tau_*(2)}=0\}.
\]
The same local estimate and stopped-process argument give
\[
\mathbb P(D_-^{(2)}\mid \mathcal F_{t_{\tau_*(1)}})\le \exp(-\lceil\tfrac12 n^{1-\delta'}\rceil)
\]
on $(D_-^{(1)})^{\mathrm c}\cap\{\tau_*(1)<\infty\}$.

Proceeding inductively, for each $i\ge 1$, if none of $D_-^{(1)},\dots,D_-^{(i)}$ occurs and $\tau_*(1),\dots,\tau_*(i)<\infty$, define $\tilde{L}^{(i+1)}_0$ to be any subset of $\tilde L^{(i)}_{\tau_*(i)}$ such that $|\tilde{L}^{(i+1)}_0|=\lceil\frac{1}{2}n^{1-\delta'}\rceil+1$, and let $\tilde{W}^{(i+1)}_0:=|\tilde{L}^{(i+1)}_0|$. Then define $\tilde{L}^{(i+1)}_r,\tilde{W}^{(i+1)}_r$ by the same inductive construction above. Since $\tilde{L}^{(i)}_{\tau_*(i)}\subseteq L_{\sum_{j=1}^i\tau_*(j)}$, we have $\tilde{L}^{(i+1)}_r\subseteq L_{\sum_{j=1}^i\tau_*(j)+r}$ for all $r$ before the first time at which $\tilde W^{(i+1)}_r=0$.
Set
\[
\tau_*(i+1):=\inf\{ r\ge 0: \tilde{W}^{(i+1)}_r\notin [1,\tfrac{3}{4}n^{1-\delta'}]\},
\qquad
D_-^{(i+1)}:=\{\tau_*(i+1)<\infty,\ \tilde{W}^{(i+1)}_{\tau_*(i+1)}=0\}.
\]
Again, on the event that none of $D_-^{(1)},\dots,D_-^{(i)}$ occurs and $\tau_*(1),\dots,\tau_*(i)<\infty$, the same argument gives
\[
\mathbb P(D_-^{(i+1)}\mid \mathcal F_{t_{\sum_{j=1}^i\tau_*(j)}})\le \exp(-\lceil\tfrac12 n^{1-\delta'}\rceil).
\]
Let $M:=\lceil\exp(n^{1-\delta})\rceil$. By a union bound over the possible lower-exit events,
$$
\mathbb P\Big(\bigcup_{i=1}^{M}D_-^{(i)}\Big)\le M\exp\!\Big(-\tfrac{1}{2}n^{1-\delta'}\Big)\le \exp(-n^{1-\delta})
$$
for all sufficiently large $n$, since $\delta=2\delta'$. On the complementary event, either some $\tau_*(i)=\infty$ for $i\le M$, in which case $W_r>0$ for all later epochs and hence $T\ge t_M$, or else $\tau_*(i)<\infty$ for every $i=1,\dots,M$, in which case $\tilde W^{(i)}_{\tau_*(i)}\ge \frac{3}{4}n^{1-\delta'}$ for every $i=1,\dots,M$ and in particular
$$
W_{\sum_{j=1}^{M}\tau_*(j)}=\left|L_{\sum_{j=1}^{M}\tau_*(j)}\right|\ge \frac{3}{4}n^{1-\delta'}>0.
$$
Since $\tilde W_{0}^{(i)}=\lceil\frac{1}{2}n^{1-\delta'}\rceil+1\in [1, \frac{3}{4}n^{1-\delta'})$, we have $\tau_*(i)\ge 1$ whenever $\tau_*(i)<\infty$. Hence, in the second case,
$$
\sum_{j=1}^{M}\tau_*(j)\ge M,
$$
and therefore
$$
t_{\sum_{j=1}^{M}\tau_*(j)}\ge t_M=M\exp(n^{\varepsilon_3/2})> M=\exp(n^{1-\delta}).
$$
In either case, $T>\exp(n^{1-\delta})$ on the complement of $\bigcup_{i=1}^{M}D_-^{(i)}$ and on the event $F_n$. Therefore,
$$
\mathbb{P}(T>\exp(n^{1-\delta}))\ge 1-\mathbb{P}(F_n^\mathrm{c})-\exp(-n^{1-\delta}).
$$
Since $\mathbb{P}(F_n^\mathrm{c})\to0$, this completes the proof.
\end{proof}

Since the proof of Theorem~\ref{thm_main1} uses the randomness of the power-law graph only through the existence of many disjoint hierarchical-stars and the logarithmic diameter bound, the same argument gives the following deterministic-graph version. The proof is identical to the proof of Theorem~\ref{thm_main1}, with the event $F_n$ replaced by the two deterministic assumptions in the statement.

\begin{proposition}\label{prop:deterministic_many_hierarchicalstars}
Fix $\lambda,\rho, c, \varepsilon>0$ and $\delta\in(0,1)$. For all sufficiently large $n$, let $G=(V,E)$ be any deterministic graph on $n$ vertices satisfying the following two conditions:
\begin{enumerate}[\normalfont (i)]
    \item $G$ contains at least $\lceil n^{1-\delta}\rceil$ vertex-disjoint $(n^{\varepsilon},n^{\varepsilon})$ hierarchical-stars;
    \item $\operatorname{diam}(G)\le c\log n$.
\end{enumerate}
Consider the process $\mathrm{SIRS}(G,\lambda,\rho)$ initialized with all vertices infected. Then
\[
\mathbb P\big(T\ge \exp(n^{1-\delta})\big)\ge 1-\exp(-n^{1-\delta}).
\]
Moreover, if the process is initialized from a configuration in which the center of at least one of these $(n^{\varepsilon},n^{\varepsilon})$ hierarchical-stars is infected, then there exists a constant $\varepsilon'>0$ such that
\[
\mathbb P\big(T\ge \exp(n^{1-\delta})\big)\ge 1-n^{-\varepsilon'}.
\]
\end{proposition}

We next prove Theorem~\ref{thm_main2}. In contrast with Theorem~\ref{thm_main1}, the process now starts from a much weaker initial condition, so the first step is to show that with a probability bounded below by a positive constant, the infection reaches a sufficiently large-degree vertex. Lemmas~\ref{lemma_Mhierarchicalstar} and~\ref{lemma_logMline} provide the graph-theoretic mechanism for this bootstrap. Once such a vertex is reached, the local hierarchical-star dynamics imply that infection survives around it for a stretched-exponential amount of time, which is enough to transfer infection to one of the large hierarchical-stars furnished by Proposition~\ref{prop:existsmanyhierarchicalstar}. From that point onward, Lemma~\ref{lemma_infectionspreadthroughhierarchicalstars} and the proof of Theorem~\ref{thm_main1} apply.

\begin{proof}[Proof of Theorem~\ref{thm_main2}]
  Let 
   $\delta':=\delta/2, \varepsilon:=\frac{\delta'}{400\tau^3}, \varepsilon_1=\varepsilon, \varepsilon_3=\frac{1}{800}\min\{\rho,\lambda,1\}\varepsilon$. By Proposition~\ref{prop:existsmanyhierarchicalstar}, there exists a constant $c_1>0$ such that, for all sufficiently large $n$, with probability at least $1-c_1n^{-\min\{\tau-2, 1\}}(\log n)^{\mathbf{1}\{\tau=3\}}$, the graph $G$ contains at least $\lceil n^{1-\delta'}\rceil$ vertex-disjoint $(n^\varepsilon,n^\varepsilon)$ hierarchical-stars, and denote their centers by
$
v^{(1)},\dots,v^{(\lceil n^{1-\delta'}\rceil)}
$. By Lemma~\ref{lemma_diameter_log_n}, there exists a constant $c_2>0$ such that $\mathbb{P}(\mathrm{diam}(G)\le c_2\log n)\to 1$. We condition throughout on the intersection of these two events, denoted by $F_n$. Then
\begin{equation}\label{eq:main2_Fn}
\mathbb{P}(F_n)\to 1.
\end{equation}
By Lemma~\ref{lemma_Mhierarchicalstar}, Lemma~\ref{lemma_logMline}, Lemma~\ref{lemma_hierarchicalstar_infectedtoHSlit}, and Proposition~\ref{hierarchicalstar_thm2}, there exist constants
$$
M^{(1)}, M^{(2)}, M^{(3)}\ge 1, \qquad \alpha_1,\alpha_2,\alpha_3,\alpha_4, \alpha_5>0
$$
depending only on $\lambda,\rho,\tau$ such that the following hold for all sufficiently large $n$.
\begin{enumerate}
    \item[\normalfont (i)] If $v$ is a vertex with degree at least $M\ge M^{(1)}$, then with probability at least
$
1-\exp(-M^{\alpha_2})
$,
the vertex $v$ is the center of a hierarchical-star $HS(v)$ with parameters $(M^{\alpha_1},M^{\alpha_1})$; denote this event by $A_v$.
    \item[\normalfont (ii)] If $v$ is infected and is the center of such a hierarchical-star, then with probability at least
$
1-M^{-\alpha_3}
$,
the vertex $v$ becomes $(M^{\alpha_3},M^{\alpha_3})$-HS-lit with respect to that hierarchical-star before time $1$; denote this event by $B_v$.
    \item[\normalfont (iii)] If $v$ is $(M^{\alpha_3},M^{\alpha_3})$-HS-lit with respect to that hierarchical-star, then with probability at least
$
1-\exp(-M^{\alpha_4})
$,
there exist random times $(T_j^{(v)})_{j=1}^{\lceil M^{\alpha_4}\rceil+1}$ such that $v$ is infected at time $T_j^{(v)}$, $T_{j+1}^{(v)}-T_j^{(v)}\ge M^{\alpha_5}$ for every $j=1,\dots, \lceil M^{\alpha_4}\rceil$; denote this event by $C_v$.
\end{enumerate}
Let
$
M>\max\{M^{(1)},M^{(2)},M^{(3)}\}
$ be a constant, chosen sufficiently large later.
Since the degree distribution is power-law and $M$ is fixed, there exists a constant $p_0=p_0(M,\tau)>0$ such that for all sufficiently large $n$,
\begin{equation}\label{eq:main2_initial_degree}
\inf_{\mu_n}\mathbb{P}^*_{\mu_n}\big(\deg(v^*)\ge M\big)\ge p_0.
\end{equation}
Indeed, the choice of $v^*$ is uniform and independent of the base configuration sampled from $\mu_n$, and overwriting the state of $v^*$ by $I$ does not affect its degree.

Therefore, it suffices to prove that, starting from an infected vertex $v_1$ with degree at least $M$, the event $T>\exp(n^{1-\delta})$ happens with probability bounded below by a positive constant, uniformly in $n$.

\noindent\textbf{Step 1: bootstrap along a degree ladder.}
Fix an infected vertex $v_1$ with $\deg(v_1)\ge M$. By Lemma~\ref{lemma_logMline}, for all sufficiently large $n$, with probability at least $1/3$, there exist vertices
$
v_2,\dots,v_k$,  $k:=\left\lfloor \frac{\log n}{\log M}\cdot \min \left\{\frac{1}{32\tau}, \frac{\tau-2}{32(\tau-1)} \right\} \right\rfloor
$, 
such that for every $i=1,\dots,k-1$,
$
d(v_i)\ge M^i,  \mathrm{dist}(v_i,v_{i+1})\le 4\tau(i+1)\log M.
$
Let $E$ denote this event.
For each $i=1,\dots,k-1$, let $A_{v_i}$ be the event that $v_i$ is the center of a hierarchical-star with parameters $((M^i)^{\alpha_1},(M^i)^{\alpha_1})$. By enlarging $M$ if necessary, we have
\(
\sum_{i=1}^\infty \exp\big(-(M^i)^{\alpha_2}\big)\le \frac{1}{20}
\).
It then follows that
\begin{equation}\label{eq:main2_EA}
\mathbb{P}\Big(E\cap \bigcap_{i=1}^{k-1}A_{v_i}\Big)\ge \frac14
\end{equation}
for all sufficiently large $n$.
Now condition on the event $E\cap \bigcap_{i=1}^{k-1}A_{v_i}$. We claim that there exists a constant $q_1>0$, depending only on $\lambda,\rho,\tau,\delta$, such that with conditional probability at least $q_1$, the infection reaches the vertex $v_{k-1}$.

We first handle the initial step separately. Since $v_1$ is infected at time $0$, by the definition of the event $B_{v_1}$, with probability at least
$
1-M^{-\alpha_3}
$,
the vertex $v_1$ becomes $(M^{\alpha_3},M^{\alpha_3})$-HS-lit with respect to $HS(v_1)$ before time $1$. Therefore, in order to propagate the infection through the degree ladder, it suffices to estimate, for each $i=1,\dots,k-2$, the probability that the next vertex $v_{i+1}$ eventually becomes $((M^{i+1})^{\alpha_3},(M^{i+1})^{\alpha_3})$-HS-lit with respect to $HS(v_{i+1})$ for some time.

Fix $i\in\{1,\dots,k-2\}$ and suppose that $v_i$ is $((M^i)^{\alpha_3},(M^i)^{\alpha_3})$-HS-lit with respect to $HS(v_i)$ at some time.  Let $L_i:=4\tau(i+1)\log M$, so that $\mathrm{dist}(v_i,v_{i+1})\le L_i$. Fix a shortest path
$
v_i=x_0,x_1,\dots,x_{L_i'}=v_{i+1}
$
with $L_i'\le L_i$. Applying the time-shifted version of the event $C_{v_i}$ from the first time at which $v_i$ becomes $((M^{i})^{\alpha_3},(M^{i})^{\alpha_3})$-HS-lit with respect to $HS(v_{i})$, there exist random times $(T_j^{(v_i)})_{j=1}^{\lceil (M^i)^{\alpha_4}\rceil+1}$ such that $v_i$ is infected at time $T_j^{(v_i)}$ for every $j=1,\dots,\lceil (M^i)^{\alpha_4}\rceil$, and $T_{j+1}^{(v_i)}-T_j^{(v_i)}\ge (M^i)^{\alpha_5}$. Since $(M^i)^{\alpha_5}>L_i+2$ for all sufficiently large constant $M$, the time intervals $[T_j^{(v_i)},T_j^{(v_i)}+L_i+2]$ are pairwise disjoint. Conditionally on the times $T_j^{(v_i)}$, these trials depend on disjoint increments of the modified Harris  construction and are therefore conditionally independent. We use the first
$
N_i:=\lceil (M^i)^{\alpha_4}\rceil
$
such intervals as independent trials.

Consider one such trial, starting at time $T_j^{(v_i)}$. During the first $L_i'+1$ units of time, by repeated applications of Lemma~\ref{lemma_edge_transmission_one} along the path $x_0,\dots,x_{L_i'}$, there is probability at least
$
c_{\mathrm{edge}}^{\,L_i'+1}\ge c_{\mathrm{edge}}^{\,L_i+1}
$
that the infection is transmitted successively from $v_i$ to $v_{i+1}$, so that $v_{i+1}$ is infected by time $T_j^{(v_i)}+L_i'+1\le T_j^{(v_i)}+L_i+1$. Conditional on this, during the remaining unit of the trial interval, Lemma~\ref{lemma_hierarchicalstar_infectedtoHSlit} implies that, with probability at least
$
1-(M^{i+1})^{-\alpha_3},
$
the vertex $v_{i+1}$ becomes $((M^{i+1})^{\alpha_3},(M^{i+1})^{\alpha_3})$-HS-lit with respect to $HS(v_{i+1})$. Hence, after enlarging $M$ if necessary, the probability that a given trial succeeds is at least
$
\frac12 c_{\mathrm{edge}}^{\,L_i+1}.
$

Therefore the probability that all $N_i$ trials fail is at most
$$
\left(1-\frac12 c_{\mathrm{edge}}^{\,L_i+1}\right)^{N_i}
\le \exp\!\left(-\frac12 c_{\mathrm{edge}}^{\,L_i+1}N_i\right).
$$
Combining this with the probability of the event $C_{v_i}$, we conclude that, conditional on $v_i$ being $((M^i)^{\alpha_3},(M^i)^{\alpha_3})$-HS-lit at some time, the probability that $v_{i+1}$ becomes $((M^{i+1})^{\alpha_3},(M^{i+1})^{\alpha_3})$-HS-lit at some later time is at least
$$
1-\exp(-(M^i)^{\alpha_4})-\exp\!\left(-\frac12 c_{\mathrm{edge}}^{\,L_i+1}N_i\right).
$$
By choosing $M$ large enough, we may assume that
\begin{equation}\label{eq:main2_sum_small}
M^{-\alpha_3}+\sum_{i\ge 1}\left(\exp(-(M^i)^{\alpha_4})+\exp\!\left(-\frac12 c_{\mathrm{edge}}^{\,L_i+1}N_i\right)\right)\le \frac15.
\end{equation}
It follows that, on the event $E\cap \bigcap_{i=1}^{k-1}A_{v_i}$, the probability that $v_{k-1}$ becomes infected at some time is at least $4/5$. Combining this with \eqref{eq:main2_EA}, we conclude that there exists a constant $q_1>0$ such that
\begin{equation}\label{eq:main2_hit_vk}
\mathbb{P}\Bigl(\text{$v_{k-1}$ becomes infected  at some time}\Bigr)\ge q_1
\end{equation}
for all sufficiently large $n$.

By construction,
$
d(v_{k-1})\ge M^{k-1}
$.
Since
$
k:=\left\lfloor \frac{\log n}{\log M}\cdot \min \left\{\frac{1}{32\tau}, \frac{\tau-2}{32(\tau-1)} \right\} \right\rfloor
$,
for all sufficiently large $n$ we have
\begin{equation}\label{eq:main2_vk_degree}
d(v_{k-1})\ge n^{\min\{\frac{1}{64\tau}, \frac{\tau-2}{64(\tau-1)} \} }.
\end{equation}

\noindent\textbf{Step 2: from polynomial degree to the large-scale metastable regime.}
Suppose that at some time an infected vertex $x$ satisfies $\deg(x)\ge n^{\min\{\frac{1}{64\tau}, \frac{\tau-2}{64(\tau-1)}\}}$. By Lemma~\ref{lemma_Mhierarchicalstar}, with probability tending to $1$, the vertex $x$ is the center of a hierarchical-star with parameters $(n^{\beta_1},n^{\beta_1})$ for some constant $\beta_1>0$ depending only on $\tau$. By Lemma~\ref{lemma_hierarchicalstar_infectedtoHSlit} and Proposition~\ref{hierarchicalstar_thm2}, there exist constants $\beta_2,\beta_3>0$ such that, with probability tending to $1$, there exist random times $(T_j^{(x)})_{j=1}^{\lceil\exp(n^{\beta_2})\rceil+1}$ such that $x$ is infected at time $T_j^{(x)}$ and $T_{j+1}^{(x)}-T_j^{(x)}\ge n^{\beta_3}$ for every $j=1,\dots, \lceil\exp(n^{\beta_2})\rceil$.

On the event $F_n$, the graph diameter is at most $c_2\log n\le n^{\beta_3}$ for all  sufficiently large $n$. Fix one of the $(n^{\varepsilon}, n^{\varepsilon})$ hierarchical-star center to be $v^{(1)}$. Using the times $(T_j^{(x)})_{j=1}^{\lceil\exp(n^{\beta_2})\rceil+1}$ as disjoint trials and arguing as in the previous degree-ladder step, the infection can repeatedly attempt to travel from $x$ to $v^{(1)}$. Each attempt has probability at least $c_{\mathrm{edge}}^{c_2\log n+1}=n^{-C_3}$ for some constant $C_3>0$, independently conditionally on $(T_j^{(x)})_{j=1}^{\lceil\exp(n^{\beta_2})\rceil+1}$, while the number of available attempts is exponential in $n^{\beta_2}$. Therefore, with probability tending to $1$, $v^{(1)}$ gets infected for some time.

Once $v^{(1)}$ is infected, Lemma~\ref{lemma_hierarchicalstar_infectedtoHSlit} implies that, with probability tending to $1$, it becomes $(n^{\varepsilon_3},n^{\varepsilon_1/24})$-HS-lit with respect to its own hierarchical-star. Then, by Lemma~\ref{lemma_infectionspreadthroughhierarchicalstars}, with probability tending to $1$, all of the $\lceil n^{1-\delta'}\rceil$ vertex-disjoint $(n^{\varepsilon},n^{\varepsilon})$ hierarchical-stars become $(n^{\varepsilon_3},n^{\varepsilon_1/24})$-HS-lit with respect to their respective hierarchical-stars at some time. From that moment onward, exactly the same large-scale argument as in the proof of Theorem~\ref{thm_main1} applies. Therefore, conditional on the event that some vertex of degree at least $n^\varepsilon$ is infected at some time, we have
\begin{equation}\label{eq:main2_from_neps}
\mathbb{P}\big(T\ge \exp(n^{1-\delta})\big)\to 1.
\end{equation}
Combining \eqref{eq:main2_initial_degree}, \eqref{eq:main2_hit_vk}, \eqref{eq:main2_vk_degree},\eqref{eq:main2_from_neps}, and \eqref{eq:main2_Fn}, we obtain that there exists a constant $p>0$, depending only on $\lambda,\rho,\tau,\delta$, such that for all sufficiently large $n$,
\[
\inf_{\mu_n} \mathbb{P}^*_{\mu_n}\big(T\ge \exp(n^{1-\delta})\big)\ge p.
\]
This completes the proof.
\end{proof}

\bibliographystyle{abbrvnat}
\bibliography{refs}

\end{document}